\documentclass{amsart}

\usepackage{amsmath}
\usepackage{amssymb}
\usepackage{stmaryrd}
\usepackage{xcolor}
\usepackage{enumitem}
\usepackage[all,cmtip]{xy}
\usepackage{tikz}
\usepackage{subfig}
\usepackage{algorithm}
\usepackage{algpseudocode}
\usepackage[hidelinks]{hyperref}

\theoremstyle{plain}
\newtheorem{theorem}{Theorem}
\newtheorem{corollary}[theorem]{Corollary}
\newtheorem{lemma}[theorem]{Lemma}
\newtheorem{proposition}[theorem]{Proposition}

\theoremstyle{remark}
\newtheorem{remark}[theorem]{Remark}

\numberwithin{equation}{section}
\numberwithin{theorem}{section}
\numberwithin{figure}{section}
\newcounter{AssumptionCounter}
\newcommand{\R}{\mathbb{R}}
\newcommand{\Domain}{\Omega}
\newcommand{\Dim}{d}
\newcommand{\Normal}{\mathsf{n}}
\newcommand{\Tangent}{\mathsf{t}}
\newcommand{\dt}{}
\newcommand{\Grad}{\nabla}
\newcommand{\SymGrad}{\Grad_{\!S}}
\newcommand{\Div}{\mathrm{div}}
\newcommand{\Norm}[1]{\| #1 \|}
\newcommand{\Err}[2]{\texttt{ERR}(#1,#2)}
\newcommand{\bbU}{\mathbb{U}}
\newcommand{\bbP}{{\mathbb{P}}}
\newcommand{\bbD}{\mathbb{D}}
\newcommand{\bbH}{\mathbb{H}}
\newcommand{\bbY}{\mathbb{Y}}
\newcommand{\bbX}{\mathbb{X}}
\newcommand{\bbS}{\mathbb{S}}
\newcommand{\calE}{\mathcal{E}}
\newcommand{\calD}{\mathcal{D}}
\newcommand{\calL}{\mathcal{L}}
\newcommand{\calP}{\mathcal{P}}
\newcommand{\calS}{\mathcal{S}}
\newcommand{\calI}{\mathcal{I}}
\newcommand{\calJ}{\mathcal{J}}
\newcommand{\sfs}{\mathsf{s}}
\newcommand{\sft}{\mathsf{t}}
\newcommand{\Mesh}{\mathfrak{T}}
\newcommand{\Faces}{\mathfrak{F}}
\newcommand{\Vertices}{\mathfrak{V}}
\newcommand{\sfT}{\mathsf{T}}
\newcommand{\sfF}{\mathsf{F}}
\newcommand{\sfv}{\mathsf{v}}
\newcommand{\sfh}{\mathsf{h}}
\newcommand{\sfk}{\mathsf{k}}
\newcommand{\Poly}[1]{P_{#1}}
\newcommand{\Lagr}[2]{S_{#1}^{#2}}
\newcommand{\Jump}[1]{\llbracket #1 \rrbracket}
\newcommand{\Avg}[1]{\{\!\!\{ #1 \}  \!\! \}}
\newcommand{\Tot}{\text{tot}}
\begin{document}

% Title
\title[Inf-sup stable discretization of the Biot's equations]{Inf-sup stable discretization of the \\quasi-static Biot's equations in poroelasticity}

% Authors
\author[C.~Kreuzer]{Christian Kreuzer}
\address{TU Dortmund \\ Fakult{\"a}t f{\"u}r Mathematik \\ D-44221 Dortmund \\ Germany}
\email{christian.kreuzer@tu-dortmund.de}

\author[P.~Zanotti]{Pietro Zanotti}
\address{Universit\`{a} degli Studi di Pavia, Dipartimento di Matematica, 27100, Pavia, Italy}
\email{pietro.zanotti@unipv.it}

% Keywords

% Subject classification

% 65M15: Error bounds for initial value and initial-boundary value problems involving PDEs
%
% 65M60: Finite element, Rayleigh-Ritz and Galerkin methods for initial value and initial-boundary value problems involving PDEs
%
% 74F10: Fluid-solid interactions (including aero- and hydro-elasticity, porosity, etc.)
%
% 75S05: Flows in porous media; filtration; seepage

\begin{abstract}
We propose a new full discretization of the Biot's equations in
poroelasticity. The construction is driven by the inf-sup theory, which we
recently developed. It builds upon the four-field formulation of the
equations obtained by introducing the total pressure and the total
fluid content. We discretize in space with Lagrange finite elements
and in time with backward Euler. We establish inf-sup stability and
quasi-optimality of the proposed discretization, with robust constants
with respect to all material parameters. We further construct an interpolant
showing how the error decays for smooth
solutions. % Remarkably, in some cases, the regularity that we require
% for this purpose matches the one provided by the equations with smooth
% data. 
\end{abstract} 

\maketitle

\section{Introduction}
\label{S:introduction}
This paper is the second one in a series initiated by \cite{Kreuzer.Zanotti:23+}, regarding the analysis and the discretization of the quasi-static Biot's equations in poroelasticity. (See \eqref{E:BiotProblem} below for the statement of the problem). The series centers around the use of the inf-sup theory for the stability and the error analysis, with the aim of highlighting the possible advantages stemming from the proposed approach, which appears to be new in this context.  

The inf-sup theory is a framework for the analysis of general linear variational problems. The main result therein is the so-called Banach-Ne\u{c}as theorem (see, e.g., \cite[Section~25.3]{Ern.Guermond:21b}), that characterizes the well-posedness of such problems. Successful applications of the Banach-Ne\u{c}as theorem additionally provide a two-sided stability estimate, entailing that the space of all possible solutions is
isomorphic to that of all possible data, cf. Theorem~\ref{T:well-posedness-weak-formulation} below.  Such an estimate is of
special interest, as it is the starting point for the derivation of sharp a posteriori error estimates, since it establishes the equivalence of error and residual, cf. \cite[Section~4.1.4]{Verfuerth:13}. Moreover, if the same approach applies also to the discretization at hand, then the resulting stability estimate is the starting point for the derivation of sharp a priori error estimates, see \cite{Babuska:70} for conforming discretizations and \cite{Berger.Scott.Strang:72} for nonconforming ones. The inf-sup theory is useful also in the design of robust preconditioners \cite{Hiptmair:06,Mardal.Winther:11} and in the convergence analysis of some adaptive procedures \cite{Feischl:22}.

The inf-sup theory is well-established for the analysis and the
discretization of stationary linear equations. We refer to \cite{Ern.Guermond:21b}
for an overview with several examples. The situation is substantially
different for evolution equations, that are usually analyzed by other
techniques. While the application of the inf-sup theory to the heat equation has been recently considered by various authors (see, e.g.,
\cite[Chapter~71]{Ern.Guermond:21c} and the references therein), we
are aware of only a few results for other prototypical problems, like
the wave and the Schr\"{o}dinger equation \cite{Steinbach.Zank:22}.       

The quasi-static Biot's equations in poroelasticity are not an
exception in this respect. In fact, on the one hand, some authors have
used the inf-sup theory in the analysis and the discretization of the
stationary problem resulting from the semi-discretization in time, see
e.g. \cite{Hong.Kraus:18,Khan.Zanotti:22,Lee.Mardal.Winther:17}. On
the other hand, in the quasi-static case, we are not aware of any
result obtained by this approach. For the a posteriori analysis, Li
and Zikatanov \cite{Li.Zikatanov:22} used, in a sense, an equivalent
argument. For the a priori analysis, all papers we are aware of build
upon a different argument, namely an energy estimate that seems to go
back to a seminal contribution of \v{Z}en\'{\i}\v{s}ek
\cite{Zenisek:84}. A far by exhaustive list includes
\cite{Berger.Bordas.Kay.Tavener:15,Burger.Kumar.Mora.RuizBaier.Verma:21,Chen.Luo.Feng:13,Kanschat.Riviere:18,Lee:16,Phillips.Wheeler:08,Rodrigo.Gaspar.Hu.Zikatanov:16,Yi.13}. 

To our best knowledge, our series is the first contribution making a
systematic use of the inf-sup theory in the design and in the error
analysis of a discretization of the quasi-static Biot's equations. Our
first paper \cite{Kreuzer.Zanotti:23+} is devoted to the analysis of
the equations. It establishes the well-posedness as well as a two-sided
stability estimate, which is robust with respect to all material
parameters.
The latter contributions and the functional setting
distinguish our results from previous ones by \v{Z}en\'{\i}\v{s}ek
\cite{Zenisek:84} and Showalter \cite{Showalter:00}. In particular, we
consider an equivalent four-field formulation of the equations, that
is obtained from the original one by introducing the so-called total
pressure and total fluid content. In \cite{Kreuzer.Zanotti:23+} we prove also that, in certain circumstances, additional regularity in space of the data imply corresponding additional regularity in space of the solution. Establishing a similar result in time is more challenging, due to possible
singularities at the initial time;~compare e.g. with
\cite{Murad.Thomee.Loula:96,Showalter:00} and the discussion in
Remark~\eqref{R:regularity-shift-time} below.  Further 
previous contributions to the regularity theory
are in \cite{Botti.Botti.DiPietro:21,Yi:17}. %\todo{C: Shall we say here more or... less? P: I believe it's fine}

In this paper we propose a backward Euler discretization in time and
an abstract discretization in space of the Biot's equations. We establish the well-posedness
and a two-sided stability estimate via the inf-sup theory, by mimicking
the argument in \cite{Kreuzer.Zanotti:23+}. Then we consider a simple
realization of the abstract discretization in space, making use of
Lagrange finite elements for all variables. We prove a
quasi-optimal a priori error estimate, meaning that the error is equivalent to (i.e. bounded from above and below by) the best error. To our best knowledge, it is the first time that such a result is
established for a discretization of the Biot's equations.  

The error notion we consider is motivated by the inf-sup theory and it
is relatively involved, as all variables are coupled in a nontrivial
way. Therefore, we further elaborate on our error estimate, by showing
that the best error can be bounded by a sum of decoupled best errors in standard norms, that are much easier to investigate. All constants in our
results are robust with respect to all material parameters and do not
require any additional regularity beyond the one guaranteed by the
well-posedness of the equations. Finally, we establish first-order
convergence, with respect to the space and time meshsize, for sufficiently smooth solutions. % Remarkably, in some
% cases, the regularity that is needed for this purpose matches the one
% provided by the regularity theorem in the first part
% \cite{Kreuzer.Zanotti:23+}. Perhaps, this is the most substantial
% difference between our work and previous ones.   

\subsection*{Organization} In section~\ref{S:Biot-equations} we state
the equations and recall the main results from
\cite{Kreuzer.Zanotti:23+}. Section~\ref{S:abstract-discretisation}
establishes the stability of an abstract
discretization. Section~\ref{S:concrete-discretization} is devoted to
the a priori error analysis for an exemplary concrete
discretization, which is also tested numerically in section~\ref{S:numerics}.

\subsection*{Notation} 
Throughout the paper, we denote by
$\Norm{\cdot}_{\bbX}$ the norm of a Hilbert space $\bbX$. The dual space ${\bbX}^*$ acts on ${\bbX}$ through the pairing $\left\langle
\cdot, \cdot \right\rangle_{\bbX}$. We denote by $L^2(\bbX)$, $H^1(\bbX)$ and $C^0(\bbX)$, respectively, the
Bochner spaces of all $L^2$, $H^1$ and $C^0$ functions mapping the time
interval $[0, T]$ into $\bbX$. For a measurable set $\Domain \subseteq \R^\Dim$, we adopt the simplified
notation $(\cdot, \cdot)_\Domain$ and $\Norm{\cdot}_\Domain$ for the
scalar product and the norm in $L^2(\Domain)$. We write $a \lesssim b$
and $a \eqsim b$ when there are constants $0 < \underline{c} \leq
\overline{c}$, possibly different at different occurrences, such that
$a \leq \overline{c}\,b$ and $\underline{c} \,b \leq a \leq
\overline{c}\,b$, respectively. The hidden
constants are independent of the material parameters involved in the
equations. The dependence on other relevant quantities  
is addressed case by case, see e.g. Remark~\ref{R:hidden-constants}.  

\section{Inf-sup theory for the Biot's equations}
\label{S:Biot-equations}

This section introduces the Biot's equations and summarizes some
results from \cite{Kreuzer.Zanotti:23+}, that are
useful for the construction and the analysis of the discretization in
the next sections.  

\subsection{Initial-boundary value problem}
\label{SS:initial-boundary-value-problem}
Let $\Domain \subseteq \R^\Dim$, $1 \leq \Dim \leq 3$, be a bounded domain with polyhedral and Lipschitz continuous boundary. The
flow of a Newtonian fluid 
inside a linear elastic porous medium located in $\Domain$, in the
time interval $(0, T)$ with $T > 0$, is modeled by the quasi-static
Biot's equations as follows  
\begin{subequations}
\label{E:BiotProblem}
\begin{equation}
\label{E:BiotProblem-equations}
\begin{alignedat}{3}
 -\Div( 2 \mu \SymGrad u + (\lambda \Div u  - \alpha p)\mathrm{I}) &= f_u \qquad &\text{in $\Domain \times (0,T)$  }\\
\partial_t (\alpha \Div u + \sigma p) - \Div(\kappa \Grad p) &= f_p \qquad  &\text{in $\Domain \times (0,T)$}.
\end{alignedat}
\end{equation}
The first equation states the momentum balance for the elastic porous
medium, whereas the second one is the mass balance for the fluid. 

The unknown functions in the equations are the displacement $u:
\Domain \to \R^\Dim$ of the elastic porous medium and the pressure $p:
\Domain \to \R$ of the fluid. The differential operator $\SymGrad$ is
the symmetric part of the gradient and $\mathrm{I}$ is $\Dim \times
\Dim$ identity tensor. The material parameters, denoted by Greek
letters, are the Lam\'{e} constants $\mu, \lambda > 0$, the 
Biot-Willis constant $\alpha~>~0$, the constrained specific storage
coefficient $\sigma \geq 0$ and the hydraulic conductivity $\kappa >
0$. Consistently with \cite{Kreuzer.Zanotti:23+}, we assume that all
parameters are constant in $\overline \Domain \times [0,T]$ for
simplicity. We refer to \cite[Remark~2.1]{Kreuzer.Zanotti:23+} for a
discussion on possible extensions. 

We are interested in the initial-boundary value problem obtained by prescribing also the initial condition 
\begin{equation}
\label{E:BiotProblem-initialcondition}
(\alpha \Div u + \sigma p)_{|t=0} = \ell_0 \quad \text{in} \quad \Domain
\end{equation}
as well as the boundary conditions 
\begin{equation}
\label{E:BiotProblem-boundary-conditions}
\begin{alignedat}{2}
u &= 0
&\quad &\text{on} \quad \Gamma_{u,E} \times (0,T)\\
(2 \mu \SymGrad u + (\lambda \Div u - \alpha p) I)\Normal &= g_{u}&
\quad& \text{on} \quad \Gamma_{u,N} \times (0,T) \\
p &= 0
&\quad &\text{on} \quad \Gamma_{p,E} \times (0,T)\\
\kappa \Grad p \cdot \Normal &= g_{p}
&\quad &\text{on} \quad \Gamma_{p,N} \times (0,T)
\end{alignedat}
\end{equation}
where $\Gamma_{u,E} \cup \Gamma_{u,N} = \partial \Domain =
\Gamma_{p,E} \cup \Gamma_{p,N}$ and $\Gamma_{u,E} \cap \Gamma_{u,N} =
\emptyset = \Gamma_{p,E} \cap \Gamma_{p,N}$. The letter $\Normal$
denotes the outward unit normal vector on $\partial \Domain$. The
subscripts `$E$' and `$N$' are intended to assist the reader in
distinguishing the portions of the boundary with essential and natural
conditions. 
\end{subequations} 

We point out that different statements of the initial and of the
boundary conditions are sometimes met in the literature. We refer to
\cite[Remark~2.2-2.3]{Kreuzer.Zanotti:23+} for a more extensive
discussion on this point.

\subsection{Weak Formulation and well-posedness}
\label{SS:weak-formulation}
For convenience, we introduce a compact notation for the differential
operators in the Biot's equations \eqref{E:BiotProblem}, namely  
\begin{equation}
\label{E:abstract-operators-concrete}
\calE := - \Div (2\mu  \SymGrad) 
\qquad \qquad
\calD := \Div
\qquad \qquad
\calL := -\Div(\kappa \Grad ).
\end{equation}
Notice that $\calE$ and $\calL$ act on $u$ and $p$, respectively, in
\eqref{E:BiotProblem-equations}. Comparing also with the boundary
conditions, it looks reasonable that the regularity of $u$ in space in
a weak formulation can be described via 
\begin{equation}
\label{E:abstract-spaces-displacement}
\bbU := \begin{cases}
	H^1(\Domain)^\Dim/\mathrm{RM} & \text{if $\Gamma_{u,N} =
          \partial \Domain$}
        \\
	H^1_{\Gamma_{u,E}}(\Domain)^\Dim & \text{otherwise} 
\end{cases}
\end{equation}
with $\mathrm{RM}$ denoting the rigid body motions. Analogously, for
the regularity of $p$ in space, we consider 
\begin{equation}
\label{E:abstract-spaces-pressure}
\bbP := \begin{cases}
H^1(\Domain)\cap L^2_0(\Domain) & \text{if $\Gamma_{p,N} = \partial
	\Domain$} \\
H^1_{\Gamma_{p,E}}(\Domain) \cap L^2_0(\Domain) & \text{if
  $\Gamma_{p,N} \neq \partial \Domain, \Gamma_{u,E} = \partial
  \Domain, \sigma = 0$} 
\\
H^1_{\Gamma_{p,E}}(\Domain) & \text{otherwise}
\end{cases} 
\end{equation}
where $L^2_0(\Domain) = \{ q \in L^2(\Domain) \mid \int_{\Domain} q = 0
\}$. We refer to \cite[Remark~2.5]{Kreuzer.Zanotti:23+} for a
motivation of the nonstandard definition of $\bbP$ in the second
case. 

\begin{remark}[Notation]
\label{R:abstract-notation}
We are aware of the fact, that the abstract notation introduced here
(and the subsequent one for all related spaces and operators) makes
the reading possibly harder. Still, it has the advantage that
all combinations of the boundary conditions (as well as the
critical case $\sigma = 0$) can be treated at the same time. In our
view, this is quite important, because our approach to the analysis
and the discretization of the Biot's equations is mainly the same in
all cases, but each case may require subtle minor modifications. 
\end{remark}

The action of the divergence on $u$ in \eqref{E:BiotProblem-equations}
indicates that also the space 
\begin{equation}
\label{E:abstract-spaces-pressure-total}
\bbD := \calD(\bbU) = \begin{cases}
L^2_0(\Domain) &\text{if  $\Gamma_{u,E} = \partial \Domain$}\\
L^2(\Domain) &\text{otherwise}
\end{cases}
\end{equation}
plays a relevant role in the Biot's equations. Similarly, since $p$ is
involved in \eqref{E:BiotProblem-equations} also without the action
of any differential operator in space, we repeatedly make use of 
\begin{equation}
\label{E:abstract-spaces-fluid-content}
\overline{\bbP} = \begin{cases}
L^2_0(\Domain) & \text{if
	$\Gamma_{p,N} = \partial \Domain$ or $\Gamma_{u,E} = \partial
        \Domain, \sigma = 0$}
      \\ 
L^2(\Domain) & \text{otherwise,}
\end{cases}
\end{equation}
where the closure is taken with respect to the $L^2(\Domain)$-norm. An
important point for our analysis is that both $\bbD$ and
$\overline{\bbP}$ are subspaces of $L^2(\Domain)$, but their mutual
relation depends on the boundary conditions and on the parameter
$\sigma$. Therefore, it is useful introducing   
\begin{equation*}
\label{E:projections}
\calP_{\bbD}: L^2(\Omega) \to \bbD
\qquad \text{and} \qquad
\calP_{\overline{\bbP}}: L^2(\Domain) \to \overline{\bbP},
\end{equation*}
the $L^2(\Domain)$-orthogonal projections onto $\bbD$ and $\overline{\bbP}$, respectively.

\begin{remark}[Functional setting]
\label{R:functional-setting}
The diagram in Figure~\ref{F:abstract-spaces-diagram} summarizes the
relation among the spaces and the operators introduced up to this
point. In addition, we denote by $i: \bbP\to \overline{\bbP}$  the
inclusion operator and $\calD^*$ and $i^*$ are the adjoint operators of
$\calD$ and $i$, respectively. The spaces $\bbD$ and $\overline{\bbP}$
are identified with their duals via the $L^2(\Domain)$-scalar
product. Thus, the square on the right side of the diagram involves the Hilbert triplet 
\begin{equation}
\label{E:Hilbert-triplet}
\bbP 
\hookrightarrow 
\overline{\bbP} 
\equiv 
\overline{\bbP}^*
\hookrightarrow
\bbP^*.
\end{equation} 
\end{remark}

\begin{figure}[ht]
	\[
	\xymatrixcolsep{4pc}
	\xymatrixrowsep{4pc}
	\xymatrix{
		\bbU^* 
		\ar@{<-}[r]^{\calE}
		\ar@{<-}[d]^{\calD^*}
		&
		\bbU
		\ar[d]_{\calD}
		&
		L^2(\Domain)
		\ar[dl]_{\calP_\bbD}
		\ar[dr]^{\calP_{\overline{\bbP}}}
		&
		\bbP
		\ar[r]^{\calL}
		\ar[d]^{i}
		&
		\bbP^*
		\ar@{<-}[d]_{i^*}\\
		\bbD^*
		\ar@3{-}[r]
		&
		\bbD
		&&
		\overline{\bbP}
		\ar@3{-}[r]&
		\overline{\bbP}^*
		}
	\]
	\caption{\label{F:abstract-spaces-diagram} Spaces and operators describing the regularity in space for the weak formulation \eqref{E:BiotProblem-weak-formulation} of the Biot's equations. The triple lines on the bottom indicate identification via the $L^2(\Domain)$-scalar product.}
\end{figure}

With this preparation, we are in position to recall the weak
formulation of the initial-boundary value problem
\eqref{E:BiotProblem} introduced in
\cite[section~2.3]{Kreuzer.Zanotti:23+}, namely 
\begin{equation}
\label{E:BiotProblem-weak-formulation}
\begin{alignedat}{2}
\calE u + \calD^* p_\Tot &= \ell_u \qquad && \text{in $L^2(\bbU^*)$}\\
\lambda \calD u - p_\Tot - \alpha \calP_\bbD p &= 0 \qquad && \text{in $L^2(\bbD)$}\\
\alpha \calP_{\overline \bbP} \calD u + \sigma p - m &= 0 \qquad && \text{in $L^2(\overline{\bbP})$}\\
\partial_t m + \calL p &= \ell_p \qquad  && \text{in $L^2(\bbP^*)$}\\
m(0) &= \ell_0 \qquad && \text{in $\bbP^*$}.
\end{alignedat}
\end{equation} 
The loads $\ell_u \in L^2(\bbU^*)$ and $\ell_p \in L^2(\bbP^*)$ result from the data $f_u$ and $f_p$ in the equations \eqref{E:BiotProblem-equations} as well as $g_u$ and $g_p$ in the boundary conditions \eqref{E:BiotProblem-boundary-conditions}.

\begin{remark}[Auxiliary variables]
\label{R:auxiliary-variables}
Compared to \eqref{E:BiotProblem-equations}, the weak formulation \eqref{E:BiotProblem-weak-formulation} involves two additional unknown variables, namely the total pressure $p_\Tot$ and the total fluid content $m$ defined by the second and third equation, respectively. Introducing these variables is not strictly necessary for our analysis, but it substantially simplifies the definition of the space $\bbY_1$ in \eqref{E:trial-space} below. The use of the total pressure was observed in \cite{Lee.Mardal.Winther:17} to help also the design of robust linear solvers. 
\end{remark}

The main result in \cite{Kreuzer.Zanotti:23+} states that
\eqref{E:BiotProblem-weak-formulation} is uniquely solvable within the closure $\overline{\bbY}_1$ of the space
\begin{equation}
\label{E:trial-space}
\bbY_1 := L^2(\bbU) \times L^2(\bbD) \times L^2(\bbP) \times \left( L^2(\overline{\bbP}) \cap H^1(\bbP^*) \right)
\end{equation}
with respect to the norm
\begin{equation}
\label{E:trial-norm}
\begin{split}
&\| (\widetilde u, \widetilde p_\Tot, \widetilde p, \widetilde m) \|_1^2 := \\
&\quad \phantom{+}\int_{0}^{T} \left( \Norm{\widetilde u}^2_\bbU 
+ 
\dfrac{1}{\mu}\Norm{\widetilde p_\Tot}_{\Domain}^2 
+ 
\Norm{\partial_t \widetilde m + \calL \widetilde p}^2_{\bbP^*} \right)\dt 
+ 
\Norm{\widetilde m(0)}^2_{\bbP^*}\\ 
&\quad +\int_{0}^T \left( \dfrac{1 }{\mu + \lambda} \Norm{\lambda \calD \widetilde u - \widetilde p_\Tot - \alpha \calP_\bbD \widetilde p}^2_{\Domain} 
+ 
\gamma\Norm{\alpha \calP_{\overline \bbP} \calD \widetilde u + \sigma \widetilde p - \widetilde m}^2_{\Domain} \right)\dt.
\end{split}
\end{equation}
Here $\bbU$ and $\bbP$ are equipped with the energy norm 
\begin{equation}
\label{E:norms}
\Norm{\cdot}_\bbU^2 = \left\langle \calE \cdot, \cdot \right\rangle_\bbU
\qquad
\text{and}
\qquad
\Norm{\cdot}_\bbP^2 = \left\langle \calL \cdot, \cdot \right\rangle_\bbP
\end{equation}
and the parameter $\gamma$ is given by
\begin{align}
\label{E:gamma}
\gamma :=
\begin{cases}
\min\left\lbrace  \dfrac{\mu + \lambda}{\alpha^2}, \dfrac{1}{\sigma}\right\rbrace   & \text{if $\sigma > 0$ and $\overline{\bbP} \subseteq \bbD$}
\\[8pt]
\dfrac{\mu + \lambda}{\alpha^2} + \dfrac{1}{\sigma} & \text{if $\sigma > 0$ and $\overline{\bbP} \nsubseteq \bbD$}\\[8pt]
\dfrac{\mu + \lambda}{\alpha^2} &\text{if $\sigma = 0$}.
\end{cases}
\end{align} 
Taking the closure is indeed necessary, because $\bbY_1$ is not closed with respect to $\Norm{\cdot}_1$, cf. \cite[Proposition~4.4]{Kreuzer.Zanotti:23+}.

\begin{theorem}[Well-posedness of the weak formulation]
\label{T:well-posedness-weak-formulation}
For all possible data $(\ell_u, \ell_p, \ell_0) \in L^2(\bbU^*) \times L^2(\bbP^*) \times \bbP^*$, the weak formulation \eqref{E:BiotProblem-weak-formulation} has a unique solution $y_1 = (u,p_\Tot, p, m) \in \overline{\bbY}_1$, which satisfies the two-sided stability bound
\begin{equation*}
\label{E:well-posedness}
\begin{split}
\|y_1\|_{1}^2 \eqsim \int_{0}^T \left( \Norm{\ell_u}^2_{\bbU^*} + \Norm{\ell_p}^2_{\bbP^*} \right)\dt + \Norm{\ell_0}^2_{\bbP^*}.
\end{split}
\end{equation*}
Moreover, we have $m \in C^0(\bbP^*)$ as well as the norm equivalence
\begin{equation*}
\label{E:well-posedness-equivalent-norm}
\|y_1\|_{1}^2 \eqsim \|y_1\|_{1}^2 + \Norm{m}^2_{L^\infty(\bbP^*)} + \int_0^T \left( \lambda\|\calD u\|^2_{\Domain} + \gamma^{-1}\| p\|_{\Domain}^2 \right) \dt.
\end{equation*}
All hidden constants depend only on $\Domain$ and $T$.
\end{theorem}

\begin{proof}
Combine \cite[Theorem~3.5]{Kreuzer.Zanotti:23+} with \cite[Proposition~4.1]{Kreuzer.Zanotti:23+}.
\end{proof}

Although we omit the proof of
Theorem~\ref{T:well-posedness-weak-formulation}, it is worth roughly
summarizing how it is obtained. Indeed, we shall make use of a similar
argument in order to verify the well-posedness of the discretization
introduced in the next section,
cf. Theorem~\ref{T:well-posedness-discretization} below.  

The weak formulation \eqref{E:BiotProblem-weak-formulation} can be
viewed as a special instance of the following linear variational
problem: given $\ell \in \bbY_2^*$, find $y_1 \in \overline \bbY_1$
such that 
\begin{equation}
\label{E:linear-variational-problem}
b(y_1, y_2) = \left\langle \ell, y_2\right\rangle_{\bbY_2} \qquad \forall y_2 \in \bbY_2.
\end{equation}
The test space is obtained by collecting all possible test functions
for \eqref{E:BiotProblem-weak-formulation}, namely 
\begin{equation*}
\label{E:test-space}
\bbY_2 := L^2(\bbU) \times L^2(\bbD) \times L^2(\overline{\bbP}) \times L^2(\bbP) \times \bbP
\end{equation*}
and it is equipped with the norm 
\begin{equation}
\label{E:test-norm}
\begin{split}
\Norm{(v, q_\Tot, q, n, n_0)}_2^2 &:= \int_0^T \Big ( \Norm{v}^2_\bbU +  \Norm{n}^2_{\bbP} \Big)\dt + \Norm{n_0}^2_\bbP\\
& \quad+ \int_{0}^T \left( (\mu + \lambda) \Norm{q_\Tot}_{\Domain}^2 + \gamma^{-1}
\Norm{q}^2_{\Domain} \right)\dt .
\end{split}
\end{equation}
The bilinear form $b: \overline \bbY_1 \times \bbY_2 \to \R$ is
defined according to the left-hand side of
\eqref{E:BiotProblem-weak-formulation} by
\begin{equation}
\label{E:BiotProblem-abstract-form}
\begin{split}
&b(\widetilde y_1, y_2)
:= \int_0^T \Big( \left\langle \calE \widetilde u +
\calD^* \widetilde p_\Tot, v \right\rangle_\bbU + \left\langle
\partial_t \widetilde m + \calL \widetilde p, n \right\rangle_\bbP
\Big)\dt + \left\langle \widetilde m(0), n_0 \right\rangle_\bbP
\\
&\quad + \int_0^T \Big ( \left( \lambda \calD \widetilde u - \widetilde
p_\Tot - \alpha \calP_\bbD\widetilde p, q_\Tot\right)_{\Domain} +
\left( \alpha \calP_{\overline \bbP}\calD \widetilde u + \sigma \widetilde p - \widetilde m, q
\right)_{\Domain}   \Big )\dt
\end{split} 
\end{equation}
for $\widetilde y_1 = (\widetilde u, \widetilde p_\Tot, \widetilde p, \widetilde m) \in \overline \bbY_{1}$ and $(v, q_\Tot, q, n, n_0) \in \bbY_2$.  

The so-called Banach-Ne\u{c}as theorem characterizes the
well-posedness of the linear variational problem in terms of
boundedness, inf-sup stability and nondegeneracy of the form $b$, see
e.g. \cite[theorem~25.9]{Ern.Guermond:21b}. These properties are
verified in \cite[section~3]{Kreuzer.Zanotti:23+}. Their combination
implies the well-posedness of \eqref{E:BiotProblem-weak-formulation}
as well as the two-sided stability bound in
Theorem~\ref{T:well-posedness-weak-formulation}.

\begin{remark}[Trial functions]
\label{R:trial-functions}
As in \eqref{E:trial-norm} and \eqref{E:BiotProblem-abstract-form}, we
use hereafter the superscript `$\sim$' to distinguish a general trial
function in $\overline{\bbY}_1$ from the solution of the weak
formulation~\eqref{E:BiotProblem-weak-formulation}. 
\end{remark}

\begin{remark}[Functions and functionals]
\label{R:functions-functionals}
Owing to Remark~\ref{R:functional-setting}, we often identify the
functions in $\bbD$ and $\overline{\bbP}$ with their Riesz
representatives in $\bbD^*$ and $\overline{\bbP}^*$ and vice
versa. For instance, the term $\calD^* p_\Tot$ from the first equation
in \eqref{E:BiotProblem-weak-formulation} and even the space
$L^2(\overline{\bbP}) \cap H^1(\bbP^*)$ proposed for the total fluid
content must be interpreted in this way. As usual, we omit to
explicitly indicate the Riesz isometries, accepting some ambiguity, so
as to alleviate the notation. We apply the same principles to the
discretization in the next sections. 
\end{remark}

\section{Abstract inf-sup stable discretization}
\label{S:abstract-discretisation}
In this section we design a discretization of the weak formulation \eqref{E:BiotProblem-weak-formulation} of the initial-boundary value problem for the Biot's equations \eqref{E:BiotProblem}. The space discretization is challenging, due to the nontrivial coupling of the variables and the various differential operators acting on them. Therefore, we initially work with a general discretization, so as to allow for the maximal flexibility. We discuss a concrete realization in section~\ref{S:concrete-discretization} below. The time discretization seems less problematic, because \eqref{E:BiotProblem-weak-formulation} involves only one time derivative. Hence, we directly make a concrete choice, namely the backward Euler scheme.

The space discretization in this section builds upon a number of assumptions that must be verified case by case. The set of our assumptions is identified by special tags with the dedicated enumeration (H1), (H2), etc. With a small abuse, we actually include among the assumptions also the definition of some relevant constants. In those cases, the size (and not just the existence) of the constants is the property to be verified in each concrete example.

The main result in this section is the well-posedness established in Theorem~\ref{T:well-posedness-discretization}. We do not attempt at analyzing the error at this level of generality. Indeed, we believe that the estimates we would obtain either require too many assumptions or are too much abstract to be of interest. Thus, we prefer to discuss the error analysis for the concrete example in section~\ref{S:concrete-discretization}. 

\begin{remark}[Notation for the discretization]
\label{R:notation-discretization}
In general, we mark all spaces and operators related to the
discretization in space by the subscript `$\sfs$' and those related to
the discretization in time by the subscript `$\sft$'. The combination
`$\sfs\sft$' of the two subscripts identifies the full discretization
in space and time. To alleviate the notation, we use capitol letters (in place of subscripts) to distinguish the functions involved in the discretization from those related to the original Biot's equations.
\end{remark}

\subsection{Abstract discretisation in space}
\label{SS:abstract-discretization-space}
The general concept of our discretization in space consists in replacing all spaces and operators in Figure~\ref{F:abstract-spaces-diagram} by finite-dimensional counterparts, while preserving the structure of the diagram therein, cf. Figure~\ref{F:discrete-spaces-diagram}. In order to discretize the displacement and the pressure, we consider two finite-dimensional linear spaces
\begin{equation}
\label{E:abstract-spaces-discrete-displacement-pressure}
 \bbU_\sfs \qquad\text{and}\qquad \bbP_\sfs 
\end{equation}
i.e. discrete counterparts of the spaces $\bbU$ in \eqref{E:abstract-spaces-displacement} and $\bbP$ in \eqref{E:abstract-spaces-pressure}. We replace the operators $\calE$ and $\calL$ in \eqref{E:abstract-operators-concrete} by positive definite and self-adjoint linear operators 
\begin{equation}
\label{E:abstract-operators-discrete-elliptic}
\calE_\sfs : \bbU_\sfs  \to \bbU_\sfs ^* 
\qquad\text{and}\qquad 
\calL_\sfs : \bbP_\sfs  \to \bbP_\sfs ^*.
\end{equation}
In analogy with \eqref{E:norms}, we equip the two spaces with the energy norms
\begin{equation}
\label{E:norms-discrete}
\|\cdot\|_{\bbU_\sfs }^2 := \langle \calE_\sfs  \cdot, \cdot \rangle_{\bbU_\sfs }
\qquad\text{and}\qquad 
\|\cdot\|_{\bbP_\sfs }^2 := \langle \calL_\sfs  \cdot, \cdot \rangle_{\bbP_\sfs }.
\end{equation}
Then, the dual norms on $\bbU_\sfs ^*$ and $\bbP_\sfs ^*$ are given by 
\begin{equation}
\label{E:norms-dual-discrete}
\begin{aligned}
\Norm{\cdot}^2_{\bbU^*_\sfs } 
&:= \sup_{V \in \bbU_\sfs } \dfrac{\left\langle \cdot, V\right\rangle_{\bbU_\sfs } }{\Norm{V}_{\bbU_\sfs }} \;
= \langle \cdot, \calE_\sfs ^{-1} \cdot \rangle_{\bbU_\sfs } \\
\Norm{\cdot}^2_{\bbP^*_\sfs } 
&:= \sup_{N \in \bbP_\sfs } \dfrac{\langle \cdot, N\rangle_{\bbP_\sfs } }{\Norm{N}_{\bbP_\sfs }} \;
= \langle \cdot, \calL_\sfs ^{-1} \cdot \rangle_{\bbP_\sfs }.
\end{aligned}  
\end{equation}

Notice that $\bbU_\sfs $ and $\bbP_\sfs $ are not required to be conforming, i.e. subspaces of $\bbU$ and $\bbP$, respectively. Therefore, in order to define an error notion on the sums $\bbU + \bbU_\sfs $ and $\bbP + \bbP_\sfs $, we assume the following.
\begin{equation}
\label{A:extended-norms}
\stepcounter{AssumptionCounter}
\tag{H\arabic{AssumptionCounter}}
\begin{minipage}{0.82\hsize}
The norms $\Norm{\cdot}_\bbU$ and $\Norm{\cdot}_\bbP$ in \eqref{E:norms} can be extended to $\bbU + \bbU_\sfs $ and $\bbP + \bbP_\sfs $ with $ \|\cdot \|_{\bbU}\eqsim
\|\cdot\|_{\bbU_\sfs }$ in $\bbU_\sfs $ and $ \|\cdot \|_{\bbP}\eqsim
\|\cdot\|_{\bbP_\sfs }$ in $\bbP_\sfs $.
\end{minipage}
\end{equation} 

In order to discretize the space $\bbD$ in \eqref{E:abstract-spaces-pressure-total}, we consider a linear operator 
\begin{equation}
\label{E:abstract-operators-discrete-divergence}
\calD_\sfs : \bbU_\sfs  \to L^2(\Domain)
\end{equation}
i.e. a discrete counterpart of the divergence $\calD$ in \eqref{E:abstract-operators-concrete}. Then, we set
\begin{equation}
\label{E:abstract-spaces-discrete-total-pressure}
\bbD_\sfs  := \calD_\sfs (\bbU_\sfs ).
\end{equation}
The proof of Theorem~\ref{T:well-posedness-weak-formulation} given in
\cite{Kreuzer.Zanotti:23+} exploits the norm equivalence $\mu
\Norm{\calD^* \cdot}_{\bbU^*}^2 \eqsim \Norm{\cdot}_{\Domain}^2$ in
$\bbD$, that is nothing else than boundedness and surjectivity of
$\calD$, cf. \cite[Lemma~C40]{Ern.Guermond:21b}. (Note that here
$\bbD$ is identified with its dual via the $L^2(\Domain)$-scalar
product, cf. Remark~\ref{R:functional-setting}.) In order to reproduce
this property at the discrete level, we assume the following. 
\begin{equation}
\label{A:Stokes-inf-sup}
\stepcounter{AssumptionCounter}
\tag{H\arabic{AssumptionCounter}}
\begin{minipage}{0.75\hsize}
There are constants $c = c(\bbU_\sfs )$ and $C = C(\bbU_\sfs )$ with
$0 < c \leq C$ and such that $c \Norm{\cdot}_{\Domain}^2 \leq \mu
\Norm{\calD_\sfs ^* \cdot}^2_{\bbU_\sfs ^*} \leq C
\Norm{\cdot}_{\Domain}^2$ in $\bbD_\sfs$. 
\end{minipage}
\end{equation}
Actually, this prescribes that $\bbU_\sfs /\bbD_\sfs$ is a stable pair
for the discretization of the Stokes equations. Note that also
$\bbD_\sfs$ is identified here with its dual space.   

The proof of Theorem~\ref{T:well-posedness-weak-formulation} given in
\cite{Kreuzer.Zanotti:23+} exploits also the density of $\bbP$ in
$\overline{\bbP}$, that gives rise to the Hilbert triplet in
\eqref{E:Hilbert-triplet}. Since $\bbP_\sfs $ is finite-dimensional,
we are led to discretize $\overline{\bbP}$ by $\bbP_\sfs $ itself,
giving rise to the triplet
\begin{equation}
\label{E:Hilbert-triplet-discrete}
\bbP_\sfs  = \overline \bbP_\sfs  \equiv \overline\bbP_\sfs ^*
= \bbP_\sfs ^*.
\end{equation}
Also in this case, the identification of $\overline{\bbP}_\sfs $ with
its dual space is made via the $L^2(\Domain)$-scalar product. Hence,
the pairing $\left\langle \cdot, \cdot\right\rangle_{\bbP_\sfs }$ coincides
with $( \cdot, \cdot )_{\Domain} $ upon identifying the functionals in
$\bbP_\sfs ^*$ with their Riesz representative. 

\begin{remark}[Discretization of the Hilbert triplet]
\label{L:discretization-Hilbert-triplet}
As $\bbP_\sfs$ coincides with its closure, the above Hilbert triplet
is trivial from the algebraic viewpoint. Still, the spaces in it play
different roles and are equipped with different norms, depending on
their position, in analogy with \eqref{E:Hilbert-triplet}. To
alleviate the notation, we omit hereafter the symbol of the
closure.%\todo{Do we agree on omitting the symbol of the closure?} 
\end{remark}

We denote by $\calP_{\bbD_\sfs }$ and $\calP_{\bbP_\sfs }$ the
$L^2(\Domain)$-orthogonal projections onto $\bbD_\sfs$ and $\bbP_\sfs
$, respectively. In particular, the adjoint $\calP_{\bbP_\sfs }^*$ of
the second projection maps functionals on $\bbP_\sfs $ into
functionals on $\bbP$. This observation is important for the
definition of the error notion in section~\ref{SS:well-posedness}
below, because the trial norm $\Norm{\cdot}_1$ in \eqref{E:trial-norm}
involves the dual norm $\Norm{\cdot}_{\bbP^*}$. Thus we assume the
following, in order to keep the norm of $\calP_{\bbP_\sfs }^*$ under
control. 
\begin{equation}
\label{A:L2-projection}
\stepcounter{AssumptionCounter}
\tag{H\arabic{AssumptionCounter}}
\begin{minipage}{0.8\hsize}
There are constants $c = c(\bbP_\sfs )$ and $C = C(\bbP_\sfs )$ with $0 < c \leq C$ and such that $c\Norm{\cdot}_{\bbP_\sfs ^*} \leq \Norm{\calP_{\bbP_\sfs }^*\cdot}_{\bbP^*} \leq C\Norm{\cdot}_{\bbP_\sfs ^*}$ in $\bbP_\sfs ^*$.
\end{minipage}
\end{equation}
The upper bound here is equivalent to the $\bbP$-stability of the projection $\calP_{\bbP_\sfs }$. The lower bound can be formulated as an inf-sup condition and it is equivalent to the existence of a bounded right inverse of $\calP_{\bbP_\sfs }$, cf. \cite[Proposition~3.5]{Khan.Zanotti:22}.

Finally, the proof of the inf-sup stability in Lemma~\ref{L:inf-sup} below for vanishing $\sigma$ makes use of the following assumption. 
\begin{equation}
\label{A:inclusion}
\stepcounter{AssumptionCounter}
\tag{H\arabic{AssumptionCounter}}
\begin{minipage}{0.55\hsize}
The inclusion $\bbP_\sfs  \subseteq \bbD_\sfs$ holds true if $\sigma=0$.
\end{minipage}
\end{equation}
Notice that the spaces $\overline{\bbP}$ and $\bbD$ in \eqref{E:abstract-spaces-fluid-content} and \eqref{E:abstract-spaces-pressure-total}, respectively, satisfy the same inclusion.

\begin{remark}[Spurious pressure oscillations]
\label{R:spurious-pressure-oscillations}
The combination of the assumptions \eqref{A:Stokes-inf-sup} and \eqref{A:inclusion} prescribes that $\bbU_\sfs /\bbP_\sfs $ is a stable pair for the discretization of the Stokes equations if $\sigma = 0$. This property was observed both numerically \cite{Haga.Osnes.Langtangen.12} and theoretically \cite{Mardal.Rognes.Thompson:21} to be important to prevent from spurious pressure oscillations in certain regimes. Indeed, forgetting the time derivative for a moment (or, more precisely, discretizing it in time) the Biot's equations \eqref{E:BiotProblem-equations} are close to the Stokes equations for vanishing $\sigma$ and small $\kappa$.
\end{remark} 

Figure~\ref{F:discrete-spaces-diagram} summarizes the relation among the spaces and the operators introduced in this section. As announced, the structure is exactly as in Figure~\ref{F:abstract-spaces-diagram}. We refer to Remark~\ref{R:functional-setting} for the details of the notation.

\begin{figure}[ht]
	\[
	\xymatrixcolsep{4pc}
	\xymatrixrowsep{4pc}
	\xymatrix{
		\bbU_\sfs ^* 
		\ar@{<-}[r]^{\calE_\sfs }
		\ar@{<-}[d]^{\calD_\sfs ^*}
		&
		\bbU_\sfs 
		\ar[d]^{\calD}
		&
		L^2(\Domain)
		\ar[dl]^{\calP_{\bbD_\sfs }}
		\ar[dr]_{\calP_{\bbP_\sfs }}
		&
		\bbP_\sfs 
		\ar[r]^{\calL_\sfs }
		\ar[d]^{i}
		&
		\bbP_\sfs ^*
		\ar@{<-}[d]^{i^*}\\
		\bbD_\sfs ^*
		\ar@3{-}[r]
		&
		\bbD_\sfs 
		&&
		\bbP_\sfs 
		\ar@3{-}[r]&
		\bbP_\sfs ^*
	}
	\]
	\caption{\label{F:discrete-spaces-diagram} Spaces and operators describing the regularity in space for the discretization of the Biot's equations. The triple lines on the bottom indicate identification via the $L^2(\Domain)$-scalar product.}
\end{figure}
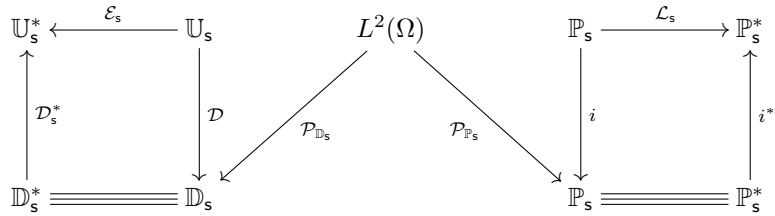

\subsection{Discretization in time}
\label{SS:abstract-discretisation-time}
As announced, we consider a simple first-order discretization in time, namely backward Euler. To this end, we introduce a partition
\begin{equation}
\label{E:mesh-time}
0=t_0<t_1<\cdots<t_J=T
\end{equation}
of the time interval $[0, T]$ with $J \geq 1$. We denote the local time intervals and their length by
\begin{equation}
\label{E:mesh-time-intervals}
I_j := [t_{j-1}, t_j]
\qquad \text{and} \qquad
|I_j| := t_j - t_{j-1} 
\end{equation}
respectively, for $j = 1,\dots, J$. 

For $\bbX \in \{ \bbU_\sfs, \bbD_\sfs, \bbP_\sfs \}$, we consider the space of piecewise time-constant functions on the above partition with values in $\bbX$
\begin{equation*}
\label{E:time-discretization-L2}
\bbS^0_\sft (\bbX) :=\big\{X \in L^\infty(\bbX) \mid X_{|I_j}=: X_j\in
\bbX,~j=1,\ldots,J\big\}. 
\end{equation*}
Whenever useful, we identify a function $X \in \bbS^0_\sft (\bbX)$ with the sequence of its values $(X_j)_{j=1}^J \subseteq \bbX$.

The space $\bbS^0_\sft(\bbX)$ is suitable for a first-order
discretization in time of $L^2(\bbX)$, i.e. of the first three
components in the trial \eqref{E:trial-space} and of the first four
components in the test \eqref{E:test-space}. The last component in the
trial space (the fluid content) is different, as it involves $H^1$-regularity in time. We discretize it in time by 
\begin{equation*}
\label{E:time-discretization-H1}
\bbS^0_\sft(\bbP_\sfs) \times \bbP_\sfs
\end{equation*}
where the second component plays the role of the initial value. Then we introduce a discrete counterpart $
\mathrm{d}_\sft : \bbS^0_\sft(\bbP_\sfs) \times \bbP_\sfs \to
\bbS^0_\sft (\bbP_\sfs^*)$ of the time derivative $\partial_t$, in the vein of the backward Euler scheme, i.e.
\begin{equation}
\label{E:discrete-time-derivative}
\mathrm{d}_\sft(\widetilde M, \widetilde M_0)_{I_j}
:=
\dfrac{\widetilde M_j - \widetilde M_{j-1}}{|I_j|}
\end{equation}
for $(\widetilde M, \widetilde M_0) \in \bbS^0_\sft(\bbP_\sfs) \times
\bbP_\sfs$ and for all $j = 1,\dots, J$.% , where we use
% \(\bbP_\sfs\hookrightarrow \bbP_\sfs^*\) by virtue of the
% representation with the \(L^2\) inner product.
% \todo{\@Pietro: Is it this what you had in mind?}

The proof of Theorem~\ref{T:well-posedness-weak-formulation} given in \cite{Kreuzer.Zanotti:23+} and, in particular, the control of the point values of the total fluid content hinge on the following integration by parts rule $2\int_0^T \left\langle \partial_t \widetilde m, \calL^{-1}\widetilde m\right\rangle_{\bbP} \dt = \Norm{\widetilde m(T)}^2_{\bbP^*} - \Norm{\widetilde m(0)}^2_{\bbP^*} $, which holds true for $\widetilde m \in H^1(\bbP^*)$, cf. \cite[Lemma~64.40]{Ern.Guermond:21c}. The operator $\mathrm{d}_\sft $ defined above satisfies a similar relation.  

\begin{lemma}[Time discrete integration by parts rule]
	\label{L:discrete-integration-by-parts}
	We have
	\begin{equation*}
	\label{E:discrete-integration-by-parts}
	2\int_{0}^{t_j} \langle \mathrm{d}_\sft  (\widetilde M, \widetilde M_0),
	\calL_\sfs^{-1} \widetilde M \rangle_{\bbP_\sfs} \dt
	\geq 
	\|\widetilde M_j\|_{\bbP_\sfs^*}^2- \|\widetilde M_0\|_{\bbP_\sfs^*}^2
	\end{equation*}
	for all $(\widetilde M, \widetilde M_0) \in \bbS^0_\sft (\bbP_\sfs) \times \bbP_\sfs$ and $j=1,\dots, J$.
\end{lemma}

\begin{proof}
We exploit \eqref{E:discrete-time-derivative}, rearrange terms and recall the second part of \eqref{E:norms-dual-discrete}. It results
	\begin{equation*}
	\begin{split}
	2\int_{0}^{t_j} \langle \mathrm{d}_\sft  (\widetilde M, \widetilde M_0),
	\widetilde \calL_\sfs^{-1}M \rangle_{\bbP_\sfs} \dt
	&=
	2 \sum_{k=1}^j\langle \widetilde M_k-\widetilde M_{k-1},
	\calL_\sfs^{-1}\widetilde M_k \rangle_{\bbP_\sfs}\\
	&=
	\Norm{\widetilde M_j}_{\bbP_\sfs^*}^2 
	+
	\sum_{k=1}^j\Norm{\widetilde M_k-\widetilde M_{k-1}}_{\bbP_\sfs^*}^2 
	- 
	\Norm{\widetilde M_0}_{\bbP_\sfs^*}^2
	\end{split}
	\end{equation*}
	cf. \cite[eq. (67.9)]{Ern.Guermond:21c}. This readily implies the claimed inequality.
\end{proof}

\subsection{Full discretization}
\label{SS:abstract-discretisation-full}
Combining the discretizations in space and time from the previous sections, we are in position to propose an abstract full discretization of the initial-boundary value problem \eqref{E:BiotProblem} for the Biot's equations. 

We consider the trial space 
\begin{equation}
\label{E:trial-space-discrete}
  \bbY_{1, \mathsf{st}}:=\bbS^0_{\mathsf{t}}(\bbU_\sfs )\times \bbS^0_\sft (\bbD_\sfs)\times
  \bbS^0_\sft (\bbP_\sfs )\times \bbS^0_\sft (\bbP_\sfs )\times \bbP_\sfs .
\end{equation}
Inspired by~\eqref{E:trial-norm} and taking the assumption \eqref{A:extended-norms} in section~\ref{SS:abstract-discretization-space} into account, we equip \(\bbY_{1,\sfs\sft}\) with
the norm
\begin{equation}
    \label{E:trial-norm-discrete}
    \begin{split}
      &\| (\widetilde U, \widetilde P_\Tot, \widetilde P, \widetilde M, \widetilde M_0) \|_{1,\sfs\sft}^2 := \\
      &\quad \phantom{+} \int_{0}^{T} \left( \Norm{\widetilde U}^2_\bbU + \dfrac{1}{\mu}\Norm{\widetilde P_\Tot}_{\Domain}^2 + \Norm{\mathrm{d}_\sft  (\widetilde M, \widetilde{M}_0) + \calL_\sfs  \widetilde P}^2_{\bbP_\sfs ^*} \right)\dt + \Norm{\widetilde M_0}^2_{\bbP^*_\sfs }\\
      &\quad  + \int_{0}^T \Big ( \dfrac{1 }{\mu + \lambda} \Norm{\lambda
          \calD_\sfs  \widetilde U - \widetilde P_\Tot - \alpha \calP_{\bbD_\sfs } \widetilde
          P}^2_{\Domain} + \gamma\Norm{\alpha \calP_{\bbP_\sfs } \calD_\sfs  \widetilde U +
          \sigma \widetilde P - \widetilde M}^2_{\Domain} \Big ) \dt.
    \end{split}
\end{equation}
  
\begin{remark}[Equivalent trial norm]
\label{R:equivalent-trial-norm}
According to the assumption \eqref{A:L2-projection} in section~\ref{SS:abstract-discretization-space}, we could equivalently replace the discrete dual norm $\Norm{\cdot}_{\bbP_\sfs ^*}$ in the definition of $\Norm{\cdot}_{1,\sfs\sft}$ by $\Norm{\calP_{\bbP_\sfs }^*\cdot}_{\bbP^*}$. All the results stated in section~\ref{SS:well-posedness} hold true also in this case, with the only difference that the hidden constants additionally depend on the constants in \eqref{A:L2-projection}. This observation is important in view of the definition of the error notion in section~\ref{SS:quasi-optimality}.
\end{remark}

For the test space, we proceed similarly and set
  \begin{equation}
    \label{E:test-space-discrete}
    \bbY_{2,\sfs\sft} := \bbS^0_\sft (\bbU_\sfs ) \times \bbS^0_\sft (\bbD_\sfs) \times
    \bbS^0_\sft (\bbP_\sfs ) \times \bbS^0_\sft (\bbP_\sfs ) \times
    \bbP_\sfs .
  \end{equation}
Recalling again the assumption \eqref{A:extended-norms} in section~\ref{SS:abstract-discretization-space}, we equip $\bbY_{2,\sfs\sft}$ with the norm \(\Norm{\cdot}_2\) in \eqref{E:test-norm}.

Let $(L_u, L_p, L_0) \in \bbS^0_\sft (\bbU_\sfs ^*) \times \bbS^0_\sft (\bbP^*_\sfs ) \times \bbP_\sfs^*$ be a discretization of the corresponding data $(\ell_u, \ell_p, \ell_0) \in L^2(\bbU^*)\times L^2(\bbP^*) \times \bbP^*$ in the weak formulation \eqref{E:BiotProblem-weak-formulation} of the Biot's equations. We consider the following full discretization of \eqref{E:BiotProblem-weak-formulation}: find $(U, P_\Tot, P, M, M_0) \in \bbY_{1, \sfs\sft}$ such that
\begin{equation}
\label{E:BiotProblem-discretization}
\begin{alignedat}{2}
\calE_\sfs  U + \calD_\sfs ^* P_\Tot &= L_u \qquad && \text{in $\bbS^0_\sft (\bbU_\sfs ^*)$}\\
\lambda \calD_\sfs  U - P_\Tot - \alpha \calP_{\bbD_\sfs } P &= 0 \qquad && \text{in $\bbS^0_\sft (\bbD_\sfs)$}\\
\alpha \calP_{\bbP_\sfs } \calD_\sfs  U + \sigma P - M &= 0 \qquad && \text{in $\bbS^0_\sft (\bbP_\sfs)$}\\
\mathrm{d}_\sft (M, M_0) + \calL_\sfs  P &= L_p \qquad  && \text{in $\bbS^0_\sft (\bbP_\sfs ^*)$}\\
M_0 &= L_0 \qquad && \text{in $\bbP_\sfs^*$}.
\end{alignedat}
\end{equation} 
Note that, also in this case, there are some ambiguities between functions and functionals, that can be clarified in the vein of Remark~\ref{R:functions-functionals}.

\begin{remark}[Two- and four-field formulation]
\label{R:2-4-field-formulation}
The weak formulation~\ref{E:BiotProblem-weak-formulation} involves four unknown variables but it can be equivalently rewritten as a two-field weak formulation of the initial-boundary value problem~\eqref{E:BiotProblem}, by eliminating the total pressure $p_\Tot$ and the total fluid content $m$, cf. \cite[Remark~2.7]{Kreuzer.Zanotti:23+}. Analogously, we could eliminate the discrete total pressure $P_\Tot$ and the discrete fluid content $M$ from \eqref{E:BiotProblem-discretization}. In this way, we would equivalently obtain a discretization of the two-field weak formulation, with trial and test spaces given by $\bbS^0_\sft (\bbU_\sfs ) \times \bbS^0_\sft (\bbP_\sfs ) \times \bbP_\sfs $. 
\end{remark}

Since we aim at establishing the well-posedness of \eqref{E:BiotProblem-discretization} via the inf-sup theory, it is convenient viewing it as an instance of the following linear variational problem: given $L \in \bbY_{2,\sfs\sft}^*$, find $Y_1 \in \bbY_{1, \sfs\sft}$ such that
\begin{equation}
\label{E:linear-variational-problem-discrete}
b_{\sfs\sft}(Y_1, Y_2) = \left\langle L, Y_2\right\rangle_{\bbY_{2,\sfs\sft}} \qquad \forall Y_2 \in \bbY_{2,\sfs\sft}.
\end{equation}
Of course, this can be seen as a discretization of \eqref{E:linear-variational-problem}. Here, the bilinear form \(b_{\sfs\sft}:\bbY_{1,\sfs\sft}\times
  \bbY_{2,\sfs\sft}\to\R\) is defined by
\begin{equation}
\label{E:BiotProblem-discrete-form}
\begin{split}
  &b_{\sfs\sft}(\widetilde Y_1, Y_2) :=\\
  &\qquad \phantom{+} \int_0^T \left( \langle \calE_\sfs  \widetilde U +
    \calD_\sfs ^* \widetilde P_\Tot, V \rangle_{\bbU_\sfs } + \langle
    \mathrm{d}_\sft  (\widetilde M, \widetilde M_0) + \calL_\sfs  \widetilde P, N \rangle_{\bbP_\sfs }
\right)\dt
+ \langle \widetilde M_0, N_0 \rangle_{\bbP_\sfs }
\\
& \qquad 
+ \int_0^T  \Big(( \lambda \calD_\sfs  \widetilde U - \widetilde
    P_\Tot - \alpha \calP_{\bbD_\sfs }\widetilde P, Q_\Tot)_{\Domain} +
  ( \alpha \calP_{\bbP_\sfs }\calD_\sfs  \widetilde U + \sigma \widetilde P - \widetilde M, Q
  )_{\Domain} \Big)\dt
\end{split} 
\end{equation}
for $\widetilde Y_1 = (\widetilde U, \widetilde P_\Tot, \widetilde P, \widetilde M,\widetilde M_0) \in
\bbY_{1,\sfs\sft}$ and $Y_2 = (V, Q_\Tot, Q, N, N_0) \in
\bbY_{2,\sfs\sft} $. 

A remarkable difference between \eqref{E:linear-variational-problem}
and \eqref{E:linear-variational-problem-discrete} is that, in the
latter one, we do not have to explicitly take the closure of the trial
space. Indeed, $\bbY_{1, \sfs\sft}$ is certainly closed, being
finite-dimensional. 

\subsection{Well-posedness}
\label{SS:well-posedness}
The goal of this section is to establish the well-posedness of the
discretization \eqref{E:BiotProblem-discretization} by means of the
inf-sup theory. Since the trial space $\bbY_{1, \sfs\sft}$ and the
test space $\bbY_{2,\sfs\sft}$ are finite-dimensional with equal
dimension, we can use a simplified version of the Banach-Ne\u{c}as
theorem, which does not require to verify the nondegeneracy of the
form $b_{\sfs\sft}$, see
e.g. \cite[Theorem~26.6]{Ern.Guermond:21b}. Therefore, the
well-posedness is equivalent to the properties verified in the next
two lemmas. 

  \begin{lemma}[Boundedness]\label{L:boundedness}
    The bilinear form $b_{\sfs\sft}$
    in~\eqref{E:BiotProblem-discrete-form} satisfies
    \begin{align*}
      \sup_{Y_2 \in \bbY_{2,\sfs\sft}} \dfrac{b_{\sfs\sft}(\widetilde Y_1, Y_2)}{\Norm{Y_2}_2} \lesssim \Norm{\widetilde Y_1}_{1,\sfs\sft}
    \end{align*}
    for all $\widetilde Y_1\in \bbY_{1,\sfs\sft}$. The hidden constant
    depends only on the constants in the assumptions
    \eqref{A:extended-norms} and \eqref{A:Stokes-inf-sup} in
    section~\ref{SS:abstract-discretization-space}.  
  \end{lemma}
  
  \begin{proof}
    The claimed bound follows from the Cauchy-Schwarz
    inequality applied to each term in the definition of
    $b_{\sfs\sft}$, in combination with \eqref{E:norms-discrete} and
    with the norm equivalences in the assumptions
    \eqref{A:extended-norms} and \eqref{A:Stokes-inf-sup}. 
  \end{proof}

\begin{lemma}[Inf-sup stability]
  \label{L:inf-sup}
  The bilinear form $b_{\sfs\sft}$
  in~\eqref{E:BiotProblem-discrete-form} satisfies
\begin{equation*}
  \label{E:inf-sup}
\begin{split}
&\sup_{Y_2\in \bbY_{2,\sfs\sft}}
\frac{b_{\sfs\sft} (\widetilde Y_1,Y_2)}{\Norm{Y_2}_{2}}
\gtrsim\\
& \qquad\qquad (1+T)^{-\frac{1}{2}}
\left( \| \widetilde Y_1 \|_{1,\sfs\sft}^2 +
  \Norm{\widetilde M}^2_{L^\infty(\bbP_\sfs ^*)}
+ \int_{0}^T \left( \lambda
    \Norm{\calD_\sfs  \widetilde U}^2_{\Domain} + \gamma^{-1} \Norm{\widetilde P
    }^2_{\Domain}\right) \dt \right)^{\frac{1}{2}} .
\end{split}
\end{equation*}
for all \(\widetilde Y_1 = (\widetilde U, \widetilde P_\Tot,
\widetilde P, \widetilde M, \widetilde M_0) \in
\bbY_{1,\sfs\sft}\). The hidden constant depends only on the constants
in the assumptions \eqref{A:extended-norms} and
\eqref{A:Stokes-inf-sup} in
section~\ref{SS:abstract-discretization-space}. 
\end{lemma}

\begin{proof}
See section~\ref{SS:inf-sup}.
\end{proof}

The combination of the two above lemmas implies the main result of this section. 

  \begin{theorem}[Well-posedness of the full discretization]
\label{T:well-posedness-discretization}
For all possible loads $(L_u, L_p, L_0) \in \bbS^0_\sft (\bbU^*_\sfs ) \times \bbS^0_\sft (\bbP_\sfs ^*) \times \bbP_\sfs ^*$, the equations \eqref{E:BiotProblem-discretization} have a unique solution $Y_1  = (U, P_\Tot, P, M, M_0) \in \bbY_{1, \sfs\sft}$, which satisfies the two-sided stability bound
\begin{equation}
\label{E:well-posedness-discretization}
\begin{split}
\Norm{Y_1}_{1,\sfs\sft}^2  
\eqsim 
\int_{0}^T \left(
 \Norm{L_u}^2_{\bbU^*_\sfs } +
 \Norm{L_p}^2_{\bbP^*_\sfs } \right)\dt +
 \Norm{L_0}^2_{\bbP^*_\sfs }.
\end{split}
\end{equation}
Moreover, we have the norm equivalence
\begin{equation}
\label{E:well-posedness-discretization-norm-equivalence}
\Norm{Y_1}_{1,\sfs\sft}^2  \eqsim \Norm{Y_1}_{1,\sfs\sft}^2   + \Norm{M}^2_{L^\infty(\bbP_\sfs ^*)}
+\int_0^T\left( \lambda \Norm{\calD_\sfs  U}^2_{\Domain}  
+ \gamma^{-1} \Norm{P}^2_{\Domain} \right)\dt.
\end{equation}
The hidden constants depend only on the final time
$T$ and on the constants in the assumptions \eqref{A:extended-norms}
and \eqref{A:Stokes-inf-sup} in
section~\ref{SS:abstract-discretization-space}. 
\end{theorem}

\begin{proof}
The equations \eqref{E:BiotProblem-discretization} are equivalent to
the linear variational problem
\eqref{E:linear-variational-problem-discrete} with load $L = (L_u, 0,
0, L_p, L_0) \in \bbY_{2,\sfs\sft}^*$. Then, the simplified version of
the Banach-Ne\u{c}as theorem for finite-dimensional linear variational
problems \cite[Theorem~26.6]{Ern.Guermond:21b} implies existence and
uniqueness of the solution, in combination with
Lemmas~\ref{L:boundedness} and \ref{L:inf-sup}. The identity
$b_{\sfs\sft}(Y_1, \cdot) = L$ in $\bbY_{2,\sfs\sft}^*$ further
implies 
\begin{equation*}
\begin{split}
\| Y_1 \|_{1,\sfs\sft}^2 +
\Norm{ M}^2_{L^\infty(\bbP_\sfs ^*)}
+ \int_{0}^T \Big( \lambda&
\Norm{\calD_\sfs   U}^2_{\Domain} + \gamma^{-1} \Norm{ P}^2_{\Domain}\Big) \dt \\
&\lesssim
\Norm{(L_u, 0, 0, L_p, L_0)}_{2,*}^2
\lesssim \Norm{Y_1}_{1,\sfs\sft}^2
\end{split}
\end{equation*}
according to estimates in Lemmas~\ref{L:boundedness} and
\ref{L:inf-sup}. The hidden constants depend only on the ones
therein. Here $\Norm{\cdot}_{2,*}$ denotes the dual of the test
norm. This readily implies the second claimed equivalence. The first one follows
by recalling the definition of the norm $\Norm{\cdot}_2$ in \eqref{E:test-norm}. 
\end{proof}

\begin{remark}[Jumps of the discrete fluid content]
Recall that the component $M$ of the solution $Y_1$ in
Theorem~\ref{T:well-posedness-discretization} represents the
discretization of the total fluid content $m$ in
Theorem~\ref{T:well-posedness-weak-formulation}. While the latter one
is guaranteed to be continuous in time, the former one is not, (indeed
it is piecewise constant). Still, a careful inspection at the proofs
of Lemmas~\ref{L:discrete-integration-by-parts} and \ref{L:inf-sup}
reveals that the left-hand side in the stability bound
\eqref{E:well-posedness-discretization} controls the jumps of $M$ in
time, namely 
\begin{equation*}
\sum_{j=1}^J\Norm{M_j- M_{j-1}}_{\bbP_\sfs ^*}^2 \lesssim \| Y_1 \|_{1,\sfs\sft}^2. 
\end{equation*}
This can be viewed as a discrete counterpart of the continuity of $m$.
\end{remark}

\subsection{Inf-sup stability}
\label{SS:inf-sup}
This section is devoted to the proof of Lemma~\ref{L:inf-sup}. The
argument is similar to the one in
\cite[section~3.1]{Kreuzer.Zanotti:23+}, which establishes the inf-sup
stability of the form $b$ in
\eqref{E:BiotProblem-abstract-form}. Nevertheless, some subtleties
exist with regard to the discretization and we therefore include a
detailed proof so as to keep the discussion as much complete and
self-contained as possible. 

Let \(\widetilde Y_1=(\widetilde U,\widetilde P_\Tot,\widetilde
P,\widetilde M,\widetilde M_0)\in\bbY_{1,\sfs\sft}\) be given and set  
\begin{equation*}
\widetilde{H}_{\bbD_\sfs} := \lambda \calD_\sfs  \widetilde U - \widetilde P_\Tot - \alpha
\calP_{\bbD_\sfs }\widetilde P
\qquad \text{and} \qquad
\widetilde{H}_{\bbP_\sfs} := \alpha \calP_{\bbP_\sfs }\calD_\sfs 
\widetilde U + \sigma \widetilde P - \widetilde M
\end{equation*}
for shortness. Since the projections $\calP_{\bbD_\sfs }$ and
$\calP_{\bbP_\sfs }$ map onto $\bbD_\sfs$ and $\bbP_\sfs $,
respectively, we have $\widetilde H_{\bbD_\sfs} \in
\bbS^0_\sft(\bbD_\sfs)$ as well as $\widetilde H_{\bbP_\sfs} \in
\bbS^0_\sft(\bbP_\sfs) $. We consider the test function  
\begin{equation*}
\begin{split}
Y_{2,j} := \Big (\; (\widetilde U + \calE^{-1}_\sfs  \calD_\sfs ^*\widetilde P_\Tot)\chi_{j}, \;
\dfrac{4\max \{1, C\}}{\mu+\lambda}\widetilde H_{\bbD_\sfs}\chi_{j},
\dfrac{4\gamma}{\min\{1,c\}} \widetilde H_{\bbP_\sfs}\chi_{j},&\\
\; \calL_\sfs ^{-1} (2 \widetilde M + \mathrm{d}_\sft  (\widetilde M,
\widetilde M_0) + \calL_\sfs  \widetilde P)\chi_{j},& 
\; \calL_\sfs ^{-1} \widetilde M_0 \;\Big) 
\end{split}
\end{equation*}  
with $j = 1,\dots, J$, where $\chi_j: [0, T] \to \R$ denotes the
indicator function of the interval $[0, t_j]$. Here the constants $c$
and $C$ are as in the assumption \eqref{A:Stokes-inf-sup} in
section~\ref{SS:abstract-discretization-space}. We refer to
\cite[Remark~3.9]{Kreuzer.Zanotti:23+} for a motivation of the test
function.   
  
First of all, notice that we indeed have that $Y_{2,j} \in
\bbY_{2,\sfs\sft}$, i.e. it is an admissible test function. Indeed,
recall the definition of the spaces $\bbY_{1, \sfs\sft}$ and
$\bbY_{2,\sfs\sft}$ in \eqref{E:trial-space-discrete} and
\eqref{E:test-space-discrete}. The operator $\calE_\sfs ^{-1}
\calD_\sfs ^*$ maps $\bbD_\sfs$ into $\bbU_\sfs $ and $\calL_\sfs $ is
an isometry between $\bbP_\sfs $ and $\bbP_\sfs ^*$.  
Moreover, the indicator function $\chi_j$ is piecewise constant on the
partition \eqref{E:mesh-time} of $[0, T]$. Hence, the multiplication
by $\chi_j$ preserves the inclusion in $\bbY_{2,\sfs\sft}$. 

The proof consists of two main steps. First, we establish the lower bound
\begin{equation}
\label{E:inf-sup-proof-lower-bound}
\begin{split}
b_{\sfs\sft}(\widetilde Y_1, Y_{2,j} + Y_{2,J}) \gtrsim 
&\| \widetilde Y_1 \|_{1,\sfs\sft}^2 + \Norm{\widetilde
  M}^2_{L^\infty(\bbP^*_\sfs)} + \int_{0}^T \left( \lambda
  \Norm{\calD_\sfs  \widetilde U}^2_{\Domain} + \gamma^{-1}
  \Norm{\widetilde P}^2_{\Domain}\right) \dt  
\end{split}
\end{equation}
for a suitable index $1\leq j\leq J$. Then, we estimate the norm of the test function
\begin{equation}
\label{E:inf-sup-proof-upper-bound}
\Norm{Y_{2,j}}^2_2 \lesssim \| \widetilde Y_1 \|_{1,\sfs\sft}^2 + (1+T)\Norm{\widetilde M}^2_{L^\infty(\bbP^*_\sfs)}
\end{equation}
for some hidden constant independent of $j$. The combination of these
inequalities readily implies the inf-sup stability claimed in
Lemma~\ref{L:inf-sup}. 

We start with the proof of \eqref{E:inf-sup-proof-lower-bound}. Using
the test function $Y_{2,j}$ in the definition
\eqref{E:BiotProblem-discrete-form} of the form $b_{\sfs\sft}$ yields 
\begin{equation}
\label{E:inf-sup-semiD-action}
\begin{alignedat}{4}
b_{\sfs\sft}(\widetilde Y_1, Y_{2,j}) &=
\int_0^{t_j} \langle \calE_\sfs  \widetilde U + \calD^*_\sfs
\widetilde P_\Tot, \widetilde U + \calE^{-1}_\sfs  \calD^*_\sfs
\widetilde P_\Tot \rangle_{\bbU_\sfs } \dt &  \qquad (=: \mathfrak
I_1)
\\ 
&+ 2\int_{0}^{t_j} \langle \mathrm{d}_\sft  (\widetilde M,
\widetilde{M}_0) + \calL_\sfs  \widetilde P, \calL_\sfs ^{-1}
\widetilde M \rangle_{\bbP_\sfs }\dt & \qquad (=: \mathfrak I_2)
\\ 
&+ \Norm{\widetilde M_0}_{\bbP_\sfs ^*}^2
+ \int_{0}^{t_j} \Norm{\mathrm{d}_\sft  (\widetilde M,
  \widetilde{M}_0) + \calL_\sfs  \widetilde P}_{\bbP_\sfs ^*}^2\dt
\\
&+ \int_0^{t_j} \dfrac{4\max\{1,C\}}{\mu+\lambda} \Norm{\widetilde
  H_{\bbD_\sfs}}_{\Domain}^2 \dt
\\
&+ \int_0^{t_j} \dfrac{4\gamma}{\min\{1,c\}} \Norm{\widetilde H_{\bbP_\sfs}}^2_{\Domain} \dt.
\end{alignedat}
\end{equation}
Notice that the dual norms on the third line are obtained thanks to
the second part of \eqref{E:norms-dual-discrete}. We are led to
analyze the terms $\mathfrak{I}_1$ and $\mathfrak{I}_2$ on the
right-hand side.  

Regarding the first term, we have
\begin{align*}
  \mathfrak{I}_1=\int_0^{t_j} \left(\langle  \calE_\sfs  \widetilde U,\widetilde
  U\rangle_{\bbU_\sfs }+2\langle\calD^*_\sfs \widetilde P_\Tot, \widetilde U \rangle_{\bbU_\sfs }
  +\langle{\calD_\sfs ^* \widetilde P}, \calE_\sfs ^{-1}\calD_\sfs ^* \widetilde P_\Tot\rangle_{\bbU_\sfs }\right) \dt.
\end{align*}
We rewrite the first summand with the help of the first identity in \eqref{E:norms-discrete}. Analogously, we estimate the third summand from below by means of the first identity in \eqref{E:norms-dual-discrete} and the assumption \eqref{A:Stokes-inf-sup} in section~\ref{SS:abstract-discretization-space}
\begin{align*}
  \langle{\calE_\sfs \widetilde
  U,\widetilde U\rangle_{\bbU_\sfs }= \|\widetilde
  U\|_{\bbU_\sfs }^2\quad\text{and}\quad \langle{\calD_\sfs ^* \widetilde P_\Tot},
  \calE_\sfs ^{-1}\calD_\sfs ^* \widetilde P_\Tot}\rangle_{\bbU_\sfs } \ge
  \frac{c}{\mu}\|\widetilde P_\Tot\|_{\Domain}^2.
\end{align*}
We estimate the remaining term by observing that the assumption \eqref{A:Stokes-inf-sup} implies also $\mu \Norm{\calD_\sfs  \cdot}_{\Domain}^2 \leq C \Norm{\cdot}_{\bbU_\sfs }^2$ in $\bbU_\sfs $. This bound and a Young's inequality imply
\begin{equation*}
\begin{split}
\langle \calD^*_\sfs  \widetilde P_\Tot, \widetilde U  \rangle_{\bbU_\sfs }
&= 
-( \calD_\sfs  \widetilde U, \widetilde H_{\bbD_\sfs} )_{\Domain} 
+ 
\lambda \Norm{\calD_\sfs \widetilde U}^2_{\Domain} 
- 
\alpha (  \calD_\sfs  \widetilde U, \calP_{\bbD_\sfs } \widetilde P)_{\Domain}
\\
&\geq \dfrac{3\lambda}{4} \Norm{\calD_\sfs  \widetilde U}^2_{\Domain} 
- 
\frac{1}{4}\Norm{ \widetilde U}^2_{\bbU_\sfs } 
-
\dfrac{\max\{1,C\}}{\mu+\lambda} \Norm{\widetilde H_{\bbD_\sfs}}^2_{\Domain}
-
\alpha ( \calD_\sfs  \widetilde U, \widetilde P )_{\Domain}. 
\end{split}
\end{equation*}
Notice that we were able to omit the projection $\calP_{\bbD_\sfs }$
in the last term thanks to the inclusion $\calD_\sfs  \widetilde U \in
\bbS^0_\sft (\bbD_\sfs)$. Combining this bound with the identities
above reveals  
\begin{equation*}
  \begin{split}
    \mathfrak{J}_1&\ge \int_0^{t_j} \left( \frac12\|\widetilde
        U\|_{\bbU_\sfs }^2+\frac{c}{\mu}\|\widetilde P_\Tot\|_{\Domain}^2
      +\dfrac{3\lambda}{2} \Norm{\calD_\sfs  \widetilde U}^2_{\Domain}\right.
    \\
    &\qquad\qquad -\left.  \dfrac{2\max\{1,C\}}{\mu+\lambda} \Norm{\widetilde H_{\bbD_\sfs}}^2_{\Domain} -2\alpha
      ( \calD_\sfs  \widetilde U, \widetilde P )_{\Domain}\right)\dt.
  \end{split}
\end{equation*}

Regarding the other critical term in \eqref{E:inf-sup-semiD-action}, we have
\begin{equation*}
\mathfrak{I}_2 = 2 \int_{0}^{t_j} \Big( \langle \mathrm{d}_\sft  (\widetilde M, \widetilde M_0),
  \calL^{-1}_\sfs  \widetilde M \rangle_{\bbP_\sfs } + ( \widetilde M,
  \widetilde P )_{\Domain}  \Big)\dt.  
\end{equation*}
The use of the $L^2(\Domain)$-scalar product in the second summand is justified by the discussion after \eqref{E:Hilbert-triplet-discrete}. We estimate the first summand from below by invoking Lemma~\ref{L:discrete-integration-by-parts}
\begin{equation*}
    2\int_{0}^{t_j} \langle \mathrm{d}_\sft  (\widetilde M, \widetilde{M}_0), \calL^{-1}_\sfs 
    \widetilde M \rangle_{\bbP_\sfs } \dt
    \ge \|\widetilde M_j\|_{\bbP_\sfs ^*}^2- \|\widetilde M_0\|_{\bbP_\sfs ^*}^2.
\end{equation*}
To investigate the second summand, assume first $\sigma > 0$.
A Young's inequality reveals 
\begin{equation*}
\begin{split}
( \widetilde M, \widetilde P)_\Domain 
&= 
( \widetilde H_{\bbP_\sfs}, \widetilde P )_\Domain 
+ 
\alpha ( \calP_{\bbP_\sfs }\calD_\sfs 
\widetilde U, \widetilde P )_\Domain 
+ 
\sigma \Norm{\widetilde P}^2_\Domain 
\\
&\geq  \alpha ( \calD_\sfs  \widetilde U, \widetilde P )_\Domain 
+ 
\dfrac{\sigma}{2} \Norm{\widetilde P}^2_\Domain 
- 
\dfrac{1}{2\sigma} \Norm{\widetilde H_{\bbP_\sfs}}^2_\Domain. 
\end{split}
\end{equation*}
As before, we omit the projection $\calP_{\bbP_\sfs }$
thanks to the inclusion $\widetilde P \in \bbS^0_\sft (\bbP_\sfs )$.  
Alternatively, for general $\sigma \geq 0$, it holds that 
\begin{equation*}
\begin{split}
( \widetilde M, \widetilde P)_\Domain
&= 
( \widetilde H_{\bbP_\sfs}, \widetilde P )_\Domain
+ 
\alpha ( \calP_{\bbP_\sfs }\calD_\sfs  \widetilde U, \widetilde P )_\Domain 
+ 
\sigma \Norm{\widetilde P}^2_\Domain 
\\
&=-\dfrac{1}{\alpha} ( \widetilde H_{\bbP_\sfs}, \widetilde H_{\bbD_\sfs} )_\Domain  
+
\dfrac{\lambda}{\alpha} ( \widetilde H_{\bbP_\sfs},  \calD_\sfs  \widetilde U )_\Domain
-
\dfrac{1}{\alpha} ( \widetilde H_{\bbP_\sfs}
, \widetilde P_\Tot )_\Domain\\
& \quad+
( \widetilde H_{\bbP_\sfs}, \widetilde P - \calP_{\bbD_\sfs } \widetilde P )_\Domain
+ 
\alpha ( \calD_\sfs  \widetilde U, \widetilde P )_\Domain 
+ 
\sigma \Norm{\widetilde P}^2_\Domain.  
\end{split}
\end{equation*}
Bounding the term  $( \widetilde H_{\bbP_\sfs}, \widetilde P - \calP_{\bbD_\sfs } \widetilde P
)_\Domain$ conveniently is subtle. Whenever $\bbP_\sfs 
\subseteq \bbD_\sfs$, we have $\calP_{\bbD_\sfs } \widetilde P = \widetilde P$, hence
$( \widetilde H_{\bbP_\sfs}, \widetilde P - \calP_{\bbD_\sfs } \widetilde P 
)_\Domain = 0$. When
the above inclusion fails, we have $\sigma > 0$ due to the assumption \eqref{A:inclusion} in section~\ref{SS:abstract-discretization-space}. Then, we apply Young's inequality to obtain $( \widetilde H_{\bbP_\sfs}, \widetilde P - \calP_{\bbD_\sfs } \widetilde P
)_\Domain \leq  \Norm{\widetilde
	H_{\bbP_\sfs}}^2_\Domain/(2\sigma) + \sigma\Norm{\widetilde P}^2_\Domain/2$.   
Combining this observation with other applications of Young's inequality, the previous lower bound and the definition \eqref{E:gamma} of $\gamma$,  we arrive at
\begin{equation*}
\begin{split}
( \widetilde M, \widetilde P)_\Domain 
\geq  
&-\dfrac{\max\{1,C\}}{2(\mu+\lambda)} \Norm{\widetilde H_{\bbD_\sfs}}^2_\Domain 
-
\dfrac{c}{4 \mu} \Norm{\widetilde P_\Tot}^2_\Domain 
-
\dfrac{\lambda}{2} \Norm{\calD_\sfs  \widetilde U}^2_\Domain
\\
&+ \dfrac{\sigma}{2} \Norm{\widetilde P}^2_\Domain
-\dfrac{3\gamma}{2 \min \{1,c\}}
\Norm{\widetilde H_{\bbP_\sfs}}_\Domain^2 
+ 
\alpha ( \calD_\sfs  \widetilde U, \widetilde
  P )_\Domain. 
\end{split}
\end{equation*}
By combining the last two bounds with the above identity for
$\mathfrak{I}_2$, we obtain
\begin{equation*}
\begin{split}
\mathfrak{I}_2 
&\geq \|\widetilde M_j\|_{\bbP_\sfs ^*}^2- \|\widetilde M_0\|_{\bbP_\sfs ^*}^2
\\ 
&\quad+ \int_0^{t_j} \left(-\dfrac{\max\{1,C\}}{\mu+\lambda} \Norm{\widetilde{H}_{\bbD_\sfs}}^2_\Domain 
-
\dfrac{c }{2 \mu} \Norm{\widetilde P_\Tot}^2_\Domain 
-
\dfrac{\lambda}{2} \Norm{\calD_\sfs  \widetilde U}^2_\Domain 
+ 
\sigma
\Norm{\widetilde P}^2_\Domain\right.
\\ 
&\qquad \qquad \; - \left.\dfrac{3\gamma}{\min\{1,c\}} \Norm{\widetilde H_{\bbP_\sfs}}^2_\Domain + 2\alpha ( \calD_\sfs  \widetilde U, \widetilde P )_\Domain\right)\dt.
\end{split} 
\end{equation*}

After this preparation, we are in position to choose a specific test
function in order to establish \eqref{E:inf-sup-proof-lower-bound}. We
let $\overline j\in\{1,\ldots,J\}$ be such that $\Norm{\widetilde
  M_{\overline j}}_{\bbP_\sfs^*} = \Norm{\widetilde
  M}_{L^\infty(\bbP_\sfs^*)}$. Then, we insert the previous lower
bounds of $\mathfrak{I}_1$ and $\mathfrak{I}_2$ into
\eqref{E:inf-sup-semiD-action}, resulting in
\begin{equation*}
\label{E:inf-sup-test-function-action}
\begin{split}
b_{\sfs\sft}&(\widetilde Y_1, Y_{2, \overline{j}} + Y_{2, J}) \geq\\ 
&\int_0^T \left( \frac12\Norm{\widetilde U}^2_{\bbU_\sfs } +\dfrac{c}{2\mu} \Norm{\widetilde P_\Tot}^2_\Domain +  \Norm{\mathrm{d}_\sft  (\widetilde M, \widetilde M_0) + \calL_\sfs  \widetilde
  P}^2_{\bbP^*_\sfs }  \right)\dt
\\
+&\int_0^T \left( \dfrac{1}{\mu+\lambda} \Norm{\widetilde{H}_{\bbD_\sfs}}^2_\Domain + \gamma \Norm{\widetilde H_{\bbP_\sfs}}^2_\Domain +
\dfrac{\lambda}{2} \Norm{\calD_\sfs  \widetilde U}^2_\Domain + \sigma
\Norm{\widetilde P}^2_\Domain  \right)\dt + \Norm{\widetilde M}^2_{L^{\infty}(\bbP^*_\sfs )} .
\end{split}
\end{equation*}
In combination with the definition \eqref{E:trial-norm-discrete} of the norm $\Norm{\cdot}_{1,\sfs\sft}$ and the assumption \eqref{A:extended-norms} in section~\ref{SS:abstract-discretization-space}, this almost establishes \eqref{E:inf-sup-proof-lower-bound}. Indeed, we have
\begin{equation*}
\begin{split}
b_{\sfs\sft}(\widetilde Y_1, Y_{2,\overline j} + Y_{2,J}) \gtrsim 
&\| \widetilde Y_1 \|_{1,\sfs\sft}^2 + \Norm{\widetilde M}^2_{L^\infty(\bbP^*_\sfs)} + \int_{0}^T \left( \lambda \Norm{\calD_\sfs  \widetilde U}^2_{\Domain} + \sigma \Norm{\widetilde P}^2_{\Domain}\right) \dt.
\end{split}
\end{equation*}
The proof that we can indeed replace $\sigma$ by $\gamma^{-1}$ in the last summand follows verbatim the argument in the proof of \cite[Proposition~4.1]{Kreuzer.Zanotti:23+}.

The last step of the proof consists in establishing
\eqref{E:inf-sup-proof-upper-bound}. According to the definition \eqref{E:test-norm}
of the test norm $\Norm{\cdot}_2$, we have for all $j = 1, \dots, J$, that
\begin{equation*}
\begin{alignedat}{4}
\Norm{Y_{2,j}}_2^2 
&= 
\int_{0}^{t_j}  \Big (\Norm{\widetilde U + \calE^{-1}_\sfs  \calD^*_\sfs  \widetilde
  P_\Tot}_{\bbU}^2  + \Norm{\calL_\sfs ^{-1}(2\widetilde M + \mathrm{d}_\sft  (\widetilde M, \widetilde M_0)
  + \calL_\sfs  \widetilde P)}^2_\bbP \Big )\dt
\\
&+\int_0^{t_j} \Big(\dfrac{16 \max\{1,C^2\}}{\mu+\lambda} \Norm{\widetilde H_{\bbD_\sfs}}^2_\Domain 
+ 
\dfrac{16 \gamma}{\min\{1,c^2\}} \Norm{\widetilde H_{\bbP_\sfs}}^2_\Domain \Big )\dt+ \Norm{\calL_\sfs ^{-1} \widetilde M_0}^2_\bbP.
\end{alignedat}
\end{equation*}

We exploit the identities \eqref{E:norms-discrete} and \eqref{E:norms-dual-discrete} and the norm equivalences in the assumptions \eqref{A:extended-norms} and \eqref{A:Stokes-inf-sup} from section~\ref{SS:abstract-discretization-space}. We extend also the integrals from the interval $[0, t_j]$ to $[0, T]$. Hence, we obtain
\begin{equation*}
\begin{alignedat}{4}
\Norm{Y_{2,j}}_2^2 
&\lesssim 
\int_{0}^T  \Big (\Norm{\widetilde U}^2_{\bbU} +
\dfrac{1}{\mu}\Norm{\widetilde P_\Tot}_\Domain^2  + \Norm{\widetilde
  M}_{\bbP^*_\sfs }^2 + \Norm{\mathrm{d}_\sft (\widetilde M, \widetilde M_0) + \calL_\sfs 
  \widetilde P}^2_{\bbP_\sfs ^*} \Big )\dt
\\ 
&\quad+\int_0^T \Big(\dfrac{1}{\mu+\lambda} \Norm{\widetilde H_{\bbD_\sfs}}^2_\Domain 
+ 
\gamma \Norm{\widetilde H_{\bbP_\sfs}}^2_\Domain \Big )\dt 
+ \Norm{\widetilde M_0}^2_{\bbP^*_\sfs }.
\end{alignedat}
\end{equation*}
Finally, we recall the definition \eqref{E:trial-norm-discrete} of the
norm $\Norm{\cdot}_{1,\sfs\sft}$ and exploit the upper bound $\int_0^T
\Norm{\cdot}_{\bbP_\sfs ^*}^2\dt \leq T
\Norm{\cdot}^2_{L^\infty(\bbP_\sfs ^*)}$. This establishes
\eqref{E:inf-sup-proof-upper-bound} for some hidden constant
independent of $j$ and concludes the proof. 

\begin{remark}[Time-dependent space discretization]
\label{R:time-dependent-space-discretization}
In our framework the space discretization does not change in time and
this is important for the proof of the inf-sup stability in this
section. Indeed, a change in the space discretisation would imply the
change of the operator \(\calL_\sfs \) in time and this would have an
effect on our use of the integration by parts formula from
Lemma~\ref{L:discrete-integration-by-parts}. For the same reason, we
argued in \cite[Remark~2.1]{Kreuzer.Zanotti:23+} that the parameter
$\kappa$ in the Biot's equations \eqref{E:BiotProblem-equations} could
be allowed to vary in space but not in time.  
\end{remark} 

\section{Discretisation with Lagrange elements in space}
\label{S:concrete-discretization}
In this section we propose an exemplary concrete space discretization, based on
$H^1$-confor\-ming Lagrange finite elements for all variables. Hence,
we first verify the assumptions
\eqref{A:extended-norms}-\eqref{A:inclusion} from
section~\ref{SS:abstract-discretization-space}, in order to infer the
well-posedness. Then, we discuss the a priori error analysis. 

Our notation and assumptions for the finite elements are as
follows. We denote by $\Mesh$ a face-to-face simplicial mesh of
$\Domain$. The shape constant of $\Mesh$ is given by   
\begin{equation}
\label{E:shape-constant}
\max_{\sfT \in \Mesh} \; \dfrac{\sfh_{\sfT}}{\mathsf r_{\sfT}}
\end{equation}
where $\sfh_\sfT$ is the diameter of a $\Dim$-simplex $\sfT \in \Mesh$
and $\mathsf r_\sfT$ is the diameter of the largest ball inscribed in
$\sfT$. 

We
denote by $\Faces$ the set of all faces of $\Mesh$ and by \(\Faces^i\) the interior faces. We
assume that the mesh is compatible with the boundary
conditions~\eqref{E:BiotProblem-boundary-conditions}. This means that
each face $\sfF \in \Faces \setminus \Faces^i$ (i.e. each boundary
face) satisfies 
\begin{align*}
\text{either}~\sfF\subseteq \Gamma_{u,N}~\text{or}~ \sfF\subseteq
\Gamma_{u,E}
\qquad\text{and}\qquad
\text{either}~\sfF\subseteq \Gamma_{p,N}~\text{or}~\sfF\subseteq \Gamma_{p,E}.
\end{align*}
Hence, we have two partitions of the boundary faces $\Faces_{{u,E}}
\cup \Faces_{{u,N}}$ and $\Faces_{p,E} \cup \Faces_{p,N}$, where each
set $\Faces_*$ is defined as  
\begin{equation}
\label{E:boundary-portions}
\Faces_* := \{ \sfF \in \Faces \mid \sfF \subseteq \Gamma_* \}.
\end{equation}

We let the meshsize $\sfh$ and the normal $\Normal$ be the piecewise
constant functions on $\Mesh$ and on the skeleton of $\Mesh$ (i.e. the
union of all faces) defined as  
\begin{equation*}
\label{meshsize+normal}
\sfh_{|\sfT} := \sfh_{\sfT} 
\qquad \text{and} \qquad
\Normal_{|\sfF} := \Normal_\sfF 
\end{equation*}
for all $\sfT \in \Mesh$ and $\sfF \in \Faces$, respectively. Here
$\Normal_\sfF$ is a fixed unit normal vector of 
$\sfF$, pointing outside $\Domain$ if $\sfF$ is a boundary face.

The space $\Poly{\sfk}(S)$, $\sfk \geq
0$, consist of all polynomials of total degree $\leq \sfk$ on an
$n$-simplex $S \subseteq \R^\Dim$ with $1\leq n \leq \Dim$. The
corresponding space of (possibly discontinuous) piecewise polynomials
on the mesh $\Mesh$ is denoted by $\Poly{\sfk}(\Mesh)$.

\subsection{Concrete space discretisation}
\label{SS:concrete-discretization-space}
A concrete realization of the abstract space discretization from
section~\ref{SS:abstract-discretization-space} requires a specific
choice of the spaces $\bbU_\sfs $ and $\bbP_\sfs $ and of the
operators $\calE_\sfs $ and $\calL_\sfs $ in
\eqref{E:abstract-spaces-discrete-displacement-pressure} and
\eqref{E:abstract-operators-discrete-elliptic}, respectively. We have
to prescribe also the operator $\calD_\sfs $ in
\eqref{E:abstract-operators-discrete-divergence} and to characterize
its range $\bbD_\sfs$ in
\eqref{E:abstract-spaces-discrete-total-pressure}. Finally, we need to
verify the assumptions \eqref{A:extended-norms}-\eqref{A:inclusion}. 

For the discretization of the displacement, we consider
$H^1$-conforming Lagrange finite elements of degree $\sfk + 1$ with
$\sfk \geq 1$, i.e., we set 
\begin{equation}
\label{E:concrete-spaces-displacement}
\bbU_\sfs :=
\Poly{\sfk+1}(\Mesh)^d \cap \bbU.
\end{equation}
The intersection with the space $\bbU$ from
\eqref{E:abstract-spaces-displacement} enforces the global continuity
(hence the  $H^1$-conformity) as well as the boundary conditions. As
usual in conforming finite elements, we discretize the operator
$\calE$ in \eqref{E:abstract-operators-concrete} by $\calE_\sfs : \bbU_\sfs  \to \bbU_\sfs ^*$ defined
as 
\begin{equation}
\label{E:concrete-operator-elasticity}
\langle\calE_\sfs \widetilde U
, V\rangle_{\bbU_\sfs }:=  \langle\calE\widetilde U
, V\rangle_{\bbU}
\end{equation}
for all $\widetilde U, V \in \bbU_\sfs $. 

Recall that the operator $\calD_\sfs $ should be a discretization of
the divergence and that the pair $\bbU_\sfs /\bbD_\sfs$ with
$\bbD_\sfs = \calD_\sfs (\bbU_\sfs )$ should enjoy a discrete Stokes
inf-sup condition in order to satisfy assumption
\eqref{A:Stokes-inf-sup}. Given the above definition of $\bbU_\sfs $,
one option for $\bbD_\sfs$ is suggested by the so-called Hood-Taylor pair,
cf. \cite[section~54.4]{Ern.Guermond:21b}. Hence, we consider
$H^1$-conforming Lagrange finite elements of degree $\sfk$ 
\begin{align}
\label{E:concrete-spaces-pressure-total}
\bbD_\sfs:=
\Poly{\sfk}(\Mesh) \cap H^1(\Domain) \cap \bbD
\end{align}
and we define $\calD_\sfs : \bbU_\sfs  \to \bbD_\sfs$ by
\begin{equation}
\label{E:concrete-operator-divergence}
( \calD_\sfs  \widetilde U, Q_\Tot)_\Domain = ( \calD \widetilde U, Q_\Tot )_\Domain
\end{equation}
for all $\widetilde U \in \bbU_\sfs $ and $Q_\Tot \in
\bbD_\sfs$. According to \eqref{E:abstract-spaces-pressure-total}, the
intersection with $\bbD$ in \eqref{E:concrete-spaces-pressure-total}
simply enforces the vanishing mean value when $\Gamma_{u,E} = \partial
\Domain$. Note also that $\calD_\sfs  \widetilde U$ is nothing else
than the $\bbH$-orthogonal projection of $\calD \widetilde U$ onto
$\bbD_\sfs$.  

Finally, the assumption \eqref{A:inclusion} suggests that also the
space $\bbP_\sfs $, for the discretization of the pressure and of the
total fluid content, should consist of $H^1$-conforming Lagrange
finite elements of degree $\sfk$. Therefore, we set 
\begin{equation}
\label{E:concrete-spaces-pressure}
\bbP_\sfs :=
\Poly{\sfk}(\Mesh) \cap \bbP.
\end{equation}
The intersection with the space $\bbP$ from
\eqref{E:abstract-spaces-pressure} enforces global continuity, the
boundary conditions and, possibly, the vanishing mean value. In
this case, we define $\calL_\sfs : \bbP_\sfs  \to \bbP_\sfs ^*$ in terms of  $\calL$ in \eqref{E:abstract-operators-concrete} by
\begin{equation}
\label{E:concrete-operator-laplacian}
\langle\calL_\sfs \widetilde P
, Q\rangle_{\bbP_\sfs }:=  \langle\calL\widetilde P
, Q\rangle_{\bbP}
\end{equation}
for all $\widetilde P, Q \in \bbP_\sfs $.

\begin{remark}[Hidden constants]
\label{R:hidden-constants}
In order to simplify the statement of the next results, it is
implicitly understood hereafter that all hidden constants in our
estimates potentially depend on the final time $T$, the shape constant
\eqref{E:shape-constant} of $\Mesh$ and the polynomial degree
$\sfk$. In particular, the latter is arbitrary but fixed in our
setting. The possible dependence on other relevant quantities is
addressed case by case.  
\end{remark}

Having introduced all the spaces and the operators required for the
space discretization in
section~\ref{SS:abstract-discretization-space}, we can verify the
validity of the assumptions
\eqref{A:extended-norms}-\eqref{A:inclusion}.

\begin{proposition}[Verification of the assumptions]
\label{P:verification-assumptions}
Let the spaces $\bbU_\sfs $, $\bbD_\sfs$ and $\bbP_\sfs $ and the
operators $\calE_\sfs $, $\calD_\sfs $ and $\calL_\sfs $ be defined by
\eqref{E:concrete-spaces-displacement}-\eqref{E:concrete-operator-laplacian}. 
\begin{enumerate}
	\item \label{P:va-i} The assumption~\eqref{A:extended-norms} holds true and
          the equivalences therein are actually identities. 
	\item  \label{P:va-ii}
          The assumption \eqref{A:Stokes-inf-sup} holds true and
          the constants therein depend only on the quantities
          mentioned in Remark~\ref{R:hidden-constants}, provided that
          each $\Dim$-simplex in $\Mesh$ has at least one vertex in
          the interior of $\Domain$. 
	\item \label{P:va-iii} The assumption \eqref{A:L2-projection}
          holds true and
          the constants therein depend only on the quantities
          mentioned in Remark~\ref{R:hidden-constants}, provided that the
          grading of  $\Mesh$, defined as in \cite{Diening.Storn.Tscherpel:23}, is strictly less than 
          \(  (\sqrt{2 \sfk+d}+\sqrt{\sfk})/(\sqrt{2 \sfk+d}-\sqrt{\sfk})\).
          % is quasi-uniform or obtained by adaptive bisection from a colored
          % initial mesh.
	%
	\item  \label{P:va-iv} The assumption \eqref{A:inclusion} holds true.
\end{enumerate}
\end{proposition}

\begin{proof}
\eqref{P:va-i}\; Owing to the inclusions $\bbU_\sfs  \subseteq \bbU$ and
$\bbP_\sfs  \subseteq \bbP$, there is no need to extend the norms
$\Norm{\cdot}_\bbU$ and $\Norm{\cdot}_\bbP$. Moreover, the combination
of \eqref{E:norms-discrete} with
\eqref{E:concrete-operator-elasticity} and
\eqref{E:concrete-operator-laplacian} readily implies
$\Norm{\cdot}_{\bbU_\sfs } = \Norm{\cdot}_\bbU$ in $\bbU_\sfs $ as
well as $\Norm{\cdot}_{\bbP_\sfs } = \Norm{\cdot}_\bbP$ in $\bbP_\sfs
$. 

\eqref{P:va-ii}\; The upper bound in \eqref{A:Stokes-inf-sup} follows
from the boundedness of $\calD_\sfs $, which satisfies
$\Norm{\calD_\sfs  \cdot}_\Domain^2 \leq \Norm{\calD \cdot}_\Domain^2
\lesssim \mu^{-1} \Norm{\cdot}_\bbU^2$
in $\bbU_\sfs $, according to
\eqref{E:abstract-operators-concrete}, \eqref{E:norms} and
\eqref{E:concrete-operator-divergence}. 
The lower bound can be
rephrased as 
\begin{equation*}
c \Norm{\cdot}_\Domain^2 \leq \mu \left( \sup_{V \in \bbU_\sfs }
  \dfrac{( \calD V, \cdot)_\Domain}{\Norm{V}_\bbU}\right)^2  
\end{equation*}
in $\bbD_\sfs$. The definitions \eqref{E:abstract-operators-concrete}
and \eqref{E:norms} reveal that this is equivalent to the discrete
Stokes inf-sup stability of the pair $\bbU_\sfs /\bbD_\sfs$, i.e. of
the Hood-Taylor pair. The latter condition is known to hold true under
the above assumption on the mesh; see \cite[sections~54.3 and
54.4]{Ern.Guermond:21b}. 

\eqref{P:va-iii}\; We have
\begin{equation*}
\Norm{\calP_{\bbP_\sfs }^*\cdot}_{\bbP^*} 
= 
\sup_{w \in \bbP}
\dfrac{\left\langle \cdot, \calP_{\bbP_\sfs } w\right\rangle_{\bbP_\sfs } }{\Norm{w}_\bbP}
=
\sup_{w \in \bbP}
\dfrac{( \cdot, \calP_{\bbP_\sfs } w)_\Domain }{\Norm{w}_\bbP}
\end{equation*}
in $\bbP_\sfs ^*$. The inclusion $\bbP_\sfs  \subseteq \bbP$ implies
the lower bound in \eqref{A:L2-projection} with $c =1$, because
$\calP_{\bbP_\sfs }$ is a projection onto $\bbP_\sfs $. The upper
bound is equivalent to the $\bbP$-stability of $\calP_{\bbP_\sfs
}$; cf.~\cite{Tantardini.Veeser:16}.
Owing to \eqref{E:norms} and \eqref{E:concrete-spaces-pressure},
this is the $H^1(\Domain)$-stability of the $L^2(\Domain)$-orthogonal
projection onto $H^1$-conforming Lagrange finite element spaces. Such
a condition is known to hold under the asserted grading assumption
thanks to~\cite[Theorem~4.14(ii)]{Diening.Storn.Tscherpel:21} together
with \cite[Theorem~4.4]{Diening.Storn.Tscherpel:21} and \cite[Remark
4.4]{Diening.Storn.Tscherpel:21}. 

\eqref{P:va-iv}\; Let $\sigma = 0$. The combination of
\eqref{E:concrete-spaces-pressure-total} and
\eqref{E:concrete-spaces-pressure}  with
\eqref{E:abstract-spaces-pressure} and
\eqref{E:abstract-spaces-pressure-total} implies, for $\Gamma_{u,E} =
\partial \Domain$ that
\begin{equation*}
\bbP_\sfs  = 
\Poly{\sfk}(\Mesh) \cap H^1_{\Gamma_{p,E}}(\Domain) \cap L^2_0(\Domain)
\quad \text{and} \quad
\bbD_\sfs = \Poly{\sfk}(\Mesh) \cap H^1(\Domain) \cap L^2_0(\Domain).
\end{equation*}
For $\Gamma_{u,E} \neq \partial \Domain$ we have instead,
\begin{equation*}
\bbP_\sfs  \subseteq
\Poly{\sfk}(\Mesh) \cap H^1_{\Gamma_{p,E}}(\Domain)
\quad \text{and} \quad
\bbD_\sfs = \Poly{\sfk}(\Mesh) \cap H^1(\Domain).
\end{equation*}
In both cases we have $\bbP_\sfs  \subseteq \bbD_\sfs$, i.e.,
assumption \eqref{A:inclusion} is verified. 
\end{proof}

\begin{remark}[Assumptions for \eqref{A:L2-projection}]\label{R:L2-projection}
  \label{R:assumption-L2-projection}
  We refer to \cite[Definitions 1.1]{Diening.Storn.Tscherpel:23} for
  the exact definition of mesh grading. Clearly the grading of quasi-uniform 
  meshes equals 1 and thus satisfies the grading
  condition in
  Proposition~\ref{P:verification-assumptions}\eqref{P:va-iii}. It
  follows from 
  \cite[Theorem~1.3]{Diening.Storn.Tscherpel:23}, that the grading
  condition is also
  satisfied for all polynomial degrees \(\sfk\ge1\) and
  dimensions \(d\le 6\) provided \(\Mesh\) is obtained by adaptive
  bisection of a colored initial mesh; see
  \cite[Assumption~3.1]{Diening.Storn.Tscherpel:23} for the notion of
  colored mesh. Note that assumption~\eqref{A:L2-projection} can be
  verified under different assumptions. Indeed, the
  $H^1(\Domain)$-stability of the $L^2(\Domain)$-orthogonal projection
  onto finite element spaces has been extensively analyzed by various
  authors; we refer~\cite{Diening.Storn.Tscherpel:21} for an overview of
  the existing results. 
\end{remark}

Having verified all assumptions in
section~\ref{SS:abstract-discretization-space}, we deduce the
well-posedness of the discretization
\eqref{E:BiotProblem-discretization} with the spaces and the operators
proposed in this section. In particular, the inclusions $\bbU_\sfs
\subseteq \bbU$ and $\bbP_\sfs  \subseteq \bbP$ suggest to define the
data in the discretization by restriction of the data in the weak
formulation \eqref{E:BiotProblem-weak-formulation} 
\begin{equation}
\label{E:data-discrete}
L_u = \ell_{u|\bbS^0_\sft (\bbU_\sfs )},
\qquad 
L_p = \ell_{p|\bbS^0_\sft (\bbP_\sfs )}
\qquad \text{and}\qquad
L_0 = \ell_{0|\bbP_\sfs }.
\end{equation}

\begin{theorem}[Well-posedness with Lagrange elements]
\label{T:well-posedness-Lagrange}
Let the spaces $\bbU_\sfs $, $\bbD_\sfs$ and $\bbP_\sfs $ and the
operators $\calE_\sfs $, $\calD_\sfs $ and $\calL_\sfs $ be defined by
\eqref{E:concrete-spaces-displacement}-\eqref{E:concrete-operator-laplacian}. Assume
that $\Mesh$ is obtained from a colored initial mesh by newest vertex
bisection and that each $\Dim$-simplex in $\Mesh$ has at least one
vertex in the interior of $\Domain$. Then, the
discretization~\eqref{E:BiotProblem-discretization} with the data
\eqref{E:data-discrete} has a unique
solution \label{E:well-posedness-solution} $Y_1 = (U, P_\Tot, P, M,
M_0) \in \bbY_{1, \sfs\sft}$ with 
\begin{equation}
\label{E:well-posedness-Lagrange}
\begin{split}
\Norm{Y_1}_{1,\sfs\sft}^2
&\eqsim
\Norm{Y_1}_{1,\sfs\sft}^2 + 
\Norm{M}^2_{L^\infty(\bbP_\sfs ^*)}
+\int_0^T\left( \lambda \Norm{\calD_\sfs  U}^2_{\Domain}  
+ \gamma^{-1} \Norm{P}^2_{\Domain} \right)\dt \\
& \eqsim \int_{0}^T \left(
\Norm{L_u}^2_{\bbU^*_\sfs } +
\Norm{L_p}^2_{\bbP^*_\sfs } \right)\dt +
\Norm{L_0}^2_{\bbP^*_\sfs }\\
& \lesssim \int_{0}^T \left(
\Norm{\ell_u}^2_{\bbU^*} +
\Norm{\mathcal{\ell}_p}^2_{\bbP^*} \right)\dt +
\Norm{\ell_0}^2_{\bbP^*}.
\end{split}
\end{equation}
All hidden constants depend only on the quantities mentioned in Remark~\ref{R:hidden-constants}.
\end{theorem}

\begin{proof}
The combination of Theorem~\ref{T:well-posedness-discretization} with
Proposition~\ref{P:verification-assumptions} implies the existence and
the uniqueness of the solution and guarantees that the two
equivalences in \eqref{E:well-posedness-Lagrange} hold true. Then, the
upper bound follows from the discretization \eqref{E:data-discrete} of
the data and the inclusions $\bbU_\sfs  \subseteq \bbU$ and $\bbP_\sfs
\subseteq \bbP$. 
\end{proof}

\subsection{Quasi-optimality}
\label{SS:quasi-optimality}
In order to quantify the accuracy of the proposed discretization, we
need an error notion that is related to both the trial norm
$\Norm{\cdot}_1$ in \eqref{E:trial-norm} and to the discrete trial
norm $\Norm{\cdot}_{1,\sfs\sft}$ in \eqref{E:trial-norm-discrete},
cf. \eqref{E:error-notion-restrictions} below. A major issue is that
the former one involves the dual norm in $\bbP$, whereas the latter
one involves the dual norm in $\bbP_\sfs $.  

In order to compare functionals defined on the two spaces, we can use
the adjoint $\calP^*:\bbP_\sfs ^* \to \bbP^*$ of a bounded projection
$\calP: \bbP \to \bbP_\sfs $. The specific choice of $\calP$ is
critical because of the term $\partial_t \widetilde m + \calL
\widetilde p $ in the norm $\Norm{\cdot}_1$. Since this term acts in
space via the pairing $\left\langle \cdot, \cdot \right\rangle_{\bbP}
$, one may want to employ the $L^2(\Domain)$-orthogonal projection,
because the pivot space in the Hilbert triplet
\eqref{E:Hilbert-triplet} is equipped with the
$L^2(\Domain)$-norm. Still, according to \eqref{E:norms}, the action
of $\calL$ induces the $\bbP$-norm, thus suggesting to make use of the
$\bbP$-orthogonal projection. Of course, similar comments apply to the
counterpart of $\partial_t \widetilde m + \calL \widetilde p$ in
$\Norm{\cdot}_{1,\sfs\sft}$. 

Both options have advantages and disadvantages. The latter one, being
related to the $\bbP$-norm, suggests the use of piecewise polynomials
of degree $\sfk+1$ for the discretization of $\bbP$, but this would be
critical for the assumption~\eqref{A:inclusion}, as $\bbD_\sfs$
consists of piecewise polynomials of degree $\sfk$,
cf. \eqref{E:concrete-spaces-pressure-total}-\eqref{E:concrete-spaces-pressure}. The
former option does not have this issue, therefore we go for it. Still,
the derivation of higher-order decay rates in space appears
critical in this case, cf. Remark~\ref{R:decay-rate}. 

In accordance with the above discussion and with the inclusions
$\bbU_\sfs  \subseteq \bbU$ and $\bbP_\sfs  \subseteq \bbP$, we
consider the error notion $\texttt{Err}:\overline \bbY_1\times\bbY_{1,
  \sfs\sft}\to [0, +\infty)$ defined as  
\begin{align}
\label{E:error-notion}
\begin{aligned}
&\Err{\widetilde y_1}{\widetilde Y_1}^2
:= \int_{0}^{T} \left( \Norm{u- \widetilde U}^2_\bbU +
\dfrac{1}{\mu}\Norm{p_\Tot- \widetilde P_\Tot}_\Domain^2 \right) \dt\\ 
&\quad+ \int_0^T
\Norm{\partial_t \widetilde m+\calL \widetilde p-\calP_{\bbP_\sfs }^*(\mathrm{d}_\sft (\widetilde M, \widetilde M_0) + \calL_\sfs  \widetilde
	P)}^2_{\bbP^*} \dt 
\\
&\quad + \Norm{ \widetilde m(0)-\calP_{\bbP_\sfs }^*\widetilde M_0}^2_{\bbP^*}
\\
&\quad+ \int_{0}^T  \dfrac{1 }{\lambda+\mu} \Norm{\lambda \calD \widetilde u - \widetilde p_\Tot - \alpha \calP_\bbD \widetilde p - (\lambda
	\calD_\sfs  \widetilde U - \widetilde P_\Tot - \alpha \calP_{\bbD_\sfs }
	\widetilde P)}^2_\Domain \dt\\
&\quad+ \int_0^T \gamma\Norm{\alpha \calP_{\overline \bbP} \calD
	\widetilde u + \sigma \widetilde p - \widetilde m - (\alpha \calP_{\bbP_\sfs } \calD_\sfs 
	\widetilde U + \sigma \widetilde P - \widetilde M)}^2_\Domain \dt
\end{aligned}
\end{align}
for \(\widetilde y_1=(\widetilde
u,\widetilde p_\Tot,\widetilde p,\widetilde m)\in \overline \bbY_1\) and \(\widetilde
Y_1=(\widetilde U,\widetilde P_\Tot,\widetilde P,\widetilde M,\widetilde M_0)\in \bbY_{1,\sfs\sft}\). Notice that this is not a norm on the sum $\overline \bbY_1 + \bbY_{1, \sfs\sft}$. For instance, in general, we have $\calD \widetilde U \neq \calD_\sfs  \widetilde U$ for $\widetilde U \in \bbS^0_\sft (\bbU_\sfs ) \subseteq L^2(\bbU)$. Still, the above error notion measures the accuracy in the approximation of all functions and functionals involved in the trial norm $\Norm{\cdot}_1$ and it holds that
\begin{equation}
\label{E:error-notion-restrictions}
\Err{\widetilde y_1}{0} = \Norm{\widetilde y_1}_{1}
\qquad \text{and} \qquad
\Err{0}{\widetilde Y_1} \eqsim \Norm{\widetilde Y_1}_{1,\sfs\sft}.
\end{equation}
Indeed, the second equivalence follows from~\eqref{P:va-i} and~\eqref{P:va-iii} in Proposition~\ref{P:verification-assumptions}.

The equations \eqref{E:BiotProblem-discretization} with the spaces and
the operators defined by
\eqref{E:concrete-spaces-displacement}-\eqref{E:concrete-operator-laplacian}
are not a conforming Petrov-Galerkin discretization of the weak
formulation \eqref{E:BiotProblem-weak-formulation}. Indeed, the trial
space $\bbY_{1, \sfs\sft}$ in the former problem is not a subspace of
its counterpart $\overline \bbY_1$ in the latter one. Nevertheless, we
have for the test spaces the inclusion 
\begin{align*}
 \bbY_{2,\sfs\sft}\subseteq\bbY_2.
\end{align*}
This property and the definition \eqref{E:data-discrete} of the load
in \eqref{E:BiotProblem-discretization} by restriction of the one in
\eqref{E:BiotProblem-weak-formulation} ensure that we can still
guarantee the fundamental \emph{quasi-optimality} property of inf-sup stable
conforming Petrov-Galerkin methods by a standard 
argument; cf. \cite{Babuska:70}.  

\begin{theorem}[Quasi-optimality]
\label{T:quasi-optimality}
Let all assumptions in Theorem~\ref{T:well-posedness-Lagrange} be verified. Denote by \(y_1\in\overline{\bbY}_1\) and $Y_1 \in \bbY_{1, \sfs\sft}$, respectively, the solutions of \eqref{E:BiotProblem-weak-formulation} and \eqref{E:BiotProblem-discretization} with \eqref{E:data-discrete}. Then, we have
\begin{align}
\label{E:quasi-optimality}
  \Err{y_1}{Y_1}\lesssim \inf_{\widetilde Y_1\in\bbY_{1,\sfs\sft}}\Err{y_1}{\widetilde
  Y_1}.
\end{align}
The hidden constant depends only on the quantities mentioned in Remark~\ref{R:hidden-constants}.
\end{theorem}

\begin{proof}
For $\widetilde Y_1 \in \bbY_{1, \sfs\sft}$, the triangle inequality and \eqref{E:error-notion-restrictions} imply
\begin{align*}
\Err{y_1}{Y_1}
\leq
\Err{y_1}{\widetilde Y_1} + \Err{0}{Y_1 - \widetilde Y_1}
\lesssim 
\Err{y_1}{\widetilde Y_1}+\Norm{Y_1-\widetilde Y_1}_{1,\sfs\sft}.
\end{align*}
According to the inf-sup stability in Lemma~\ref{L:inf-sup}, we have
\begin{align*}
\Norm{Y_1-\widetilde Y_1}_{1,\sfs\sft}
\lesssim
\sup_{Y_2\in \bbY_{2,h}}
\frac{b_{\sfs\sft}(Y_1-\widetilde Y_1,Y_2)}{\|Y_2\|_2}
= 
\sup_{Y_2\in \bbY_{2}^h}
\frac{b(y_1,Y_2)-b_{\sfs\sft}(\widetilde Y_1,Y_2)}{\|Y_2\|_2},
\end{align*}
where the identity follows by comparing  problems
\eqref{E:BiotProblem-weak-formulation} and
\eqref{E:BiotProblem-discretization}. We exploit the definitions
\eqref{E:BiotProblem-abstract-form},
\eqref{E:BiotProblem-discrete-form}, and
\eqref{E:concrete-operator-elasticity}-\eqref{E:concrete-operator-laplacian},
as well as the invariance of the projection $\calP_{\bbP_\sfs }$ onto
$\bbP_\sfs$. For $y_1 = (u, p_\Tot, p, m)$ and $\widetilde Y_1 =
(\widetilde U, \widetilde P_\Tot, \widetilde P, \widetilde M,
\widetilde M_0)$, this results in  
\begin{equation}
\label{E:quasi-optimality-proof}
\begin{split}
&b(y_1, Y_2) - b_{\sfs\sft}(\widetilde Y_1, Y_2) :=\\
&\qquad \phantom{+} 
\int_0^T \Big( \big\langle \calE (u - \widetilde U), V\big\rangle_\bbU 
+ 
\big( p_\Tot - \widetilde P_\Tot,\calD V \big)_{\Domain} \Big)\\
&\qquad +
\int_{0}^T \big\langle \partial_t m + \calL \widetilde P - \calP_{\bbP_\sfs }^*( \mathrm{d}_\sft  (\widetilde M, \widetilde M_0) + \calL_\sfs  \widetilde P), N \big\rangle_{\bbP}\dt\\
& \qquad +
\big\langle m - \calP_{\bbP_\sfs }^*\widetilde M_0, N_0 \big\rangle_{\bbP}
\\
& \qquad 
+ \int_0^T  \big(\lambda \calD u - p_\Tot - \alpha\calP_{\bbD} p - (
\lambda \calD_\sfs  \widetilde U - \widetilde 
P_\Tot - \alpha \calP_{\bbD_\sfs }\widetilde P),
Q_\Tot\big)_{\Domain}
\\ 
& \qquad + \int_0^T
\big(\alpha \calP_{\overline \bbP} \calD u + \sigma p - m - ( \alpha
\calP_{\bbP_\sfs }\calD_\sfs  \widetilde U + \sigma \widetilde P -
\widetilde M), Q 
\big)_{\Domain} \dt. 
\end{split} 
\end{equation}
Then Cauchy-Schwarz inequalities, the bound $\Norm{\calD V}_{\Domain}^2 \lesssim \mu^{-1} \Norm{V}^2_\bbU$ and the definition \eqref{E:test-norm} of the test norm $\Norm{\cdot}_2$ yield
\begin{equation*}
\Norm{Y_1 - \widetilde Y_1}_{1,\sfs\sft}
\lesssim \Err{y_1}{\widetilde Y_1}.
\end{equation*}
Inserting this estimate into the first inequality above concludes the proof.
\end{proof}

\begin{remark}[Augmented error notion]
\label{R:augmented-error-notion}
Our definition of the error notion is in a sense minimal, because we have included only the terms that arise by applying the Cauchy-Schwarz inequality in \eqref{E:quasi-optimality-proof}. Still, the first equivalence in \eqref{E:well-posedness-Lagrange} clarifies that that the statement and the proof of Theorem~\ref{T:quasi-optimality} will remain unchanged if we augment $\Err{\cdot}{\cdot}$ by adding any of the following terms
\begin{equation*}
\Norm{\widetilde m - \calP_{\bbP_\sfs }^*\widetilde M}^2_{L^\infty(\bbP^*)},
\qquad
\int_0^T \lambda \Norm{\calD \widetilde u - \calD_\sfs  \widetilde U}^2_{\Domain}\dt,
\qquad\text{and}\qquad
\int_0^T \gamma^{-1} \Norm{\widetilde p - \widetilde P}^2_{\Domain} \dt .
\end{equation*} 
\end{remark}

\subsection{A priori error estimates}
\label{sec:error-analysis}
The quasi-optimality stated above has two remarkable properties and
(at least) one clear disadvantage. One the one hand, the estimate
\eqref{E:quasi-optimality} holds true for any solution of the weak
formulation \eqref{E:BiotProblem-weak-formulation} (i.e. no additional
regularity is required) and the hidden constant is robust with respect
to all material parameters. On the other hand, it is not immediate how (and even if) the error decays to zero, because the
definition \eqref{E:error-notion} of $\Err{\cdot}{\cdot}$ involves a
nontrivial coupling of the various components of the solution.   

In this section, we investigate the latter aspect. Roughly speaking,
we aim at showing that the best error in the right-hand side of
\eqref{E:quasi-optimality} is equivalent to a sum of best errors. In
order to preserve the two above-mentioned properties, we make sure
also that the constants in the equivalence are independent of the
material parameters and that no additional regularity of the solution
of \eqref{E:BiotProblem-weak-formulation} is required beyond the
minimal one, namely $y_1 \in \overline{\bbY}_1$. When such a result is
available, the decay of the error to zero can be easily discussed in
terms of classical results from approximation theory.  

The precise statement of our main result is as follows. Notice that
the last term on the right-hand side of \eqref{E:decoupling} does not
appear in our definition of the error notion, but we could
equivalently include it  according to Remark~\ref{R:augmented-error-notion}.

\begin{theorem}[Best error decoupling]
\label{T:decoupling}
Let all assumptions in Theorem~\ref{T:well-posedness-Lagrange} be
verified and denote by $y_1 = (u, p_\Tot, p, m) \in \overline{\bbY}_1$
the solution of \eqref{E:BiotProblem-weak-formulation}. Then we have that 
\begin{equation}
\label{E:decoupling}
\begin{split}
&
\inf_{\widetilde Y_1 \in \bbY_{1, \sfs\sft}} \Err{y_1}{\widetilde Y_1}^2
\lesssim \\
&\;\phantom{+}
\inf_{\widehat U\in\bbS^0_\sft (\bbU_\sfs )}\int_0^T\|u-\widehat U\|_\bbU^2\dt
+
\inf_{\widehat P_\Tot\in\bbS^0_\sft (\bbD)}\int_0^T\frac1\mu\|p_\Tot-\widehat{P}_\Tot\|_\Domain^2\dt\\
&\;+
\inf_{\widehat W\in\bbS^0_\sft (\bbP_\sfs )}\int_0^T\|\partial_t m+\calL p-\calP_{\bbP_\sfs }^*\widehat W\|_{\bbP^*}^2\dt
+
\inf_{\widehat M_0\in\bbP_\sfs}\|m(0)-\calP_{\bbP_\sfs}^*\widehat M_0\|_{\bbP^*}^2 
\\
&\;
+
\varepsilon_\sfs^{-1}
\inf_{\widehat P\in\bbS^0_\sft (\bbP_\sfs )}\int_0^T \dfrac{1}{\gamma}\|p-\widehat P\|^2_\Domain\dt
.
\end{split}
\end{equation}
The hidden constant depends on the quantities mentioned in Remark~\ref{R:hidden-constants} and on the constant in \eqref{E:interpolation-time-shape-constant} and $\varepsilon_\sfs$ is defined in \eqref{E:discrete-elliptic-regularity} below.
\end{theorem}

\begin{proof}
See sections~\ref{SS:interpolation-time}-\ref{SS:interpolation-space-time}.
\end{proof}

A first consequence of Theorem~\ref{T:decoupling} is plain convergence: the error converges to zero as the mesh-size converges to zero, both in space and time. This holds true irrespective of the regularity of the solution. Moreover, first-order convergence can be established under additional regularity assumptions.

\begin{corollary}[First-order convergence]
\label{C:first-order-convergence}
Let all assumptions in Theorem~\ref{T:well-posedness-Lagrange} be verified. Denote by \(y_1 = (u, p_\Tot, p, m) \in\overline{\bbY}_1\) and $Y_1 \in \bbY_{1, \sfs\sft}$, respectively, the solutions of \eqref{E:BiotProblem-weak-formulation} and \eqref{E:BiotProblem-discretization} with \eqref{E:data-discrete}. Assume additionally
\begin{equation}
\label{E:first-order-convergence-assumptions}
\begin{array}{ll}
u\in H^1(\bbU)\cap L^2(H^2(\Domain)^\Dim) \qquad
&
p_\Tot \in H^1(\bbD) \cap L^2(H^1(\Domain))
\\[2pt]
p\in H^1(\overline \bbP) \cap L^2(H^1(\Domain))
&
(\partial_tm+\calL p)\in H^1(\bbP^*)\cap L^2(L^2(\Domain))
\\[2pt]
m(0)\in L^2(\Domain). &
	\end{array}
\end{equation}
	Then, the error can be bounded from above as follows
	\begin{equation*}
	\label{E:first-order-convergence-estimate}
	\begin{split}
	&\Err{y_1}{Y_1}^2
	\lesssim
	\Big(\max \limits_{\Domain} \sfh \Big )^2 \Bigg \lbrace \dfrac{1}{\kappa} \|m(0)\|_{\Domain}^2 \\
	& \hspace{5pt} +
	\int_0^T  \left( 
	\mu \| \nabla^2 u\|_{\Domain}^2 + 
	\dfrac{1}{\mu} \|\nabla p_\Tot\|_{\Domain}^2  + 
	\dfrac{1}{\varepsilon_\sfs\gamma} \|\nabla p\|_{\Domain}^2  + 
	\dfrac{1}{\kappa} \|\partial_t m+\calL p\|_{\Domain}^2 \right)\dt \Bigg \rbrace\\
	& \hspace{5pt} + \Big ( \max_{j = 1,\dots, J} |I_j| \Big)^2
	\int_0^T \left( 
	\|\partial_t  u\|_{\bbU}^2 + 
	\dfrac{1}{\mu} \|\partial_t p_\Tot\|^2_\Domain +
	\dfrac{1}{\varepsilon_\sfs\gamma} \|\partial_t p\|^2_{\Domain} +
	 \|\partial_t(\partial_t m+\calL p)\|_{\bbP^*}^2
	\right) \dt.
	\end{split}
	\end{equation*}
	The hidden constant depends on the quantities mentioned in Remark~\ref{R:hidden-constants} and on the constant in \eqref{E:interpolation-time-shape-constant} and $\varepsilon_\sfs$ is defined in \eqref{E:discrete-elliptic-regularity} below. 
\end{corollary}

\begin{proof}
	Combine Theorem~\ref{T:quasi-optimality} with Theorem~\ref{T:decoupling}. Then, use standard
	Bramble-Hilbert-like estimates (see, e.g., \cite[Chapter
	7]{Tantardini:14}) in order to bound each term on the right-hand side of \eqref{E:decoupling}.
\end{proof}

\begin{remark}[Regularity in space]
\label{R:regularity}
According to \cite[Theorem~5.4]{Kreuzer.Zanotti:23+}, the
solution $y_1$ of \eqref{E:BiotProblem-weak-formulation} satisfies the
space regularity in~\eqref{E:first-order-convergence-assumptions} at least in some circumstances. To
be more precise, we have
\begin{equation*}
\begin{array}{ll}
  u\in L^2(H^2(\Domain)^2)
  \qquad 
  &p_\Tot\in L^2(H^1(\Domain))\\
  p\in L^2(H^1(\Domain))
  &(\partial_tm+\calL p)\in
  L^2(L^2(\Domain))
\end{array}
\end{equation*}
when the data are such that
\begin{equation*}
\label{E:regularity}
\ell_u \in L^2(L^2(\Domain)^2),
\qquad
\ell_p \in L^2(L^2(\Domain)),
\qquad
\ell_0 \in L^2(\Domain),
\end{equation*}
upon additionally assuming that $\Domain \subseteq \R^2$ is convex, the boundary conditions \eqref{E:BiotProblem-boundary-conditions} are posed on
\(\Gamma_{u,E} = \partial\Domain = \Gamma_{p,N}\), and $\lambda \gtrsim \mu$.
\end{remark}

\begin{remark}[Regularity in time]\label{R:regularity-shift-time}
According to~\cite{Murad.Thomee.Loula:96}, the time regularity
in~\eqref{E:first-order-convergence-assumptions} may fail to hold even
for smooth data. For a rough explanation, assume the data are smooth
in time, the solution enjoys the time regularity
in~\eqref{E:first-order-convergence-assumptions} and, in addition, we
have $m \in C^1(\bbP^*)$, i.e. we have one more time derivative than
in Theorem~\ref{T:well-posedness-weak-formulation}. Then, on the
one-hand, the fourth equation in
\eqref{E:BiotProblem-weak-formulation} implies $p \in C^0(\bbP)$. On
the other hand, the evaluation of the other equations at $t = 0$
implies that the pair $(u(0), p(0))$ solves the Stokes-like problem  
\begin{equation*}
\label{E:regularity-shift-time}
\begin{alignedat}{2}
(\calE + \lambda \calD^*\calD) u(0)
+\alpha \calD^* \calP_{\bbD}p(0) &= f_u(0) \quad && \text{in $\bbU^*$}
\\
\alpha\calP_{\overline \bbP}\Div u(0)+\sigma p(0) &= \ell_0 && \text{in $\overline \bbP$}.
\end{alignedat}
\end{equation*}
In general, the solution of this problem is in $\bbU \times \overline \bbP$. Hence the condition $p(0) \in \bbP$ can hold true only for compatible data, otherwise some singularity must be expected at the initial time. We refer to the Terzaghi test case discussed in section \ref{S:numerics} below for a 
numerical illustration.
\end{remark}

\begin{remark}[Grading of the time partition]\label{R:LagrangeIt}
The constant in Theorem~\ref{T:decoupling} depends on the one in
\eqref{E:interpolation-time-shape-constant}, which measures the
grading of the partition employed for the discretization in time. The
latter constants enters into play because Theorem~\ref{T:decoupling}
does not assume additional regularity in time of the solution. This
calls for a Scott-Zhang-like interpolation in time, see
section~\ref{SS:interpolation-time} below for the details. When all components of the solution are continuous in time, a Lagrange-like interpolation is possible and the constant in
\eqref{E:interpolation-time-shape-constant} does not enter into the
error estimation, cf. \cite[Theorem 
4.5]{Tantardini:14}. Still, according to
Remark~\ref{R:regularity-shift-time}, the continuity at $t=0$ is not
obvious.  
\end{remark}

\begin{remark}[Decay rate]
\label{R:decay-rate}
The space discretization considered in this section is actually of
higher-order, because we make use of $H^1$-conforming Lagrange finite
elements of degree $\sfk+1$ for the displacement and of degree $\sfk$
for the other components of the solution with $\sfk \geq 1$,
cf. \eqref{E:concrete-spaces-displacement}-\eqref{E:concrete-spaces-pressure}. Therefore,
the question arises if a higher-order decay rate with respect to
$\sfh$ can be obtained under higher regularity assumptions on the
solution. On the one hand,
the $\bbP^*$-norm errors on the right-hand side of \eqref{E:decoupling} appear to be critical in this respect, because the space $\bbP_\sfs $, possibly incorporating boundary conditions, is
used to approximate functional from $\bbP^*$, which have no prescribed
boundary conditions. On the other hand, it is unclear to us if the above-mentioned regularity of the solution can be expected in general. Indeed, the regularity result mentioned in Remark 4.9,
assumes $\Gamma_{p,E} = \emptyset$. Due to the subtlety of this matter, we postpone any further investigation on this point to future work.
\end{remark}

The remaining part of this section is devoted to the proof of
Theorem~\ref{T:decoupling}. For this purpose, we aim at constructing a
bounded interpolant $\calI: \overline \bbY_{1} \to \bbY_{1,
  \sfs\sft}$. The operator $\calI$ cannot be obtained by just
approximating each component of the solution irrespective of the
others, because of the nontrivial coupling of the components in 
the error notion~\eqref{E:error-notion}. Thus, to make sure that $\calI$
is bounded and robust with respect to the material parameters, we must
guarantee that it is compatible with the coupling. Roughly speaking
this means that, when the error notion involves some combination of
the components of the solution, it must be possible to accurately
approximate it by the corresponding combination of the components of
the interpolant. We achieve this goal with the help of a number of
commutative diagrams, cf. Figures~\ref{F:Time-Operator-Diagram} and
\ref{F:Space-Operator-Diagram}. 

We divide our construction into three parts. In
Section~\ref{SS:interpolation-time} we first discuss the interpolation
in time. In Section \ref{SS:interpolation-space} we discuss the
interpolation in space. Finally, in
Section~\ref{SS:interpolation-space-time}, we combine the
interpolation in time and space in order to prove
Theorem~\ref{T:decoupling}. 

\subsection{Time interpolation}
\label{SS:interpolation-time}
The error notion $\Err{\cdot}{\cdot}$ involves several terms with
$L^2$ regularity in time, so we initially consider the problem
of approximating function from $L^2(\bbX)$ into $\bbS^0_\sft (\bbX)$
for some Hilber space $\bbX$. Since we cannot control the point values
of a function in $L^2(\bbX)$, we cannot use Lagrange
interpolation; cf. remark~\ref{R:LagrangeIt}. Thus, we resort to Scott-Zhang-like
\cite{Scott.Zhang:90} interpolation in the vein of \cite[Section~4]{Tantardini:14}. 

Recall \eqref{E:mesh-time}-\eqref{E:mesh-time-intervals}. For
\(j\in\{1,\ldots,J\}\), the polynomial $\psi_j \in \Poly{1}(I_j)$
defined as 
\begin{equation}
\label{E:definition-psi}
\psi_j(t):=6\frac{t-t_{j-1}}{|I_j|^2}-\frac{2}{|I_j|}
\end{equation}
is such that, 
\begin{equation}\label{eq:property_psi}
\int_{I_j} Q\psi_j  \dt=Q(t_j) \qquad \forall Q \in \Poly{1}(I_j).
\end{equation}
We define $\calJ : L^2(\bbX)\to\bbS^0_\sft (\bbX)$ by 
\begin{align*}
  (\calJ  x)_{|I_j}:=\int_{I_j}x \psi_j\dt
\end{align*}
for $x \in L^2(\bbX)$ and $j=1,\dots J$. 

Since the error notion in \eqref{E:error-notion} involves the discrete time derivative $\mathrm{d}_\sft $, it is useful to introduce another interpolant $\widetilde \calJ : L^2(\bbX) \to \bbS^0_\sft (\bbX)$, defined as  
\begin{equation*}
(\widetilde{\calJ}  x)_{|I_j}
:= 
\dfrac{ \displaystyle \int_{I_{j}}\left( \int_{t_{j-1}}^t x(s)\,\textrm{d}s\right) \psi_{j}(t)\,\textrm{d}t
+
\int_{I_{j-1}}\left( \int_{t}^{t_{j-1}}x(s) \,\textrm{d}s\right) \psi_{j-1}(t)\,\textrm{d}t }{|I_{j}|}
\end{equation*}
for $j=2,\dots J$ and
\begin{equation*}
(\widetilde{\calJ}  x)_{|I_1}:= \dfrac{\displaystyle \int_{I_1}\left( \int_{0}^tx(s)\textrm{d}s\right) \psi_{1}(t)\dt}{|I_1|}.
\end{equation*}
The first part of Lemma~\ref{L:interpolation-time} below ensures that both $\calJ x$ and $\widetilde{\calJ}x$ are near best approximations of $x$ in the $L^2(\bbX)$-norm. The second part additionally clarifies that the two interpolants are related, through the weak and the discrete time derivative, as summarized by the commuting diagram in Figure~\ref{F:Time-Operator-Diagram}.  

\begin{lemma}[Time interpolation]
\label{L:interpolation-time}
The operators $\calJ $ and $\widetilde{\calJ} $ defined above are such that 
	\begin{subequations}
		\label{E:interpolation-time}
		\begin{align}
		\label{E:interpolation-time-1st-interpolant}
		\int_0^T \Norm{x - \calJ  x}_{\bbX}^2\dt 
		&\leq 4
		\inf_{X \in \bbS^0_\sft(\bbX) }\int_0^T\Norm{x-X}_{\bbX}^2 \dt
		\\
		\label{E:interpolation-time-2nd-interpolant}
		\int_0^T \Norm{x - \widetilde{\calJ}  x}_{\bbX}^2\dt &\lesssim
		\inf_{X \in \bbS^0_\sft(\bbX) }\int_0^T\Norm{x-X}_{\bbX}^2\dt 
		\end{align}
		for all $x \in L^2(\bbX)$. Moreover, for $x \in H^1(\bbX)$, we have
		\begin{equation}
		\label{E:interpolation-time-commutative}
		\mathrm{d}_\sft (\calJ x, x(0)) = \widetilde{\calJ}  \partial_t x.
		\end{equation}
	\end{subequations}
		The hidden constant in \eqref{E:interpolation-time-2nd-interpolant} is an increasing function of 
		\begin{equation}
		\label{E:interpolation-time-shape-constant}
		\max_{j = 2,\dots, J} \dfrac{|I_{j-1}|}{|I_j|}.
		\end{equation}
\end{lemma}

\begin{proof}
According to
\cite[Remark~4.1]{Tantardini:14}, the operator $\calJ $ is invariant
on $\bbS^0_\sft(\bbX) $ and it is bounded with 
\begin{equation*}
\int_{0}^{T}\Norm{\calJ  x}_{\bbX}^2\dt 
\leq 4
\int_0^T \Norm{x}_{\bbX}^2\dt
\end{equation*}	
for all $x \in L^2(\bbX)$. The combination of the two properties
implies \eqref{E:interpolation-time-1st-interpolant}.  

Invariance of \(\widetilde{\calJ} \) on \(\bbS^0_\sft (\bbX)\) follows by elementary
 calculations employing~\eqref{eq:property_psi}.
In order to prove stability for \(\widetilde{\calJ} \), we observe for $x \in L^2(\bbX)$ and \(j\ge2\) that
\begin{align*}\label{eq:esttJx}
   & \|(\widetilde{\calJ}  x)_{|I_{j}}\|_\bbX
   \le
  \dfrac{1}{|I_j|}\displaystyle\int_{I_{j}}\|
   x\|_\bbX\int_{I_{j}} |\psi_{j}|
   +
   \dfrac{1}{|I_j|} \int_{I_{j-1}} \|x\|_\bbX \int_{I_{j-1}}|\psi_{j-1}|\\
   &\quad \le 
   \left( \int_{I_j} \Norm{x}^2_\bbX\right)^{\frac{1}{2}} 
   \left( \int_{I_j} |\psi_j|^2\right)^{\frac{1}{2}}
   +
   \dfrac{|I_{j-1}|}{|I_j|} \left( \int_{I_{j-1}} \Norm{x}^2_\bbX\right)^{\frac{1}{2}}
   \left( \int_{I_{j-1}} |\psi_{j-1}|^2\right)^{\frac{1}{2}}.
\end{align*}
By recalling \eqref{E:definition-psi}-\eqref{eq:property_psi}, we arrive at 
\begin{align*}
   \int_{I_j}\|\widetilde{\calJ}  x\|_\bbX^2\le
   4 \left(1 + \dfrac{|I_{j-1}|}{|I_j|}\right)  
   \int_{I_{j-1}\cup I_j}\|x\|_\bbX^2.
\end{align*}
The same argument applies for \(j=1\). Summing over \(j=1,\ldots,J\) proves \eqref{E:interpolation-time-2nd-interpolant}.
   
Finally \eqref{E:interpolation-time-commutative} follows from the fundamental theorem of calculus and \eqref{eq:property_psi}.
\end{proof}

\begin{figure}[ht]
	\[
	\xymatrixcolsep{4pc}
	\xymatrixrowsep{4pc}
	\xymatrix{
		H^1(\bbX)
		\ar[d]^{(\:\calJ ,\;(\cdot)(0) \:)}
		\ar[r]^{\partial_t}
		& L^2(\bbX)
		\ar[d]^{\widetilde{\calJ} } \\
		\bbS^0_\sft (\bbX)\times \bbX
		\ar[r]^{\mathrm{d}_\sft }
		& \bbS^0_\sft (\bbX)}
	\]
	\caption{\label{F:Time-Operator-Diagram} Commutative diagram representing the relation between the time interpolants. The second component in the operator $(\,\calJ ,\;(\cdot)(0) \,)$ on the left is the evaluation at $t=0$.}
\end{figure}

\subsection{Space interpolation}
\label{SS:interpolation-space}
Regarding the approximation in space, the error notion
$\Err{\cdot}{\cdot}$ leads to the problem of interpolating functions
from $\bbU$ into $\bbU_\sfs $, from $\bbD$ into $\bbD_\sfs$ as well as
from $\overline{\bbP}$ and $\bbP^*$ into $\bbP_\sfs $. Moreover, the
spaces are related via the operators $\calD$ and $\calL$ and their
discrete counterparts and the error notion involves various
$L^2(\Domain)$-orthogonal projections. Therefore, commutative diagrams
like the one in Figure~\ref{F:Time-Operator-Diagram} are of interest
also in this context.  

In order to deal with all these tasks, we invoke the existence of an
operator which maps $H^1$-conforming piecewise polynomials into
$H^2$-conforming ones and enjoys several other properties, whose
importance is made clear along the proof of the next results in this
section. The operator is defined on the Lagrange space of degree
$\sfk$ without boundary conditions, namely 
\begin{equation*}
\label{E:Lagrange-space}
\Lagr{\sfk}{1} := \Poly{\sfk}(\Mesh) \cap H^1(\Domain).
\end{equation*}
Moreover, we denote the jumps and the averages across faces by the usual symbols $\Avg{\cdot}$ and $\Jump{\cdot}$, cf. \cite[Section~38.2.1]{Ern.Guermond:21b}.

\begin{lemma}[Smoothing operator]
\label{L:smoothing}
There is a linear operator \(\calS: \Lagr{\sfk}{1}\to H^2(\Domain)\) which satisfies the following properties
\begin{subequations}
\label{E:smoothing}
    \begin{align}
     \label{E:smoothing-inclusion}
     &\calS(\bbP_\sfs ) \subseteq \bbP &&
     \\
     \label{E:smoothing-normal}
     &\nabla \calS Q\cdot\Normal = 0 &&\text{on}~\Gamma_{p,N} 
     \\
     \label{E:smoothing-elements}
     &\int_{\sfT}(\calS Q)N = \int_{\sfT} QN &&\forall N\in\Poly{\sfk}(\sfT),\; \sfT\in\Mesh
     \\
     \label{E:smoothing-faces}
     &\int_{\sfF}(\calS Q)N = \int_{\sfF} 
     QN &&\forall N\in\Poly{k-1}(\sfF),\; F\in\Faces^i\cup\Faces_{{p,N}}
    \end{align}
  as well as the stability estimate
  \begin{align}
  \label{E:smoothing-stability}
    \|D^j(Q-\calS Q)\|_{\sfT}^2\lesssim \sum_{\sfF\in
    \Faces^i\cup\Faces_{{p,N}}, \; \sfF\cap \sfT\neq\emptyset} \int_\sfF \Avg{\sfh}^{3-2j}
    |\Jump{\nabla Q} \cdot \Normal|^2
  \end{align}
\end{subequations}
for all $Q \in \Lagr{\sfk}{1}$ and $j \in \{0,1,2\}$. The hidden constant depends only on the quantities mentioned in Remark~\ref{R:hidden-constants}.
\end{lemma}

\begin{proof}
See Appendix~\ref{S:proof-smoothing}.
\end{proof}

Having the operator $\mathcal{S}$ at hand, we define $\calI_{\Domain}:L^2(\Domain)\to \Lagr{\sfk}{1}$ via the problem
\begin{subequations}
\label{E:interpolation-space}
	\begin{equation}
	\label{E:interpolation-space-H}
	( \calI_{\Domain}\widetilde p_\Tot, Q )_\Domain
	= ( \widetilde p_\Tot,\calS
	Q )_\Domain \qquad\forall Q \in\Lagr{\sfk}{1}
	\end{equation}
	for $\widetilde{p}_\Tot \in L^2(\Domain)$. Analogously, we
        define  $\calI_{\bbP^*}:\bbP^*\to\bbP_\sfs $ via the problem 
	\begin{equation}
	\label{E:interpolation-space-W*}
	(
	\calI_{\bbP^*}\widetilde m,N )_{\Domain} = \langle \widetilde
	m,\calS N\rangle_\bbP \qquad\forall N\in\bbP_\sfs 
	\end{equation}
	for $\widetilde{m} \in \bbP^*$. Note that the main difference
        between $\calI_{\Domain}$ and $\calI_{\bbP^*}$ is in the
        range. We introduce also $\calI_\bbU: \bbU \to \bbU_\sfs $ by 
	\begin{equation}
	\label{E:interpolation-space-V}
	\calI_\bbU \widetilde u := \widehat{U} + \mathcal{R}_\sfs (
        \calI_\Domain \calD \widetilde u - \calD_\sfs  \widehat{U} ) 
	\end{equation}
	for $\widetilde u \in \bbU$, where $\widehat{U}$ is the
        $\bbU$-orthogonal projection of $u$ onto $\bbU_\sfs $ and
        $\mathcal{R}_\sfs : \bbD_\sfs \to \bbU_\sfs $ is a right inverse of
        $\calD_\sfs $. The existence and the boundedness of the latter
        operator are equivalent to
        Proposition~\ref{P:verification-assumptions}\eqref{P:va-ii},
        cf. \cite[Lemma~C.42]{Ern.Guermond:21b}. Notice also that we 
        have $\calI_\Domain \calD \widetilde u \in \bbD_\sfs$ in view of
        Lemma~\ref{L:space-interpolation-1}\eqref{P:si1-ii} below.   
\end{subequations}

Let us clarify the logic behind the definitions in
\eqref{E:interpolation-space}. The interpolant $\calI_\Domain$ is
intended for the approximation of the total pressure, $\calI_{\bbP^*}$
for the total fluid content and the pressure, and $\calI_{\bbU}$ for
the displacement. It is important for our purpose that
$\calI_{\bbP^*}$ is $\bbP^*$-stable and this requires the use of
an operator $\calS$ in the right-hand side of
\eqref{E:interpolation-space-W*}, mapping the test functions (at
least) into $\bbP$. (We mention, in passing, that $\calS$ is actually
required to map into even more regular functions, in order to enforce
the $L^2(\Domain)$-stability of the interpolant $\calI_{\bbP}$
introduced below, cf. Lemma~\ref{L:space-interpolation-2}.) In
principle, the use of $\calS$ is not needed in
\eqref{E:interpolation-space-H}, when only stability in the
$L^2(\Domain)$-norm is required. Still, it is important at some point relating
$\calI_\Domain$ and $\calI_{\bbP^*}$ via a commutative diagram,
cf. Figure~\ref{F:Space-Operator-Diagram} below. Therefore, we use
$\calS$ also in \eqref{E:interpolation-space-H}. Finally, the
interpolant $\calI_{\bbU}$ is nothing else than the $\bbU$-orthogonal
projection plus a correction, that is necessary for another
commutative diagram.  

\begin{lemma}[Space interpolation -- Part 1]
\label{L:space-interpolation-1}
The operators $\calI_\Domain$, $\calI_{\bbP^*}$ and $\calI_\bbU$ defined in \eqref{E:interpolation-space} enjoy the following properties.
\begin{enumerate}
	\item \label{P:si1-i} $\calI_\Domain$ maps $\bbD$ into $\bbD_\sfs$ and, for all $\widetilde p_\Tot \in L^2(\Domain)$, we have
	\begin{equation*}
	\Norm{\widetilde p_\Tot - \calI_\Domain \widetilde p_\Tot}_\Domain
	\lesssim
	\inf_{\widehat P_\Tot \in \Lagr{\sfk}{1}} \Norm{\widetilde p_\Tot - \widehat P_\Tot}_\Domain.
	\end{equation*}
	\item \label{P:si1-ii} For all $\widetilde m \in \bbP^*$ and $\widetilde p \in \overline{\bbP}$, we have, respectively,
	\begin{equation*}
	\Norm{\widetilde m - \calP_{\bbP_\sfs }^*\calI_{\bbP^*} \widetilde m}_{\bbP^*}
	\lesssim
	\inf_{\widehat M \in \bbP_\sfs } \Norm{\widetilde m - \calP_{\bbP_\sfs }^*\widehat M}_{\bbP^*}
	\end{equation*}
	and
	\begin{equation*}
	\Norm{\widetilde p - \calI_{\bbP^*} \widetilde p}_{\Domain}
	\lesssim
	\inf_{\widehat P \in \bbP_\sfs } \Norm{\widetilde p - \widehat P}_{\Domain}.
	\end{equation*}
	\item \label{P:si1-iii} For all $\widetilde u \in \bbU $, we have $\calD_\sfs  \calI_\bbU \widetilde u = \calI_\Domain \calD \widetilde u$ as well as  
	\begin{equation*}
	\Norm{\widetilde u - \calI_\bbU \widetilde u}_\bbU
	\lesssim
	\inf_{\widehat U \in \bbU_\sfs } \Norm{\widetilde u - \widehat U}_\bbU.
	\end{equation*}
	\item \label{P:si1-iv} The following identities hold true in $L^2(\Domain)$ 
	\begin{equation*}
	\calP_{\bbD_\sfs } \calI_\Domain
	=
	\calI_\Domain \calP_{\bbD}
	\qquad \text{and} \qquad
	\calP_{\bbP_\sfs } \calI_\Domain
	=
	\calI_{\bbP^*} \calP_{\overline\bbP}.
	\end{equation*}  
\end{enumerate}
The hidden constants depend only on the quantities mentioned in Remark~\ref{R:hidden-constants}.
\end{lemma}

\begin{proof}
\eqref{P:si1-i}\; We first check that $\calI_\Domain$ indeed maps $\bbD$
into $\bbD_\sfs$. Recall \eqref{E:abstract-spaces-pressure-total} and
\eqref{E:concrete-spaces-pressure-total}. For $\bbD = L^2(\Domain)$
there is nothing to prove. If $\bbD = L^2_0(\Domain)$, then
\eqref{E:interpolation-space-H} implies, for $\widetilde p_\Tot \in \bbD$, that
\begin{equation*}
( \calI_\Domain\widetilde p_\Tot, 1 )_\Domain 
= 
( \widetilde p_\Tot, \calS 1 )_\Domain
=
( \widetilde p_\Tot, 1)_\Domain 
= 
0.
\end{equation*}
Note that the identity $\calS 1
= 1$ follows from \eqref{E:smoothing-stability} with $j=1$. This
confirms the inclusion $\calI_\Domain \widetilde p_\Tot \in
\bbD_\sfs$. A similar argument reveals that $\calI_\Domain$ is
invariant on $\Lagr{\sfk}{1}$. In fact, owing to
\eqref{E:smoothing-elements}, we have, for $\widetilde p_\Tot \in \Lagr{\sfk}{1}$, that  
\begin{equation*}
( \calI_\Domain \widetilde p_\Tot, Q )_\Domain
= 
( \widetilde p_\Tot,\calS Q )_\Domain
=
( \widetilde p_\Tot, Q )_\Domain \qquad \forall Q  \in \Lagr{\sfk}{1}.
\end{equation*}
Finally, for general $\widetilde p_\Tot \in \bbH$, the boundedness \eqref{E:smoothing-stability} of $\calS$ with $j=0$, combined with a discrete trace inequality \cite[Lemma~12.8]{Ern.Guermond:21a} and an inverse estimate \cite[Lemma~12.1]{Ern.Guermond:21a}, yields
\begin{equation*}
\Norm{\calI_\Domain \widetilde p_\Tot}_\Domain
=
\sup_{Q \in \Lagr{\sfk}{1}}
\dfrac{( \widetilde p_\Tot, \calS Q )_\Domain }{\Norm{Q}_\Domain}
\lesssim
\Norm{\widetilde p_\Tot}_\Domain.
\end{equation*}
The claimed estimate follows by the invariance and the boundedness of $\calI_\Domain$.

\eqref{P:si1-ii}\; For $\widetilde m \in \bbP^*$ and $\widehat{M} \in \bbP_\sfs $, we have
\begin{equation*}
\widetilde m - \calP_{\bbP_\sfs }^*\calI_{\bbP^*} \widetilde m
= 
\widetilde m - \calP_{\bbP_\sfs }^* \widehat{M} + \calP_{\bbP_\sfs }^*(\widehat{M} - \calI_{\bbP^*} \widetilde m).
\end{equation*}
We bound the norm of the second summand on the right-hand side with the help of Proposition~\ref{P:verification-assumptions}(3), the inclusion $\bbP_\sfs  \subseteq L^2(\Domain)$ and \eqref{E:smoothing-stability} with $j=1$ and a discrete trace inequality
\begin{equation*}
\begin{split}
\Norm{\calP_{\bbP_\sfs }^*(\widehat{M} - \calI_{\bbP^*} \widetilde m)}_{\bbP^*}
&\eqsim
\Norm{\widehat{M} - \calI_{\bbP^*} \widetilde m}_{\bbP_\sfs ^*}
= 
\sup_{N \in \bbP_\sfs } \dfrac{\langle\widehat M - \calI_{\bbP^*} \widetilde m, N\rangle_{\bbP_\sfs } }{\Norm{N}_{\bbP}}\\
&=
\sup_{N \in \bbP_\sfs } \dfrac{\langle \calP_{\bbP_\sfs }^* \widehat{M} - \widetilde m, \calS N \rangle_\bbP }{\Norm{N}_\bbP}
\lesssim 
\Norm{\calP_{\bbP_\sfs }^* \widehat{M} - \widetilde m}_{\bbP^*}.
\end{split}
\end{equation*}
The first claimed estimate follows by combining this bound with the above identity. The other estimate follows by establishing invariance on $\bbP_\sfs $ as well as boundedness in the $L^2(\Domain)$-norm exactly as in~\eqref{P:si1-i}.

\eqref{P:si1-iii}\; Recall that the operator $\mathcal R_\sfs $ in \eqref{E:interpolation-space-V} is a right inverse of $\calD_\sfs $, i.e. $\calD_\sfs  \mathcal R_\sfs $ is the identity on $\bbD_\sfs$. Then, the first part of the statement directly follows from the definition of $\calI_\bbU$. We verify the second part as in~\eqref{P:si1-i}, i.e. we prove that $\calI_\bbU$ is invariant on $\bbU_\sfs $ and bounded in the $\bbU$-norm. For $\widetilde u \in \bbU_\sfs $, we have $\widehat{U} = \widetilde u$ in \eqref{E:interpolation-space-V}. Then, according to \eqref{E:smoothing-elements} and \eqref{E:concrete-operator-divergence}, we infer
\begin{equation*}
( \calI_\Domain \calD \widetilde u, Q_\Tot )_\Domain 
=
( \calD\widetilde u, \calS Q_\Tot )_\Domain 
=
( \calD \widetilde u, Q_\Tot )_\Domain
=
( \calD_\sfs  \widetilde u, Q_\Tot )_\Domain \qquad \forall Q_\Tot \in \bbD_\sfs. 
\end{equation*}
This implies $\calI_\Domain \calD \widetilde u = \calD_\sfs  \widetilde u$ and, in turn, $\calI_\bbU \widetilde u = \widetilde u$. Next, for general $\widetilde u \in \bbU$, we have
\begin{equation*}
\Norm{\calI_\bbU \widetilde u}_\bbU
\lesssim
\Norm{\widehat{U}}_\bbU 
+
\Norm{\calI_\Domain \calD \widetilde u - \calD_\sfs  \widehat U}_\Domain.
\end{equation*}
Indeed, the boundedness of $\mathcal R_\sfs $ is equivalent to
\ref{P:verification-assumptions}\eqref{P:va-iii}; see
\cite[Lemma~C.42]{Ern.Guermond:21b}. We conclude the proof of (3) with the help of the
boundedness of $\calI_\Domain$, $\calD$ and $\calD_\sfs $ in the
respective norms and by the definition of $\widehat{U}$ as the
$\bbU$-orthogonal projection of $\widetilde u$. 

\eqref{P:si1-iv}\; Recall that $\calP_{\bbD}$ and $\calP_{\bbD_\sfs }$
are the $L^2(\Domain)$-orthogonal projections onto the spaces $\bbD$
and $\bbD_\sfs$, defined in \eqref{E:abstract-spaces-pressure-total}
and \eqref{E:concrete-spaces-pressure-total}, respectively. For
$\widetilde p_\Tot \in L^2(\Domain)$, we have $\calI_\Domain
\calP_{\bbD}\widetilde p_\Tot \in \bbD_\sfs$, in view
of~\eqref{P:si1-i} above. Then, for all $Q_\Tot \in \bbD_\sfs$, it
holds that 
\begin{equation*}
( \calI_\Domain \calP_\bbD \widetilde p_\Tot, Q_\Tot)_\Domain 
=
( \calP_{\bbD} \widetilde p_\Tot, \calS Q_\Tot )_\Domain
=
(  \widetilde p_\Tot, \calS Q_\Tot )_\Domain
=
( \calI_\Domain \widetilde p_\Tot, Q_\Tot )_\Domain.
\end{equation*}
In particular, we can remove $\calP_{\bbD}$ after the second identity,
because $\calS$ maps $\bbD_\sfs$ into $\bbD$, thanks to
\eqref{E:smoothing-elements}. Hence, the first claimed identity is
verified. The proof of the other one is similar. Recall that
$\calP_{\overline \bbP}$ and $\calP_{\bbP_\sfs }$ are the
$L^2(\Domain)$-orthogonal projections onto the spaces
$\overline{\bbP}$ and $\bbP_\sfs $, defined in
\eqref{E:abstract-spaces-pressure} and
\eqref{E:concrete-spaces-pressure}, respectively. For $\widetilde p
\in L^2(\Domain)$ and $N \in \bbP_\sfs $, we have 
\begin{equation*}
( \calI_{\bbP^*} \calP_{\overline \bbP} \widetilde p, N)_\Domain 
=
( \calP_{\overline \bbP} \widetilde p, \calS N )_\Domain
=
( \widetilde p, \calS N )_\Domain
=
( \calI_\Domain \widetilde p, N )_\Domain.
\end{equation*}
This time we could remove $\calP_{\overline \bbP}$ after the second
identity thanks to \eqref{E:smoothing-inclusion}. Thus, also the
second claimed identity is verified.  
\end{proof}

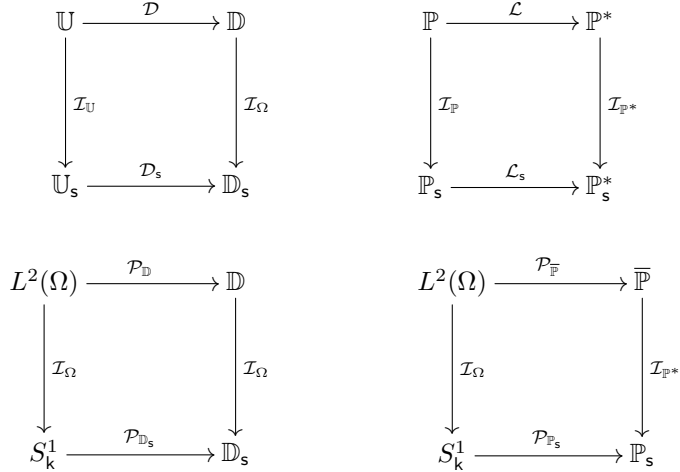
\begin{figure}[ht]
	\[
	\xymatrixcolsep{4pc}
	\xymatrixrowsep{4pc}
	\xymatrix{
		\bbU
		\ar[d]^{\calI_\bbU}
		\ar[r]^{\calD}
		& \bbD
		\ar[d]^{\calI_\Domain} \\
		\bbU_\sfs 
		\ar[r]^{\calD_\sfs }
		& \bbD_\sfs}
	\hspace{50pt}
	\xymatrix{
		\bbP
		\ar[d]^{\calI_\bbP }
		\ar[r]^{\calL}
		& \bbP^*
		\ar[d]^{\calI_{\bbP^*}} \\
		\bbP_\sfs 
		\ar[r]^{\calL_\sfs }
		& \bbP_\sfs ^*}\]\\[5pt]
	\[
	\xymatrixcolsep{4pc}
	\xymatrixrowsep{4pc}
	\xymatrix{
		L^2(\Domain)
		\ar[d]^{\calI_\Domain}
		\ar[r]^{\calP_\bbD}
		& \bbD
		\ar[d]^{\calI_\Domain} \\
		\Lagr{\sfk}{1}
		\ar[r]^{\calP_{\bbD_\sfs }}
		& \bbD_\sfs}
	\hspace{50pt}
	\xymatrix{
		L^2(\Domain)
		\ar[d]^{\calI_\Domain }
		\ar[r]^{\calP_{\overline \bbP}}
		& \overline{\bbP}
		\ar[d]^{\calI_{\bbP^*}} \\
		\Lagr{\sfk}{1}
		\ar[r]^{\calP_{\bbP_\sfs }}
		& \bbP_\sfs }
	\]
	\caption{\label{F:Space-Operator-Diagram} Commutative diagrams
          representing the relations among the space interpolants.} 
\end{figure}

The interpolants defined up to this point are compatible with the
divergence, i.e. with the operators $\calD$ and $\calD_\sfs $, and
with the various $L^2(\Domain)$-orthogonal projections, as summarized
in Figure~\ref{F:Space-Operator-Diagram}. Still, the compatibility
with the Laplacian, i.e. with the operators $\calL$ and $\calL_\sfs $,
is not guaranteed. Since $\calL$ and $\calL_\sfs $ are one-to-one, the
diagram on the upper right corner of
Figure~\ref{F:Space-Operator-Diagram} uniquely defines $\calI_\bbP:
\bbP \to \bbP_\sfs $ as 
\begin{equation}
\label{E:interpolation-space-W}
\calI_\bbP := \calL_\sfs ^{-1} \calI_{\bbP^*} \calL.
\end{equation}

The next lemma reveals that also $\calI_\bbP$ has favorable
approximation properties. To this end, we introduce the following
broken $H^2$-norm 
\begin{equation*}
\label{E:H2-norm}
\Norm{\cdot}_{H^2(\Mesh)}^2 
:=
\kappa\left(  \sum_{\sfT \in \Mesh} \Norm{D^2 \cdot}^2_{L^2(\sfT)}
+
\sum_{\sfF \in \Faces^i \cup \Faces_{{p,N}}} \Norm{ \Avg{\sfh}^{-1/2} \Jump{\Grad \cdot}}^2_{L^2(\sfF)}\right) 
\end{equation*}
as well as the constant
\begin{equation}
\label{E:discrete-elliptic-regularity}
\varepsilon_\sfs := \inf_{N \in \bbP_\sfs } 
\dfrac{\Norm{\calL_\sfs
		N}_{\Domain}}{\Norm{N}_{H^2(\Mesh)}} .
\end{equation}

\begin{remark}[Discrete elliptic regularity]
\label{R:discrete-elliptic-regularity}
The constant $\varepsilon_\sfs$ in \eqref{E:discrete-elliptic-regularity} is certainly
finite, because $\bbP_\sfs $ is finite-dimensional. The size of the constant is related to the control of a $H^2$-like norm by the $L^2$-norm of $\calL_\sfs $, i.e. our discretization of the
Laplacian. Therefore, the constant is a discrete measure of the
elliptic regularity. Note also that $\varepsilon_\sfs$ can be equivalently interpreted as an inf-sup constant
\begin{equation*}
\varepsilon_\sfs = 
\inf_{N \in \bbP_\sfs } \sup_{\widehat{P} \in \bbP_\sfs} 
\dfrac{\langle \calL N, \widetilde P\rangle_{\bbP} }{\Norm{N}_{H^2(\Mesh)}\Norm{\widetilde P}_\Domain}, 
\end{equation*}
where the stability of the standard weak formulation of the Laplacian is analyzed in a nonstandard (i.e. nonsymmetric) setting,
cf. \eqref{E:concrete-operator-laplacian}. We are aware of only a few results regarding the size of $\varepsilon_\sfs$. For $H^1$-conforming Lagrange finite elements and convex $\Domain$, Makridakis established a lower bound in terms of the shape
parameter \eqref{E:shape-constant} of $\Mesh$, provided that the mesh
is not too much graded, see~\cite{Makridakis:18}. It is unclear to us whether the latter condition is necessary or not. When $\Domain$ is not convex, the connection with the elliptic regularity suggests that a lower bound of $\varepsilon_\sfs$ only in terms of shape regularity might be not possible. As for the question raised in Remark~\ref{R:decay-rate}, we postpone further investigation on this point to future work.  
\end{remark}
%\todo[inline]{I am not fully happy with
%  remark~\ref{R:discrete-elliptic-regularity}. What happens when the
%  domain is not convex? Is there a result deteriorating with the angle
%  of reentering corners? Can we prove convergence at all?
%  I think this constant is critical and the discussion is still too
%  hidden!
%  \\
%  I think an estimate of the following form could be possible:
%  \begin{align*}
%    \|N\|_{H^2(\Mesh)}^2&\lesssim
%    \sum_{\sfT\in\Mesh}\sfh_{\sfT}^{-2}| N- \mathcal{H}_\sfT N|^2_{H^1(\omega_\sfT)}
%    \\
%    &\lesssim
%      \sum_{\sfT\in\Mesh}\sfh_{\sfT}^{-2}\left\{|N-n|^2_{H^1(\omega_\sfT)}
%      +|n-\mathcal{H}_\sfT n|^2_{H^1(\omega_\sfT)} +|\mathcal{H}_\sfT(N-
%      n)|^2_{H^1(\omega_\sfT)}\right\}
%      \\
%    &\lesssim
%      \sum_{\sfT\in\Mesh}\sfh_{\sfT}^{-2}\left\{|N-n|^2_{H^1(\omega_\sfT)}
%      +|n-\mathcal{H}_\sfT n|^2_{H^1(\omega_\sfT)} \right\},
%  \end{align*}
%  where \(-\calL n=-\calL_\sfs N\), i.e. \(N\) is the Galerkin
%  approxmation to \(n\) and thus quasi-optimal. Consequently for the
%  former term interpolation estimates can be employed and for the
%  latter term, one can apply Bramble-Hilbert 
%  in case \(\mathcal{H}_\sfT\) is e.g. the
%  \(L^2\)-projection onto \( P_1(\omega_\sfT)\).
%  Consequently, the constant~\eqref{E:discrete-elliptic-regularity} can be
%  bounded in terms of \(1/h\) depending on the regularity of the
%  boundary... 
%} 

\begin{lemma}[Space interpolation -- Part 2]
	\label{L:space-interpolation-2}
The operator $\calI_\bbP$ defined above has a unique bounded extension
(denoted by the same symbol) to $\overline{\bbP}$ which satisfies for all $\widetilde p \in \overline{\bbP}$ the
estimate 
\begin{equation*}
\label{E:space-interpolation-2}
\Norm{\widetilde p - \calI_\bbP \widetilde p}_\Domain
\lesssim \varepsilon_\sfs^{-1}
\inf_{\widehat P \in \bbP_\sfs } \Norm{\widetilde p - \widehat P}_\Domain.
\end{equation*}
The hidden constant
depends only on the quantities mentioned in
Remark~\ref{R:hidden-constants}.
\end{lemma}

\begin{proof}
As for Lemma~\ref{L:space-interpolation-1}(2)-(3), we verify the claimed estimate by showing that $\calI_\bbP$ is invariant on $\bbP_\sfs $ and that it is bounded in the $L^2(\Domain)$-norm. Note that the latter property implies also the existence of a unique bounded extension of $\calI_\bbP$ to $\overline{\bbP}$, because $\bbP$ is a dense subspace thereof. Assume first $\widetilde p \in \bbP_\sfs $. Owing to \eqref{E:interpolation-space-W} and \eqref{E:smoothing-inclusion}, we have
\begin{align*}
\langle \calI_\bbP \widetilde p, \calL_\sfs  N\rangle_{\bbP_\sfs}
=
\langle \calL \widetilde p,\calS N\rangle_\bbP
=
\langle \widetilde p,  \calL \calS N\rangle_\bbP \qquad N \in \bbP_\sfs.
\end{align*}
We recall \eqref{E:abstract-operators-concrete} and integrate by parts element-wise \cite[eq. (3.6)]{Arnold.Brezzi.Cockburn.Marini:02}
\begin{equation*}
\langle \calI_\bbP \widetilde p, \calL_\sfs  N\rangle_{\bbP_\sfs}
=
\kappa \left(  -\sum_{\sfT \in \Mesh} \int_{\sfT} (\Delta \widetilde p) \calS N  
+ 
\sum_{\sfF \in \Faces^i \cup \Faces_{{p,N}}}
\int_\sfF (\Jump{\Grad \widetilde p} \cdot \Normal) \calS N \right) . 
\end{equation*}
Then, we make use of \eqref{E:smoothing-elements}-\eqref{E:smoothing-faces} and we integrate by parts back. It results
\begin{equation*}
\langle \calI_\bbP \widetilde p, \calL_\sfs  N\rangle_{\bbP_\sfs}
=
\left\langle \widetilde p, \calL N \right\rangle_\bbP
=
\left\langle \widetilde p, \calL_\sfs  N\right\rangle_{\bbP_\sfs}
\end{equation*}
where the last identity is obtained with the help of \eqref{E:concrete-operator-laplacian}. Since $\calL_\sfs $ is one-to-one, we conclude $\calI_\bbP \widetilde p = \widetilde p$.

Next, for general $\widetilde p \in \bbP$, a similar argument as before yields
\begin{equation*}
\begin{split}
\|\calI_\bbP \widetilde p\|_{\Domain}
&= 
\sup_{N\in \bbP_\sfs }\frac{\langle \calI_\bbP \widetilde p, \calL_\sfs  N\rangle_{\bbP_\sfs}}{\|\calL_\sfs N\|_\Domain}
=
\sup_{N\in\bbP_\sfs }\frac{\langle \widetilde p, \calL \calS N\rangle_\bbP}{\|\calL_\sfs N\|_\Domain}.
\end{split}
  \end{equation*}
Recall the inclusion $\calS N \in H^2(\Domain)$ and \eqref{E:smoothing-inclusion}-\eqref{E:smoothing-normal}. We integrate by parts, then we invoke \eqref{E:smoothing-stability} with $j=2$ and \eqref{E:discrete-elliptic-regularity}
\begin{equation*}
\|\calI_\bbP \widetilde p\|_{\bbH}
= 
\kappa ( \widetilde p,  -\Delta \calS N )_\Domain
\lesssim 
\Norm{\widetilde p}_{\Domain} \Norm{N}_{H^2(\Mesh)}
\leq
\varepsilon_\sfs^{-1}\Norm{\widetilde p}_{\Domain} \Norm{\calL_\sfs  N}_{\Domain}.
  \end{equation*}
The combination of this bound with the previous identity implies the announced boundedness of $\calI_\bbP$.
\end{proof}

\subsection{Space-time interpolation}
\label{SS:interpolation-space-time}
We are in position to combine the interpolants from the two previous sections in order to conclude the proof of Theorem~\ref{T:decoupling}. To this end, let us preliminary recall that the operators acting in time commute with those acting in space. Our main device is the space-time interpolant $\calI: \bbY_1 \to \bbY_{1,\sfs\sft}$ defined as
\begin{equation*}
\label{E:interpolation-space-time}
\calI \widetilde y_1 
:=
\big(\calI_{\bbU}\calJ  \widetilde u,\; 
\calI_{\Domain} \calJ  \widetilde p_{\Tot},\;
\calI_{\bbP} \widetilde{\calJ} \widetilde p,\; \calI_{\bbP^*} \calJ  \widetilde m,\; 
\calI_{\bbP^*} \widetilde m(0)\big)
\end{equation*}
for $\widetilde y_1 = (\widetilde u, \widetilde p_\Tot, \widetilde p, \widetilde m) \in \bbY_1$. 

Our strategy to verify Theorem~\ref{T:decoupling} is as follows. We first establish \eqref{E:decoupling} for $y_1 \in \bbY_1$. Then we extend the result by density, because the right-hand side of \eqref{E:decoupling} is bounded with respect to the norm $\Norm{\cdot}_1$, cf. Theorem~\ref{T:well-posedness-weak-formulation}.

First of all, we recall the definitions of $\bbY_1$ and $\bbY_{1, \sfs\sft}$ in \eqref{E:trial-space} and \eqref{E:trial-space-discrete}, respectively, and notice that $\calI$ is well-defined. In fact, the discussion in the previous sections shows that the interpolation of each component acts as follows
\begin{equation*}
\begin{array}{cl}
L^2(\bbU) 
\stackrel{\calJ }{\longrightarrow}
\bbS^0_\sft (\bbU) 
\stackrel{\calI_\bbU}{\longrightarrow}
\bbS^0_\sft (\bbU_\sfs )
\qquad 
&\text{for $\;\widetilde u$}\\[2pt]
L^2(\bbD) 
\stackrel{\calJ }{\longrightarrow}
\bbS^0_\sft (\bbD) 
\stackrel{\calI_\Domain}{\longrightarrow}
\bbS^0_\sft (\bbD_\sfs)
\qquad 
&\text{for $\;\widetilde p_\Tot$}\\[2pt]
L^2(\bbP) 
\stackrel{\widetilde{\calJ} }{\longrightarrow}
\bbS^0_\sft (\bbP) 
\stackrel{\calI_\bbP}{\longrightarrow}
\bbS^0_\sft (\bbP_\sfs )
\qquad 
&\text{for $\;\widetilde p$}\\[2pt]
H^1(\bbP^*) 
\stackrel{\calJ }{\longrightarrow}
\bbS^0_\sft (\bbP^*) 
\stackrel{\calI_{\bbP^*}}{\longrightarrow}
\bbS^0_\sft (\bbP_\sfs )
\qquad 
&\text{for $\;\widetilde m$}\\[2pt]
\bbP^* \stackrel{\calI_{\bbP^*}}{\longrightarrow}
\bbP_\sfs  &\text{for $\;\widetilde m(0)$}.
\end{array}
\end{equation*}
In particular, the second line makes use of
Lemma~\ref{L:space-interpolation-1}(1) and the last one exploits the
inclusion $\widetilde m \in H^1(\bbP^*) \subseteq C^0(\bbP^*)$. 

In order to verify Theorem~\ref{T:decoupling}, we bound the left-hand
side of \eqref{E:decoupling} by
\begin{equation}
\label{E:decoupling-proof}
\inf_{\widetilde Y_1 \in \bbY_{1, \sfs\sft}}
\Err{\widetilde y_1}{\widetilde Y_1}^2
\leq
\Err{\widetilde y_1}{\calI \widetilde y_1}^2.
\end{equation}
Then, we recall the definition \eqref{E:error-notion} of the error
notion and we bound the six terms therein (denoted by \textit{(i)},
\textit{(ii)}, \ldots~for brevity) one by
one. Lemma~\ref{L:interpolation-time} and Lemma
\ref{L:space-interpolation-1}\eqref{P:si1-iii} state that $\calJ $ and $\calI_\bbU$
are bounded linear projections. Therefore, their combination $\calJ
\calI_\bbU$ is a $L^2(\bbU)$-bounded linear projection onto
$\bbS^0_\sft (\bbU_\sfs )$. This implies 
\begin{equation*}
(i) \lesssim
\inf_{\widehat{U} \in \bbS^0_\sft (\bbU_\sfs )}
\int_0^T \Norm{\widetilde u - \widehat{U}}^2_\bbU \dt.
\end{equation*}
The same reasoning with Lemma~\ref{L:space-interpolation-1}\eqref{P:si1-i} implies 
\begin{equation*}
(ii) \lesssim
\inf_{\widehat{P}_\Tot \in \bbS^0_\sft (\bbD_\sfs)}
\int_0^T \dfrac{1}{\mu}\Norm{\widetilde p_\Tot - \widehat{P}_\Tot}^2_\Domain \dt.
\end{equation*}
Regarding the third term, we first recall \eqref{E:interpolation-time-commutative}
\begin{equation*}
\mathrm{d}_\sft  ( \calI_{\bbP^*} \calJ  \widetilde m, \; \calI_{\bbP^*} \widetilde m(0))
=
\calI_{\bbP^*}
\mathrm{d}_\sft  ( \calJ  \widetilde m, \; \widetilde m(0))
=
\calI_{\bbP^*} \widetilde{ \calJ}  \partial_t \widetilde m,
\end{equation*}
then, we make use of \eqref{E:interpolation-space-W}
\begin{equation*}
\calL_\sfs  \calI_\bbP \widetilde{\calJ}  \widetilde p
=
\calI_{\bbP^*} \widetilde{\calJ}  \calL \widetilde p
\end{equation*}
and, combining the two identities, we arrive at
\begin{equation*}
\mathrm{d}_\sft  ( \calI_{\bbP^*} \calJ  \widetilde m, \; \calI_{\bbP^*} \widetilde m(0)) 
+
\calL_\sfs  \calI_\bbP \widetilde{\calJ} \widetilde p
=
\calI_{\bbP^*} \widetilde{\calJ}  (\partial_t \widetilde m + \calL \widetilde p).
\end{equation*}
The same argument used in the above bounds for \textit{(i)} and \textit{(ii)}, with Lemma~\ref{L:space-interpolation-1}\eqref{P:si1-ii}, implies
\begin{equation*}
(iii) \lesssim
\inf_{\widehat{W} \in \bbS^0_\sft (\bbP_\sfs )}
\int_0^T \Norm{\partial_t \widetilde m + \calL \widetilde p - \calP_{\bbP_\sfs }^* \widehat{W}}^2_{\bbP^*} \dt.
\end{equation*}
For the fourth term, we directly use
Lemma~\ref{L:space-interpolation-1}\eqref{P:si1-ii} to obtain
\begin{equation*}
(iv) \lesssim
\inf_{\widehat M_0\in\bbP_\sfs } \Norm{\widetilde m(0)  - \calP_{\bbP_\sfs }^* \widehat M_0}_{\bbP^*}^2.
\end{equation*}
Concerning the fifth term, we invoke Lemma~\ref{L:space-interpolation-1}\eqref{P:si1-iii}
\begin{equation*}
\begin{split}
&\lambda \calD_\sfs  \calI_\bbU \calJ  \widetilde u
-
\calI_\Domain \calJ  \widetilde p_\Tot 
-
\alpha \calP_{\bbD_\sfs } \calI_\bbP \widetilde{ \calJ}  \widetilde p
=\\
&\qquad \calI_\Domain \calJ ( \calD \widetilde u - \widetilde p_\Tot - \alpha \calP_{\bbD}\widetilde p)
+ 
\alpha \calI_\Domain \calJ  \calP_{\bbD}\widetilde p
- 
\alpha \calP_{\bbD_\sfs } \calI_\bbP \widetilde{ \calJ}  \widetilde p
\end{split}
\end{equation*}
The reasoning from the bound of \textit{(ii)} shows that the first
summand in the right-hand side is a near best approximation of $(\calD
\widetilde u - \widetilde p_\Tot - \alpha \calP_{\bbD}\widetilde p)$
in $\calS^0_\sft (\bbD_\sfs)$. The other two summands can be rewritten
as $\alpha \calP_{\bbD_\sfs }(\calI_\Domain \calJ  - \calI_\bbP
\widetilde{ \calJ} )\widetilde p$, thanks to
Lemma~\ref{L:space-interpolation-1}\eqref{P:si1-iv}. Arguing as
before, we see that both $\calI_\Domain \calJ  \widetilde p$ and
$\calI_\bbP \widetilde{ \calJ} \widetilde p$ are near best
approximations of $\widetilde p$ in $\bbS^0_\sft (\bbP_\sfs )$, the
latter one in view of Lemma~\ref{L:space-interpolation-2}. These
observations, the triangle inequality and the definition
\eqref{E:gamma} of $\gamma$ reveal 
\begin{equation*}
(v) \lesssim \inf_{\widehat{Q}_\Tot \in \bbS^0_\sft (\bbD_\sfs)}
\int_{0}^{T} \dfrac{1}{\lambda + \mu}\Norm{(\calD \widetilde u - \widetilde p_\Tot - \alpha \calP_{\bbD}\widetilde p) - \widehat{Q}_\Tot}^2_\Domain \dt
+ 
\varepsilon_\sfs^{-1}
\inf_{\widehat P \in \bbS^0_\sft (\bbP_\sfs )} \int_0^T \dfrac{1}{\gamma} \Norm{\widetilde p - \widehat{P}}^2_\Domain\dt.
\end{equation*}
The sixth and last term can be treated similarly. We invoke again
Lemma~\ref{L:space-interpolation-1}\eqref{P:si1-iii} and
Lemma~\ref{L:space-interpolation-1}\eqref{P:si1-iv}, so as to obtain 
\begin{equation*}
\begin{split}
&\alpha \calP_{\bbP_\sfs } \calD_\sfs  \calI_\bbU \calJ  \widetilde u 
+
\sigma \calI_\bbP \widetilde{ \calJ}  \widetilde{p}
-
\calI_{\bbP^*} \calJ  \widetilde m\\
& \qquad =
\calI_{\bbP^*} \calJ (\alpha \calP_{\overline \bbP} \calD \widetilde u + \sigma \widetilde p - \widetilde m)
 + \sigma (\calI_\bbP \widetilde{ \calJ}  - \calI_{\bbP^*} \calJ  )\widetilde p.
\end{split}
\end{equation*}
Again, all operators yield near best approximation in the respective
spaces. In particular, for $\calI_{\bbP^*} \calJ $, we make use of the
second part of
Lemma~\ref{L:space-interpolation-1}\eqref{P:si1-ii}. Thus we arrive at
the following estimate 
\begin{equation*}
(vi) \lesssim \inf_{\widehat{Q} \in \bbS^0_\sft (\bbP_\sfs )}
\int_{0}^{T} \gamma \Norm{(\calP_{\overline \bbP}\calD \widetilde u + \sigma \widetilde p - \widetilde m) - \widehat{Q}}^2_\Domain \dt
+ 
\varepsilon_\sfs^{-1}
\inf_{\widehat P \in \bbS^0_\sft (\bbP_\sfs )} \int_0^T \dfrac{1}{\gamma} \Norm{\widetilde p - \widehat{P}}^2_\Domain\dt
\end{equation*}
with the help of the definition \eqref{E:gamma} of $\gamma$. We insert
the above bounds of \textit{(i)}-\textit{(vi)} into
\eqref{E:decoupling-proof}. This establishes \eqref{E:decoupling} for
$\widetilde y_1 \in \bbY_1$ and the right-hand side in the resulting
estimate is bounded in terms of the norm $\Norm{\cdot}_1$,
cf. Theorem~\ref{T:well-posedness-weak-formulation}. Thus, we can
extend the estimate by density from $\bbY_1$ to
$\overline{\bbY}_1$. We conclude by noticing that the bounds of
\textit{(v)} and \textit{(vi)} simplify to 
\begin{equation*}
(v) + (vi) \lesssim 
\varepsilon_\sfs^{-1}\inf_{\widehat P \in \bbS^0_\sft (\bbP_\sfs )} \int_0^T \dfrac{1}{\gamma} \Norm{\widetilde p - \widehat{P}}^2_\Domain\dt
\end{equation*} 
if $\widetilde y_1 = y_1$ is the solution of
\eqref{E:BiotProblem-weak-formulation}.

\section{Numerical results}
\label{S:numerics}
In this section we test the performance of the concrete space discretization proposed in Section~\ref{S:concrete-discretization} on two well-established benchmarks for the Biot's equations. Owing to Corollary~\ref{C:first-order-convergence} and Remark~\ref{R:decay-rate}, we restrict ourselves to the lowest order case $\sfk = 1$, corresponding to quadratic (resp. linear) $H^1$-conforming Lagrange finite elements for the displacement (resp. for the other variables). All experiments have been implemented with the help of the library ALBERTA 3.0, see \cite{Heine.Koester.Kriessl.Schmidt.Siebert,Schmidt.Siebert:05}.

\subsection{Terzaghi's problem}
\label{SS:Terzaghi}
The consolidation of a vertical soil column of total depth $H > 0$ is a classical one-dimensional test case in poroelasticity. The column is subject to compression and it is drained on top, whereas it is impermeable and no displacement occurs at the bottom. Moreover, there is no other force acting in the interior of the column, the initial fluid content equals zero and there are no sources nor sinks. Therefore, in this case, the initial-boundary value problem \eqref{E:BiotProblem} for the Biot's equations reads
\begin{equation}
\label{E:Terzaghi-equations}
\begin{alignedat}{3}
-(2\mu + \lambda) u_{zz} + \alpha p_z = 0,&& 
\qquad \partial_t (\alpha u_z + \sigma p) - \kappa p_{zz} = 0, \quad &  \text{in $(0,H) \times (0,T)$  }\\
-(2\mu + \lambda) u_z + \alpha p = F,&& \qquad p = 0, \quad  &\text{on $\{0\} \times (0,T)$  } \\
u = 0, && \qquad p_z = 0, \quad & \text{on $\{H\} \times (0,T)$  } \\
 && \alpha u_z + \sigma p = 0, \quad & \text{in $(0,H) \times \{0\}$  }
\end{alignedat}
\end{equation}  
with the subscripts $(\cdot)_z$ and $(\cdot)_{zz}$ denoting the partial derivatives with respect to the depth $z \in (0, H)$ and $F$ the modulus of the force acting on top of the column.

Remarkably, the exact solution of \eqref{E:Terzaghi-equations} is explicitly known, cf. \cite[Section~4.1.1]{Phillips:05}. In particular, the pressure equals
\begin{equation}
\label{E:Terzaghi-pressure}
p(z,t) = p_0 \sum_{m=0}^{+\infty} \dfrac{4}{(2m+1)\pi}\sin\Big( \dfrac{(2m+1)\pi z}{2H} \Big) \exp\Big(- \dfrac{(2m+1)^2 \pi^2 \widetilde \gamma k t}{4H^2} \Big)
\end{equation}
with the auxiliary parameters
\begin{equation*}
\label{E:Terzaghi-auxiliary-parameters}
\widetilde \gamma :=  \Big ( \dfrac{\alpha^2}{2\mu + \lambda} + \sigma \Big)^{-1} \qquad \text{and} \qquad p_0 := \dfrac{\alpha \widetilde \gamma F} {2\mu + \lambda}.
\end{equation*}
The displacement can be derived via the relation
\begin{equation*}
\label{E:Terzaghi-displacement}
u_z = \dfrac{\alpha p - F}{2\mu + \lambda}
\end{equation*}
and the other variables are obtained through their definition, cf. Remark~\ref{R:auxiliary-variables}. 

In analogy with \cite[Section~4.1.1]{Phillips:05}, we set 
\begin{equation*}
H = 1, \qquad T = 1, \qquad F = 10^3
\end{equation*}
as well as \footnote{More precisely \cite{Phillips:05} sets the Young's modulus to $E = 10^5$ and the Poisson's ratio to $\nu = 0.2$.}
\begin{equation*}
\mu = 41667, \quad \lambda = 27778, \quad \alpha = 1, \quad \sigma = 0.1, \quad \kappa = 10^{-6}.
\end{equation*}

The explicit knowledge of the exact solution makes this test case particularly suited to observe the error decay rate. Owing to \eqref{E:gamma}, \eqref{E:error-notion}, Remark~\ref{R:augmented-error-notion} and Theorem~\ref{T:decoupling}, we consider the following error notion
\begin{equation}
\label{E:Terzaghi-error}
\int_0^T \left( \| u - U \|_\bbU^2 + \dfrac{1}{\mu}\| p_\Tot - P_\Tot\|^2_\Domain + \sigma\| p - P\|^2_\Domain \right)\dt
\end{equation}
where the pressure $L^2(L^2(\Domain))$-error is included for a better comparison with the literature. For the evaluation of the exact solution, we take into account the first $5000$ summands in the series \eqref{E:Terzaghi-pressure}. 

We first consider a fixed time discretization with $J = 5000$ intervals of equal size and observe the error decay rate with respect to the space discretization. For this purpose, we begin with a mesh in space consisting only of the interval $(0,H)$ and refine it $10$ times by dividing all intervals each time. We report on the left side of Table \ref{TB:Terzaghi-errors} below the evaluation of the error \eqref{E:Terzaghi-error} and the experimental order of convergence (EOC) with respect to the number of degrees of freedom in the space discretization, namely
\begin{equation}
\label{E:Terzaghi-DOFs-space}
\#\textrm{DOFs} = \dim(\bbU_\sfs \times \bbD_\sfs \times \bbP_\sfs \times \bbP_\sfs)
\end{equation}    
cf. \eqref{E:trial-space-discrete}. By increasing the number of refinements, the EOC appears to converge to $1.5$, with only a slight decrease in the last refinement, which is likely due to an insufficient number of intervals in the time discretization. 

\begin{table}[htp]
	\centering
	\begin{tabular}{|c|c|c|}
		\#\textrm{DOFS} &
		\textrm{error} &
		\textrm{EOC} \\[1ex]
		\hline
		&&
		\\[-1.5ex]
		9    & $1.42 \times 10^{-2}$ & \\
		14   & $1.08 \times 10^{-2}$ & 0.63\\
		24   & $7.65 \times 10^{-3}$ & 0.63\\
		44   & $5.41 \times 10^{-3}$ & 0.57\\
		84   & $3.82 \times 10^{-3}$ & 0.54\\
		164  & $2.54 \times 10^{-3}$ & 0.61\\
		324  & $1.33 \times 10^{-3}$ & 0.95\\
		644  & $5.61 \times 10^{-4}$ & 1.26\\
		1284 & $2.14 \times 10^{-4}$ & 1.40\\
		2564 & $7.86 \times 10^{-5}$ & 1.45\\
		5124 & $2.94 \times 10^{-5}$ & 1.42 
	\end{tabular}
	\hspace{40pt}
	\begin{tabular}{|c|c|c|}
		$J$ &
		\textrm{error} &
		\textrm{EOC} \\[1ex]
		\hline
		&&
		\\[-1.5ex]
		5    & $1.17 \times 10^{-4}$ & \\
		10   & $7.23 \times 10^{-5}$ & 0.70\\
		20   & $4.41 \times 10^{-5}$ & 0.71\\
		40   & $2.67 \times 10^{-5}$ & 0.72\\
		80   & $1.61 \times 10^{-5}$ & 0.73\\
		160  & $9.70 \times 10^{-6}$ & 0.73\\
		320  & $5.85 \times 10^{-6}$ & 0.73\\
	\end{tabular}
	\vspace{1ex}
	\caption{Error decay in Terzaghi's problem with respect to the space discretization (left) and the time discretization (right).}
	\label{TB:Terzaghi-errors}
\end{table} 

The decay rate $1.5$ is consistent with the one observed, e.g., in \cite[Section~4.1.1]{Phillips:05}. There the EOC is actually $0.5$ and the difference is explained by the different space discretization and error notion. Indeed, our discretization is of second-order (cf. Remark~\ref{R:decay-rate}) whereas the one in \cite{Phillips:05} is of first-order only. For a theoretical justification note that the exact pressure \eqref{E:Terzaghi-pressure} is singular, because the initial value $p(\cdot, 0) = p_0$ does not satisfy the Dirichlet boundary condition on top of the column. More precisely, for $s \in (0,1)$, we have
\begin{equation*}
\label{E:Terzaghi-EOC-space}
\Norm{p}^2_{L^2(H^{1+s}(\Domain))} \approx
\sum_{m=0}^{+\infty} m^{2s} \int_0^T \exp(-m^2t)\dt \approx \sum_{m=0}^{+\infty} m^{2s -2},
\end{equation*}
where the symbol $\approx$ indicates that we neglect all multiplicative constants apart of those depending on $m$. Therefore, we have $p \in L^2(H^{1+s}(\Domain))$ if and only if $s < 0.5$.

Second, we consider a fixed space discretization with the mesh obtained by $13$ global refinements of the interval $(0,H)$. In other words, the mesh is two levels finer than the last one on the left side of Table~\ref{TB:Terzaghi-errors} and it is such that $\#\mathrm{DOFs} = 40964$. We use time discretizations with $J$ intervals of equal size and observe the error decay for increasing $J$. According to the data on the right side of Table~\ref{TB:Terzaghi-errors}, the EOC is somehow close to $0.75$. Although we were not able to find a similar convergence test in the literature, the observed value is in line with our expectation. Indeed, by arguing as before (see also \cite[Section~7.1]{Schotzau.Schwab:00}), we observe that 
\begin{equation*}
\Norm{p}_{H^s(L^2(\Domain))}^2 \approx \sum_{m=0}^{+\infty} m^{4s-2} \int_0^T \exp(-m^2t) \approx \sum_{m=0}^{+\infty} m^{4s-4}
\end{equation*}
for $s \in (0,1)$. Therefore, we have $p \in H^s(L^2(\Domain))$ if and only if $s < 0.75$.

\subsection{Cantilever bracket problem}
\label{SS:Cantilever-bracket}
We consider the extension to poroelasticity of a well-established two-dimensional test case in linear elasticity. The elastic material initially occupies the region $\Domain = (0,1)^2$, it is clamped on the left side, a uniformly distributed load is applied on the top side and the entire boundary is assumed to be impermeable. Moreover, there is no other force acting in the interior of $\Domain$, the initial fluid content equals zero and there are no sources nor sinks. Therefore, in this case, the initial-boundary value problem \eqref{E:BiotProblem} for the Biot's equations reads
\begin{equation*}
\label{E:Cantilever-bracket-equations}
\begin{alignedat}{3}
-\Div(2\mu \SymGrad u + (\lambda \Div u - \alpha p)\mathrm{I}) = 0,&& 
\;\; \partial_t (\alpha \Div u + \sigma p) - \kappa \Delta p = 0, \;\; &  \text{in $\Domain \times (0,T)$  }\\
u = 0, && \quad \Grad p\cdot \Normal = 0, \;\; & \text{on $\Gamma_{\texttt{L}} \times (0,T)$  } \\
(2\mu \SymGrad u + (\lambda \Div u - \alpha p)\mathrm{I})\Normal = -F,&& \qquad \Grad p \cdot \Normal = 0, \;\;  &\text{on $\Gamma_\texttt{T} \times (0,T)$  } \\
(2\mu \SymGrad u + (\lambda \Div u - \alpha p)\mathrm{I})\Normal = 0,&& \quad \Grad p \cdot \Normal = 0, \;\;  &\text{on $\Gamma_\texttt{RB} \times (0,T)$  }\\
&& \alpha \Div u + \sigma p = 0, \;\; & \text{in $\Domain \times \{0\}$  }
\end{alignedat}
\end{equation*}  
with $\Gamma_\texttt{L}$, $\Gamma_\texttt{T}$ and $\Gamma_\texttt{RB}$ as on the left side of Figure~\ref{F:CantileverBracket-Domain}.

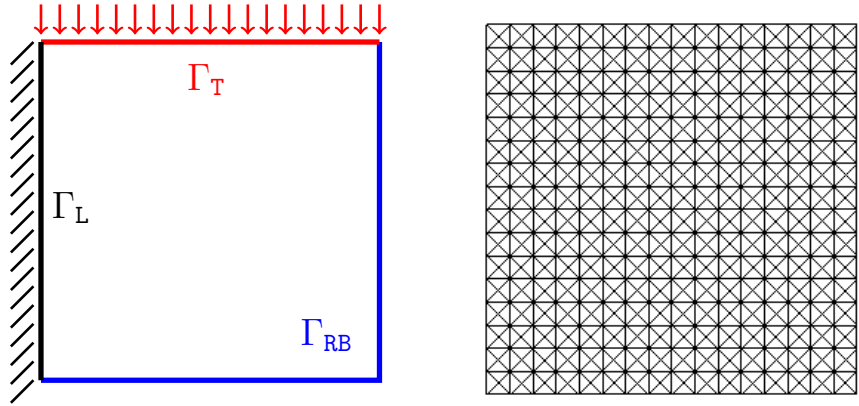
\begin{figure}[htp]
	\hfill
	\subfloat{
		\begin{tikzpicture}
		\draw[blue, line width=2pt] (0,0) -- (4.5,0) -- (4.5,4.5);
		\draw[red, line width=2pt] (0,4.5) -- (4.5,4.5);
		\draw[black, line width=2pt] (0,4.5) -- (0,0);
		\foreach \i in {0,...,18} {
			\draw[black, line width=1pt] (-0.1,{\i/4.0}) -- (-0.4,{\i/4.0 - 0.3});
			\draw[red, line width=1pt, ->] ({\i/4}, 5.0) -- ({\i/4.0}, 4.6);
		}
		\node[blue] at (3.8, 0.6) {\LARGE $\Gamma_\texttt{RB}$};
		\node[black] at (0.4, 2.3) {\LARGE$\Gamma_\texttt{L}$};
		\node[red] at (2.2, 4.0) {\LARGE $\Gamma_\texttt{T}$};
		\end{tikzpicture}
	}
	\hfill
	\subfloat{
		\includegraphics[width=0.4\hsize]{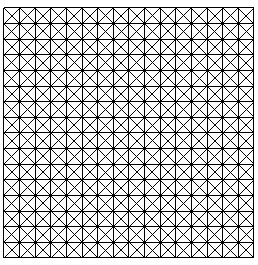}
	}
	\caption{Boundary conditions (left) and computational mesh (right) for the cantilever bracket problem.}
	\label{F:CantileverBracket-Domain}
\end{figure}

In analogy with \cite[Section~10.1]{Phillips:05} and \cite{Yi.13}, we set 
\begin{equation*}
T = 0.005, \qquad F = 1
\end{equation*}
as well as \footnote{More precisely \cite{Phillips:05,Yi.13} set the Young's modulus to $E = 10^4$ and the Poisson's ratio to $\nu = 0.4$.}
\begin{equation*}
\mu = 3571.4, \quad \lambda = 14286, \quad \alpha = 0.93, \quad \sigma = 0, \quad \kappa = 10^{-7}.
\end{equation*}

For the time discretization, we consider $J=5$ intervals of equal size. For the space discretization, we use the mesh on the right side of Figure~\ref{F:CantileverBracket-Domain} involving $5861$ degrees of freedom, whose number is computed as in \eqref{E:Terzaghi-DOFs-space}. Since the exact solution is unknown in this case, we confine ourselves to a qualitative investigation. More precisely, we plot the approximate pressure at the final time along four vertical lines at the following abscissas
\begin{equation*}
\label{E:Cantilever-bracket-abscissas}
x_1 = 0.26, \qquad x_2 = 0.33, \qquad x_3 = 0.4, \qquad x_4 = 0.45.
\end{equation*}
The plot in Figure~\ref{F:CantileverBracket-Plot} is qualitatively similar to the corresponding ones in \cite{Phillips:05,Yi.13} and, in particular, no pressure oscillations are observed. 

\begin{figure}[htp]
	\includegraphics[width=\hsize]{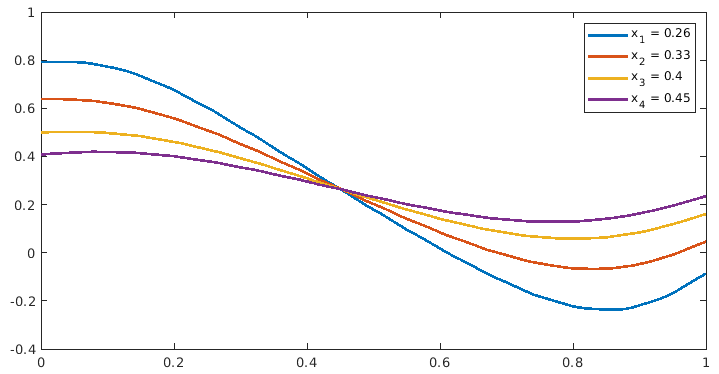}
	\caption{Pressure profile in the cantilever bracket problem along the vertical lines $x=x_j$, $j=1,\dots,4$.}
	\label{F:CantileverBracket-Plot}
\end{figure}

\subsection*{Funding}
Christian Kreuzer gratefully acknowledges
support by the DFG research grant KR 3984/5-2. Pietro Zanotti was supported by the PRIN 2022 PNRR project “Uncertainty Quantification of coupled models for water flow and contaminant transport” (No. P2022LXLYY), financed by the European Union—Next Generation EU and by the GNCS-INdAM project CUP E53C23001670001.

\begin{figure}[htp]
	\includegraphics[width=0.9\hsize]{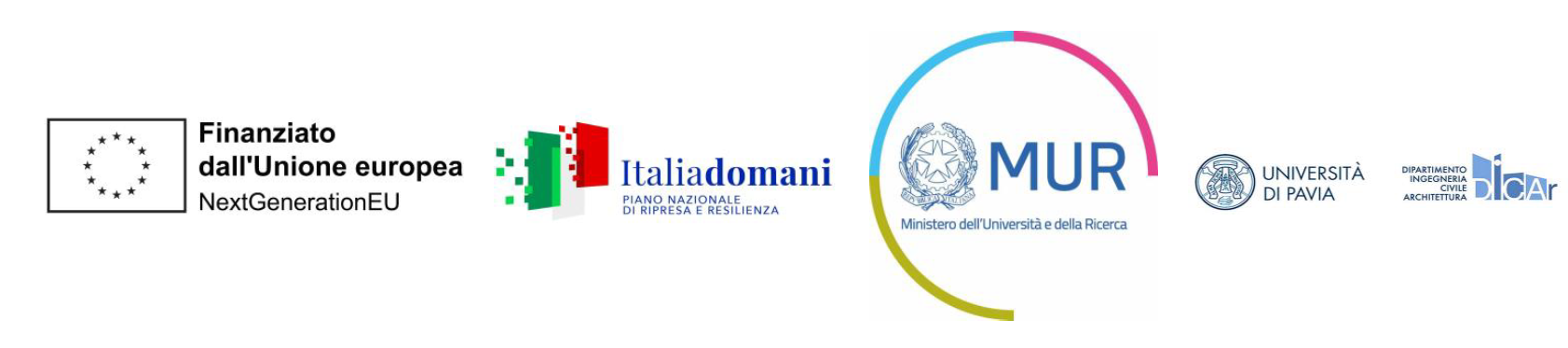}
\end{figure}

\appendix

\section{Proof of Lemma~\ref{L:smoothing}}
\label{S:proof-smoothing}
The construction of a linear operator $\calS: \Lagr{\sfk}{1} \to H^2(\Domain)$ satisfying all properties listed in Lemma~\ref{L:smoothing} is rather technical but it makes use only of classical tools from finite element analysis. To simplify the discussion as much as possible, we detail the construction only for the lowest degree and dimension, namely $\sfk = 2$ and $\Dim = 2$. This case minimizes the technical aspects but, at the same time, it illustrates all relevant issues. We address the extension to the other cases at the end of the appendix, in Remark~\ref{R:extensions}.

Our construction builds upon two preliminary steps. In the first one, we exhibit an operator mapping $\Lagr{\sfk}{1}$ to $H^2(\Domain)$, which satisfies the first, second and last conditions in Lemma~\ref{L:smoothing}. This can be done by some sort of averaging into a $H^2(\Domain)$-conforming finite element space. Similar operators exist in the literature but typically differ in the boundary conditions, see e.g. \cite{Brenner.Sung:18,Georgoulis.Houston.Virtanen:11}. In the second step we show that the third and fourth conditions in Lemma~\ref{L:smoothing} can be enforced by combinations of smooth bubble functions. Both techniques are common in the a posteriori analysis of (nonconforming) discretizations of fourth-order equations. 

Let us begin with the first step. For the sake of presentation,
we use the abbreviations
\begin{align*}
\Gamma_E=\overline\Gamma_{p,E}
,\quad \Faces_E=\Faces_{p,E}
\qquad\text{and}\qquad\Gamma_N=\overline\Gamma_{p,N}
,\quad \Faces_N=\Faces_{p,N}
\end{align*}
in what follows. The condition \eqref{E:smoothing-inclusion} can be rephrased as 
\begin{equation}
\label{E:smoothing-inclusion-equivalent}
\Lagr{2}{1}\ni Q = 0 \quad \text{on $\Gamma_E$}
\qquad \Longrightarrow \qquad
\calS Q = 0 \quad \text{on $\Gamma_E$} 
\end{equation}
cf. \eqref{E:abstract-spaces-pressure} and
\eqref{E:concrete-spaces-pressure}. The condition
\eqref{E:smoothing-inclusion}  actually requires also the inclusion in $L^2_0(\Domain)$ in some cases, but that can be obtained by enforcing \eqref{E:smoothing-elements}, so we do not care about it here. Hence \eqref{E:smoothing-inclusion-equivalent} and \eqref{E:smoothing-normal} are nothing else than Dirichlet and Neumann boundary conditions on $\Gamma_E$ and $\Gamma_N$, respectively. 

We construct a first operator $\calS_1: \Lagr{2}{1} \to H^2(\Domain)$ by mapping into the so-called HCT space, which is obtained by the finite element shown in Figure~\ref{F:HCT}, see also \cite[Section~6.1]{Ciarlet:78}. The degrees of freedom in this space are the normal derivative at the midpoint $m_\sfF$ of each edge $\sfF \in \Faces$ as well as the evaluation of the function and of its gradient at each vertex. For convenience, we denote by $\Vertices$ the set of all vertices.  

% Schematic representation of Lagrange, HCT and BFS elements
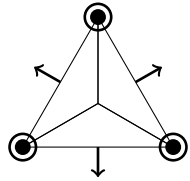
\begin{figure}[ht]
	\centering
	\begin{tikzpicture}
	%Vertices, edge midpoints and barycenter
	\coordinate (HCTVer0) at (3, 0);
	\coordinate (HCTVer1) at (5, 0);
	\coordinate (HCTVer2) at (4, 1.73);
	\coordinate (HCTMid0) at (4.5, 0.865);
	\coordinate (HCTMid1) at (3.5, 0.865);
	\coordinate (HCTMid2) at (4,0);
	\coordinate (HCTBar)  at (4, 0.58);
	%Draw the macro-triangle
	\draw (HCTVer0)--(HCTVer1) -- (HCTBar) -- (HCTVer0);
	\draw (HCTVer1)--(HCTVer2) -- (HCTBar) -- (HCTVer1);
	\draw (HCTVer2)--(HCTVer0) -- (HCTBar) -- (HCTVer2);
	%Represent evaluations of the function
	\fill (HCTVer0) circle (3pt);
	\fill (HCTVer1) circle (3pt);
	\fill (HCTVer2) circle (3pt);
	%Represent evaluations of the gradient
	\draw[line width=1pt] (HCTVer0) circle (5pt);
	\draw[line width=1pt] (HCTVer1) circle (5pt);
	\draw[line width=1pt] (HCTVer2) circle (5pt);
	%Represent evaluations of the normal derivative
	%(the length of each arrow is 0.4)
	\draw[->, line width=1pt] (HCTMid0) -- ++(0.35, 0.2);
	\draw[->, line width=1pt] (HCTMid1) -- ++(-0.35, 0.2);
	\draw[->, line width=1pt] (HCTMid2) -- ++(0, -0.4);
	\end{tikzpicture}
	\caption{Schematic representations of the HCT finite element}
	\label{F:HCT}
\end{figure}

\begin{subequations}
  \label{E:averaging}
  For the normal derivative at the edges, we set
\begin{equation}
\label{E:averaging-normal-derivative}
\Grad \calS_1 Q(m_\sfF) \cdot \Normal
=
\begin{cases}
0 & \text{if $\sfF \in \Faces_N$,}\\
\Grad Q_{|\sfT_\sfF}(m_\sfF)\cdot \Normal & \text{otherwise,} 
\end{cases}
\qquad \sfF \in \Faces
\end{equation}
where $\sfT_\sfF \in \Mesh$ is a (arbitrary but fixed) triangle in the mesh containing $\sfF$. For the point values, we just exploit the continuity of $Q$ and take
\begin{equation}
\label{E:averaging-point-values}
\calS_1 Q (\sfv)  = Q(\sfv),
\qquad \sfv \in \Vertices.
\end{equation}
We set the gradient at the interior vertices in analogy with the second case in \eqref{E:averaging-normal-derivative} to
\begin{equation}
\label{E:averaging-gradient-interior}
\Grad \calS_1 Q(\sfv) = \Grad Q_{|\sfT_\sfv}(\sfv), 
\qquad \sfv \in \Vertices \cap \Domain
\end{equation}
where $\sfT_\sfv \in \Mesh$ is a (arbitrary but fixed) triangle in the mesh containing $\sfv$. 
\end{subequations}

\begin{figure}[ht]
	\centering
	\subfloat[]{
		\begin{tikzpicture}
		\draw[line width = 2pt, color = red] (0,0)--(1,1.5);
		\draw[line width = 2pt, color = red] (1,1.5)--(2,3);
		\fill (1,1.5) circle (3pt);
		\node[red] at (0.2, 1) {$\Gamma_E$};
		\node[red] at (1.2, 2.5) {$\Gamma_E$};
		\node at (1.2, 1.3) {$\sfv$};
		\end{tikzpicture}
	}
	\hspace{60pt}
	\subfloat[]{
		\begin{tikzpicture}
		\draw[line width = 2pt, color = blue] (0,0)--(1,1.5);
		\draw[line width = 2pt, color = blue] (1,1.5)--(2,3);
		\fill (1,1.5) circle (3pt);
		\node[blue] at (0.2, 1) {$\Gamma_N$};
		\node[blue] at (1.2, 2.5) {$\Gamma_N$};
		\node at (1.2, 1.3) {$\sfv$};
		\end{tikzpicture}
	}
	\hspace{60pt}
	\subfloat[]{
		\begin{tikzpicture}
		\draw[line width = 2pt, color = red] (1,0)--(1,1.5);
		\draw[line width = 2pt, color = blue] (1,1.5)--(2.5,1.5);
		\fill (1,1.5) circle (3pt);
		\node[red] at (0.6, 0.7) {$\Gamma_E$};
		\node[blue] at (1.8, 1.9) {$\Gamma_N$};
		\node at (1.2, 1.2) {$\sfv$};
		\end{tikzpicture}
	}
	\\
	\subfloat[]{
		\begin{tikzpicture}
		\draw[line width = 2pt, color = red] (0,0)--(1,1.5);
		\draw[line width = 2pt, color = red] (1,1.5)--(2.5,1.5);
		\fill (1,1.5) circle (3pt);
		\node[red] at (0.2, 1) {$\Gamma_E$};
		\node[red] at (1.8, 1.9) {$\Gamma_E$};
		\node at (1.1, 1.2) {$\sfv$};
		\end{tikzpicture}
	}
	\hspace{40pt}
	\subfloat[]{
		\begin{tikzpicture}
		\draw[line width = 2pt, color = blue] (0,0)--(1,1.5);
		\draw[line width = 2pt, color = blue] (1,1.5)--(2.5,1.5);
		\fill (1,1.5) circle (3pt);
		\node[blue] at (0.2, 1) {$\Gamma_N$};
		\node[blue] at (1.8, 1.9) {$\Gamma_N$};
		\node at (1.1, 1.2) {$\sfv$};
		\end{tikzpicture}
	}
	\hspace{40pt}
	\subfloat[]{
		\begin{tikzpicture}
		\draw[line width = 2pt, color = red] (0,0)--(1,1.5);
		\draw[line width = 2pt, color = blue] (1,1.5)--(2.5,1.5);
		\fill (1,1.5) circle (3pt);
		\node[red] at (0.2, 1) {$\Gamma_E$};
		\node[blue] at (1.8, 1.9) {$\Gamma_N$};
		\node at (1.1, 1.2) {$\sfv$};
		\end{tikzpicture}
	}
	\caption{Different cases for the definition of the gradient of $\calS_1 Q$ at a boundary vertex $\sfv$. Note that the highlighted edges are aligned in (A)-(B) but not in (D)-(E) and that they meet at a right angle in (C) but not in (F)}
	\label{F:corners}
\end{figure}
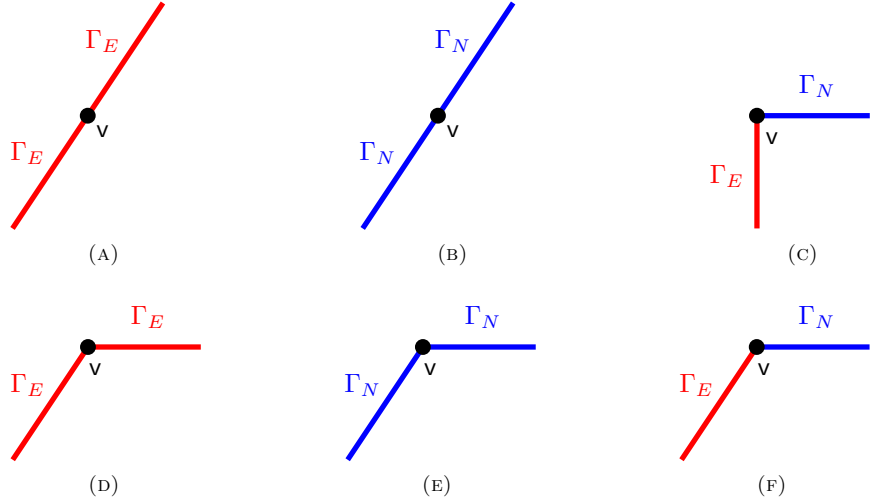 

The definition of the gradient at the other vertices is more involved. Indeed, the boundary conditions suggest to treat the normal and the tangential components differently. Of course, this is especially critical when $\sfv$ is a corner of $\Domain$ and/or if it lies at the intersection of $\Gamma_E$ and $\Gamma_N$. Figure~\ref{F:corners} summarizes all possible cases. We treat all cases simultaneously by Algorithm~\ref{Alg:gradient-boundary-vetices}. The underlying concept is that we first set the gradient to zero in the normal direction(s), so as to comply with \eqref{E:smoothing-normal}. Then, we treat the tangent direction(s) to enforce \eqref{E:smoothing-inclusion-equivalent}. Moreover, special attention is needed to make sure that the gradient is neither over- nor under-determined.

\begin{algorithm}[htp]
	\caption{Gradient of $\calS_1Q$ at boundary vertices}\label{Alg:gradient-boundary-vetices}
	\begin{algorithmic}[1]
		\algrenewcommand\algorithmicensure{\textbf{Provide}:}
		\Require $\sfv \in \Vertices \cap \partial \Domain$ boundary vertex
		\Ensure $\Grad \calS_1 Q(\sfv)$
		\State $\mathfrak{L} \gets \{0\}$
		\ForAll{$\sfF \in \Faces_N$ with $\sfv \in \sfF$}
		\State $\Grad \calS_1 Q(\sfv)\cdot \Normal_{\sfF} \gets 0$
		\State $\mathfrak{L} \gets \mathfrak{L} + \text{span}\{\Normal_{\sfF}\}$
		\EndFor
		\ForAll{$\sfF \in \Faces_E$ with $\sfv \in \sfF$}
		\If{$\Tangent_{\sfF} \notin \mathfrak{L}$}
		\State $\Grad \calS_1 Q(\sfv)\cdot \Tangent_{\sfF} \gets \Grad Q_{|\sfT_\sfF}(\sfv) \cdot \Tangent_{\sfF}$
		\State $\mathfrak{L} \gets \mathfrak{L} + \text{span}\{\Tangent_{\sfF}\}$
		\EndIf
		\EndFor
		\If{$\dim(\mathfrak{L}) = 1$}
		\State choose $\mathsf{w} \in \R^2$ with $\mathsf{w} \perp \mathfrak{L}$
		\State choose $\sfT \in \Mesh$ with $\sfv \in \sfT$
		\State $\Grad \calS_1 Q (\sfv) \cdot \mathsf{w} \gets \Grad Q_{|\sfT}(\sfv) \cdot \mathsf{w}$ 
		\EndIf
	\end{algorithmic}
\end{algorithm}

Within Algorithm~\ref{Alg:gradient-boundary-vetices}, we denote by $\Normal_\sfF$ and $\Tangent_{\sfF}$, respectively, the outward normal unit vector and a tangent unit vector of a boundary edge $\sfF \in \Faces \cap \partial \Domain$. We make use also of the unique triangle $\sfT_\sfF \in \Mesh$ which contains $\sfF$. The set $\mathfrak{L}$ is the linear space spanned by the directions used in the definition of the gradient, so it grows from $\{0\}$ to $\mathbb{R}^2$.

\begin{remark}[Simplified averaging]
\label{R:simplified-averaging}	
In \eqref{E:averaging-normal-derivative}, \eqref{E:averaging-gradient-interior} and in Algorithm~\ref{Alg:gradient-boundary-vetices} we define $\calS_1 Q$ in terms of the restriction of $Q$ to only one triangle in mesh. Therefore, we refer to $\calS_1$ as a \textit{simplified averaging}, so as to distinguish it from classical averaging operators, which take averages over all triangles containing the support of the degree of freedom under examination. Apart of the different definition, the simplified averagings reproduce all relevant properties of the classical ones, cf. \cite[Section~3]{Veeser.Zanotti:19b}.
\end{remark}

\begin{lemma}[Simplified averaging]
\label{L:Averaging}
The linear operator $\calS_1: \Lagr{2}{1} \to H^2(\Domain)$ defined by \eqref{E:averaging} and Algorithm~\ref{Alg:gradient-boundary-vetices} satisfies \eqref{E:smoothing-inclusion-equivalent}, \eqref{E:smoothing-normal} and \eqref{E:smoothing-stability}. In particular, the hidden constant in \eqref{E:smoothing-stability} depends only on the quantities mentioned in Remark ~\ref{R:hidden-constants}.
\end{lemma}

\begin{proof}
Let $\sfF \in \Faces \cap \partial \Domain$ be a boundary face. For $Q \in \Lagr{2}{1}$, the restriction of $\calS_1 Q$ to $\sfF$ is a third-order polynomial, owing to the definition of the HCT element. If $Q = 0$ on $\Gamma_E$ and $\sfF \in \Faces_E$, then both $\calS_1Q$ and its tangential derivative along $\sfF$ vanish at the endpoints of $\sfF$, thanks to \eqref{E:averaging-point-values} and Algorithm~\ref{Alg:gradient-boundary-vetices}. This proves that $\calS_1Q$ vanishes on $\sfF$ and verifies \eqref{E:smoothing-inclusion-equivalent}. Similarly, note that the normal derivative of $\calS_1Q$ on $\sfF$ is a second-order polynomial. If $\sfF \in \Faces_N$, then the normal derivative vanishes on $\sfF$, because it vanishes at the midpoint and at the endpoints of $\sfF$, in view of \eqref{E:averaging-normal-derivative} and Algorithm~\ref{Alg:gradient-boundary-vetices}. This verifies \eqref{E:smoothing-normal}.  
	
Regarding the claimed stability estimate, let $\sfT \in \Mesh$. The scaling properties of the HCT basis functions (see e.g. \cite[Lemma~3.14]{Veeser.Zanotti:19b}) imply
\begin{equation*}
\begin{split}
&\Norm{D^j(Q - \calS_1Q)}^2_\sfT
\eqsim
\sfh_{\sfT}^{2-2j} \sum_{\sfv \in \Vertices \cap \sfT}  | (Q_{|\sfT} - \calS_1 Q)(\sfv) |^2 \\
&\qquad +
\sfh_{\sfT}^{4-2j} \Big( \sum_{\sfv \in \Vertices \cap \sfT}    |\Grad (Q_{|\sfT} - \calS_1 Q)(\sfv) |^2 
+
\sum_{\sfF \in \Faces  \cap \sfT}| \Grad (Q_{|\sfT} - \calS_1 Q)(m_\sfF)\cdot \Normal  |^2\Big) 
\end{split}
\end{equation*}  	
for $j \in \{0,1,2\}$. The first summand on the right-hand side vanishes because of \eqref{E:averaging-point-values}, while the other ones are bounded in terms of jumps. Indeed, the definition \eqref{E:averaging-normal-derivative} immediately yields
\begin{equation*}
| \Grad (Q_{|\sfT} - \calS_1 Q)(m_\sfF)\cdot \Normal_{|\sfF}  |
\leq
\begin{cases}
0 & \text{if $\sfF \in \Faces_E$,}\\
|\Jump{\Grad Q}_{|\sfF} \cdot \Normal| & \text{otherwise,} 
\end{cases}
\end{equation*}
for all edges $\sfF \in \Faces \cap \sfT$.
Similarly, according to Algorithm~\ref{Alg:gradient-boundary-vetices}, we have
\begin{equation*}
|\Grad (Q_{|\sfT} - \calS_1 Q)(\sfv) |
\leq 
\sum_{\sfF \in \Faces^i \cup \Faces_{{N}}, \sfv \in \sfF  } | \Jump{\Grad Q}_{|\sfF}\cdot \Normal|.
\end{equation*}
for all vertices $\sfv \in \Vertices \cap \sfT$. The proof of this
estimate is tedious, as it must be verified for each case in
Figure~\ref{F:corners}, but it builds only upon a classical argument
for (simplified) averaging operators, see
e.g. \cite[Lemma~3.1]{Veeser.Zanotti:19b}. We insert the two bounds
above into the previous equivalence. Then, we obtain
\eqref{E:smoothing-stability} by an inverse estimate on the edges,
noticing that $\sfh_{\sfT} \eqsim \Avg{h}_{|\sfF}$ for all edges \(\sfF\)
touching $\sfT$. 
\end{proof}

For the second preliminary step in our construction, we make use of
bubble functions. In particular, we use `element' bubbles to enforce
\eqref{E:smoothing-elements}. For $\sfT \in \Mesh$, let $b_\sfT \in
\Poly{6}(\sfT)$ be the unique polynomial of degree $6$ in $\sfT$ such
that both $b_\sfT$ and $\Grad b_\sfT$ vanish on $\partial \sfT$ and
$\int_\sfT b_\sfT = 1$, cf. \cite[Section~3.2.5]{Verfuerth:13}. We
extend $b_\sfT$ to an $H^2(\Domain)$ function by zero. 

Similarly, we use `edge' bubbles to enforce \eqref{E:smoothing-faces}
on interior edges. Thus, for $\sfF \in \Faces^i$, we denote by
$\sfT_1$ and $\sfT_2$ the two triangles in the mesh containing
$\sfF$. Let $b_\sfF \in \Poly{8}(\sfT_1 \cup \sfT_2)$ be the unique
polynomial of degree $8$ on $\sfT_1 \cup \sfT_2$ such that both
$b_\sfF$ and $\Grad b_\sfF$ vanish on $\partial (\sfT_1 \cup \sfT_2)$
and $\int_\sfF b_\sfF = 1$, cf. \cite[Section~3.2.5]{Verfuerth:13}. As
before, we extend $b_\sfF$ to an $H^2(\Domain)$ function by zero.  

The condition \eqref{E:smoothing-faces} must be enforced also on the
edges in $\Faces_{{N}}$. Here we must employ bubble functions with
vanishing normal derivative on the respective edge, because of
\eqref{E:smoothing-normal}. Let $\sfT \in \Mesh$ be the unique
triangle in the mesh containing $\sfF$ and define $\sfT^\prime$ by
mirroring $\sfT$ through $\sfF$. We define $b_\sfF$ as before on $\sfT
\cup \sfT^\prime$. Hence, the symmetry of the support implies $\Grad
b_\sfF \cdot \Normal = 0$ on $\sfF$. Then, we extend $b_{\sfF|\sfT}$
to $H^2(\Domain)$ by zero. 

Having all these bubble functions at hand, we define $\calS_2: H^1(\Domain) \to H^2(\Domain)$ by
\begin{equation}
\label{E:bubble-smoother}
\calS_2 Q := \calS_\Faces Q + \calS_\Mesh (Q - \calS_\Faces Q)
\end{equation}
for $Q \in H^1(\Domain)$, where
\begin{equation*}
\label{E:bubble-smoother-faces-elements}
\calS_\Faces Q := \sum_{\sfF \in \Faces \cup \Faces_{{N}}}
\left( \int_\sfF Q \right) b_\sfF
\qquad \text{and} \qquad
\calS_\Mesh Q :=
\sum_{\sfT \in \Mesh} \left( \int_\sfT Q \right) b_\sfT.
\end{equation*}
The next lemma summarizes the properties of $\calS_2$ that are relevant to our purpose.

\begin{lemma}[Bubble operator]
	\label{L:bubble-correction}
	The operator $\calS_2: H^1(\Domain) \to H^2(\Domain)$ defined by \eqref{E:bubble-smoother} satisfies \eqref{E:smoothing-normal},\eqref{E:smoothing-elements} and \eqref{E:smoothing-faces}. Moreover $\calS_2 Q$ vanishes on $\Gamma_E$ and satisfies the estimate 
	\begin{equation*}
	\label{E:bubble-correction-stability}
	\Norm{\calS_2 Q}_\sfT 
	\lesssim
	\Norm{Q}_\sfT + \sfh_{\sfT} \Norm{\Grad Q}_\sfT
	\end{equation*}
	for all $Q \in H^1(\Domain)$ and $\sfT \in \Mesh$. The hidden constant depends only on the quantities mentioned in Remark ~\ref{R:hidden-constants}.
\end{lemma}

\begin{proof}
The properties \eqref{E:smoothing-normal},\eqref{E:smoothing-elements} and \eqref{E:smoothing-faces} and the identity $\calS_2 Q = 0$ on $\Gamma_E$ for $Q \in H^1(\Domain)$ hold true by construction. The local stability estimate follows by the scaling of the bubble functions. We refer to \cite[Lemma~3.8]{Veeser.Zanotti:19b} where a similar result is proved.
\end{proof}

We conclude the construction of the operator $\calS: \Lagr{2}{1} \to H^2(\Domain)$ announced in Lemma~\ref{L:smoothing} by combining the operators $\calS_1$ and $\calS_2$ introduced in the two steps above. More precisely, we set
\begin{equation*}
\label{E:smoothing-definition}
\calS Q := \calS_1 Q + \calS_2 (Q - \calS_1 Q)
\end{equation*}
for $Q \in \Lagr{2}{1}$. The properties of $\calS_1$ and $\calS_2$ listed in Lemmas~\ref{L:Averaging}-\ref{L:bubble-correction} imply that \eqref{E:smoothing-inclusion}-\eqref{E:smoothing-faces} hold true. Regarding the local estimate \eqref{E:smoothing-stability}, note that we have
\begin{equation*}
\Norm{D^j(Q - \calS Q)}_\sfT
\lesssim
\Norm{D^j(Q - \calS_1 Q)}_\sfT
+
\sfh_{\sfT}^{-j} \Norm{Q - \calS_1 Q}_\sfT.
\end{equation*}
for $j \in \{0,1,2\}$ and $\sfT \in \Mesh$, by Lemma~\ref{L:bubble-correction} and inverse estimates. We conclude by invoking the stability of $\calS_1$ established in Lemma~\ref{L:Averaging}.

\begin{remark}[Higher degree/dimension]
\label{R:extensions}
Our construction extends to higher polynomial degree $\sfk \geq 2$ and/or to higher space dimension $\Dim \geq 2$ up to some additional technicalities. The construction of the simplified averaging $\calS_1$ requires $H^2(\Domain)$-conforming spaces, that can be obtained by, e.g., higher-order HCT finite elements for $\Dim = 2$ (cf. \cite[Section~3]{Georgoulis.Houston.Virtanen:11} ) or the virtual elements in \cite{Brenner.Sung:18} for $\Dim = 2,3$. The operator $\calS_2$ employs the same bubble functions mentioned above, but one has to define it via the solution of local problems on the simplices and on the faces of the mesh, in the vein of \cite[Lemma~3.8]{Veeser.Zanotti:19b}. 
\end{remark}

\end{document}